\documentclass[a4,12pt,twoside,openright]{article}
\usepackage[a4paper, top=3cm, bottom=3cm, left=2cm]{geometry}
\usepackage{bbm, amssymb, amsmath, amsfonts, fancyhdr, german}

\usepackage{graphicx,epsfig}
\usepackage{color}

\usepackage[english]{babel}
\usepackage{array}
\usepackage{graphicx}
\usepackage{oldgerm} 
\usepackage[ansinew]{inputenc}				
\usepackage[T1]{fontenc}
\usepackage[labelfont=bf]{caption}

\usepackage{fancyhdr}
\newcommand{\N}{\mathbbm N}		\newcommand{\Z}{\mathbbm Z}
\newcommand{\R}{\mathbbm R}		\newcommand{\C}{\mathbbm C}

\newcommand{\W}{\mathcal W}		
\newcommand{\Y}{\mathcal Y}
\newcommand{\V}{\mathcal V}
\newcommand{\NN}{\mathcal{N}}
\newcommand{\FT}{\textfrak F}		
\newcommand{\s}{\mathcal S}			
\newcommand{\landau}{\mathcal O}

\newcommand{\modd}{\!\!\!\mod}

\def \dist {\text{\rm dist\,}}
\def \supp {\text{\rm supp\,}}
\def \graph {\text{\rm graph\,}}

\newcommand{\e}{\varepsilon}
\def\trans{{\,}^t}

\def\pa{{\partial}}

\def\Re{\text{\rm Re\,}}

\renewcommand{\labelenumi}{(\roman{enumi})}

\usepackage[thmmarks]{ntheorem}

\theoremstyle{plain}

\theoremheaderfont{\scshape}
\newtheorem{thmnr}{Theorem}[section]
\newtheorem{lemnr}[thmnr]{Lemma}
\newtheorem{kornr}[thmnr]{Corollary}

\newtheorem{remnr}[thmnr]{Remark}
\newtheorem{remnrs}[thmnr]{Remarks}

\newtheorem{conv}[thmnr]{Convention}

\theoremstyle{nonumberbreak} 

\theoremstyle{nonumberplain}

\theoremstyle{nonumberplain}
\theoremsymbol{$\Box$} 
\theorembodyfont{\upshape}
\theoremheaderfont{}
\newtheorem{proof}{\it Proof:}				

\numberwithin{equation}{section}				

\definecolor{green}{rgb}{0,0.6,0.2}

\begin{document}


\nocite{*}

\selectlanguage{english}

\thispagestyle{empty}

\section*{A FOURIER RESTRICTION THEOREM FOR A TWO-DIMENSIONAL SURFACE OF FINITE TYPE}

\vspace{0.7cm}

\begin{center}
\textbf{Stefan Buschenhenke, Detlef Müller and Ana Vargas}
\end{center}

\vspace{0.3cm}
\abstract{ The problem of $L^p(\R^3)\to L^2(S)$ Fourier restriction estimates for   smooth hypersurfaces $S$ of finite type in $\R^3$   is by now very well understood  for a large class of hypersurfaces,  including all analytic  ones. In this article, we  take up the study of  more general $L^p(\R^3)\to L^q(S)$ Fourier restriction estimates,  by studying a prototypical model  class  of two-dimensional surfaces   for which the Gaussian curvature degenerates in one-dimensional subsets.  We obtain sharp restriction theorems in the range given  by Tao in 2003  in his work on  paraboloids. For high order degeneracies this covers the full range, closing the restriction problem in  Lebesgue spaces for those surfaces.  A surprising new feature appears, in contrast with the non--vanishing curvature case: there is an extra necessary condition. Our approach  is based on an adaptation of the  bilinear method. A careful study of the dependence of the bilinear estimates on the curvature and size of the support is required.}

\vspace{1cm}

\renewcommand\contentsname{Contents}
\tableofcontents

\footnote{The authors  acknowledge the support by the Hausdorff Research Institute for Mathematics (HIM) in Bonn  and wish to express their gratitude to its staff for the wonderful hospitality . \\
The first author was partially supported by the ERC grant 277778.\\
The first two authors were partially supported by the DFG grant MU 761/ 11-1.\\
The third author was partially supported by MICINN/MINECO grants MTM2010-16518 and MTM2013--40945 (Spain).}

\newpage

\section{Introduction}

Let $S$ be a smooth hypersurface in $\R^n$ with surface measure $d\sigma_S.$ The  Fourier restriction problem for $S,$ proposed by  by E. M. Stein in the seventies, asks for the range  of exponents $p$ and $q$ for which the estimate
\begin{equation}
\bigg(\int_S|\widehat{f}|^p\,d\sigma_S\bigg)^{1/p}\le C\|f\|_{L^q(\R^n)}
\end{equation}
holds  true for every $f\in\mathcal S(\R^n),$ with a constant $C$ independent of $f.$ There had been a lot of activity on this problem in the seventies and early eighties.
The sharp range in dimension  $n=2$ for curves with non-vanishing curvature had been determined through work by  C. Fefferman, E. M. Stein and A. Zygmund \cite{F1}, \cite{Z}. In higher dimension, the sharp $L^q-L^2$ result  for hypersurfaces with non-vanishing Gaussian curvature was obtained by E. M. Stein and P. A. Tomas \cite{To}, \cite{St1} (see also Strichartz \cite{Str}). Some more general classes of surfaces were treated by A. Greenleaf \cite{Gr}.

The question about the general $L^q-L^p$ restriction estimates is nevertheless still wide open.  Fundamental progress has been made since the nineties, with contributions by many.  Major new ideas  were introduced in particular by J. Bourgain  (see for instance \cite{Bo1}, \cite{Bo2}) and T. Wolff (\cite{W1}),  which led to important further steps towards an understanding of the case of  non-vanishing Gaussian  curvature. These ideas and methods were further developed by A. Moyua, A. Vargas, L. Vega and T. Tao (\cite{MVV1}, \cite{MVV2} \cite{TVV}), who established  the so-called bilinear approach (which   had been anticipated in the work of C. Fefferman \cite{F1} and  had  implicitly been  present in the work of J. Bourgain \cite{Bo3}) for hypersurfaces with non-vanishing Gaussian curvature for which all principal curvatures  have the same sign. The  same method was applied to  the  light cone by Tao-Vargas (see \cite{TV1}, \cite{TV2}). The climax of the application of that 
bilinear  method to these types of  surfaces is due to T. Tao \cite{T1} (for principal  curvatures of the same sign), and T. Wolff \cite{W2}  and T. Tao \cite{T4} (for the light cone). In particular, in these last  two papers the sharp  linear restriction  estimates for the  light  cone in $\R^4$  were obtained.

 For the case of non-vanishing curvature but principal curvatures of different signs, analogous results in $\R^3$ were proved by S. Lee \cite{lee05} and A. Vargas \cite{v05}. Results for the light cone were previously obtained in $\R^3$ by B. Barcel\'o \cite{Ba1}, who also considered more general cones \cite{Ba2}. These results were improved to sharp theorems by S. Buschenhenke \cite{Bu}. The bilinear approach also produced results  for hypersurfaces with $k\le n-2$ non-vanishing  principal curvatures (\cite{LV}).

More recently,  J. Bourgain and L. Guth \cite{BoG} made further important  progress on the case of non-vanishing curvature by making use  also of  multilinear  restriction estimates due to  J. Bennett, A. Carbery and T. Tao \cite{BCT}.

On the other hand,  general finite type surfaces in $\R^3$ (without  assumptions on the curvature)  have been considered in work by I. Ikromov, M. Kempe and D. M\"uller  \cite{ikm}  and Ikromov and M\"uller \cite{IM-uniform}, \cite{IM},
 and the sharp range of Stein-Tomas type   $L^q-L^2$  restriction estimates has been  determined  for a large class  of smooth, finite-type hypersurfaces, including all analytic hypersurfaces.

 \medskip
 It is our aim in this work  to take up the latter branch of development by considering a certain model class  of hypersurfaces in dimension three with varying  curvature and study more general  $L^q-L^p$  restriction estimates. Our approach will again be based on the bilinear method.\footnote{The multilinear approach seems still not sufficiently developed  for this purpose, since estimates with sharp dependence on the transversality have not been proved.}
In our model class,  the degeneracy of the curvature will take  place along one-dimensional subvarieties. For analytic hypersurfaces whose Gaussian curvature does not  vanish identically, this kind of behavior is typical, even though in our model class the  zero varieties will still be linear (or the union of two linear subsets). Even though our model class would seem to be among the simplest possible surfaces of such behavior,  we will see that they require a very intricate study. We hope that this work will  give some insight  also for  future  research on  more  general types of hypersurfaces.

Independently of our work, a result for rotationally invariant surfaces with degeneracy of the curvature at a single point has been  obtained recently  by B. Stovall \cite{Sto}.

\subsection{Outline of the problem. The adjoint setting}

We start with a description of the surfaces that we want to study. We will consider surfaces that are graps of smooth functions defined on $Q=]0,1[\times]0,1[, $
$$\Gamma=\graph(\phi)=\{(\xi,\phi(\xi)):\xi\in Q\}. $$
The surface $\Gamma$ is equipped with the surface measure, $\sigma_\Gamma.$ It will be more convenient to use duality and work in the adjoint setting. The adjoint restriction operator is given by
\begin{align}\label{defop}
	\mathcal R^\ast f(x)=\widehat{f\,d\sigma_\Gamma}(x)= \int_\Gamma f(\xi)e^{-ix\cdot\xi}\,d\sigma_\Gamma(\xi),
\end{align}
where $f\in L^q(\Gamma,\sigma_\Gamma).$ The restriction problem is therefore equivalent to the question of finding the appropriate range of exponents for which the estimate
$$
\|\mathcal R^*f\|_{L^p(\R^3)}\le C\|f\|_{L^q(\Gamma,d\sigma_\Gamma)},
$$
holds with a constant $C$ independent of the function $f\in L^q(\Gamma,d\sigma_s).$ We shall require the  following properties of the functions $\phi:$

Let $m_1,m_2\in \R$, $m_1,m_2\geq2$. We say that a function $\phi$ is of \emph{normalized  type $(m_1,m_2)$} if there exists $\psi_1,\psi_2\in C^\infty(]0,1[)$ and $a,b>0$ such that

\begin{align}\label{model}
\phi(\xi_1,\xi_2)= \psi_1(\xi_1)+ \psi_2(\xi_2)
\end{align}
on $]0,1[\times ]0,1[,$ where the derivatives of the $\psi_i$ satisfy
\begin{align}\label{psicond1}
\psi_i''(\xi_i)&\sim \xi_i^{m_i-2},\\
|\psi_i^{(k)}(\xi_i)|&\lesssim \xi_i^{m_i-k} \quad \mbox{for }\ k\ge 3.\label{psicond2}
\end{align}
The constants hidden in these estimates are assumed to be admissible in the sense that they only depend on $m_1,m_2$ and the order of the derivative, but not explicitly on the $\psi_i$'s.

One would of course expect that small perturbations of such functions $\phi,$ depending on both $\xi_1$ and
$\xi_2,$ should lead to hypersurfaces  sharing the same restriction estimates as our model class  above.  However, such perturbation terms are not covered by our proof. It seems that the treatment of these more general situations would require even more intricate arguments, which will have to take  the underlying geometry of the surface into account. We plan to study these questions in the future.

\medskip

The prototypical example of  of a normalized   function of type $(m_1,m_2)$ is of course  $\phi(\xi)=\xi_1^{m_1}+\xi_2^{m_2}.$   For $m_1$ and $m_2$ integer, others arise simply as follows:
\begin{remnrs}
\begin{enumerate}
	\item Let $\e>0$ and $\varphi\in C^\infty(]-\e,\e[)$ be of \emph{finite type $m\in \N$} in $0$, i.e.,
	$\varphi(0)=\varphi'(0)=\cdots=\varphi^{(m-1)}(0)=0\neq\varphi^{(m)}(0)$. Assume $\varphi^{(m)}(0)>0$. Then there exist $\e'\in(0,\e)$	such that
	\begin{align*}
		\varphi^{(k)}(t)\sim  t^{m-k}
	\end{align*}
	for all $0\leq k\leq m$, $|t|<\e'$.
	\item Further let $\phi(\xi)=\psi_1(\xi_1)+\psi_2(\xi_2)$, $|\xi|\leq \e$, where $\psi_i\in C^\infty(]-\e,\e[)$ is of finite type $m_i$ in $0$ with $\varphi_i^{(m_i)}(0)>0$. Then there exist an $\bar\e>0$ such that $y\mapsto\phi(\bar\e y)$ is of normalised finite type $(m_1,m_2)$.
\end{enumerate}
\end{remnrs}

\begin{proof}

(i)  Since $\varphi$ has a zero of order $m$ at the origin, we find some $\e_0>0$, a smooth function
$\chi_0:]-\e_0,\e_0[\to ]0,\infty[$ and a sign $\sigma=\pm1$ such that
\begin{align*}
	\varphi(t)=\sigma t^m \chi_0(t)
\end{align*}
 for all $|t|<\e_0$. Thus
	$\varphi'(t)=\sigma t^{m-1}[m\chi_0(t)+t\chi_0'(t)]$. Since $\chi_0'$ is bounded on any compact neighborhood of the origin, and $\chi_0$ is bounded from below by a positive constant if the neighborhood is small enough, we have $t\chi_0'(t) \ll m\chi_0(t)$ for $t$ small enough, and thus we find $0<\e_1<\e_0$ such that
\begin{align*}
	\varphi'(t)=\sigma t^{m-1} \chi_1(t),
\end{align*}
	where $\chi_1(t)=m\chi_0(t)+t\chi_0'(t)>0$ for all $|t|\leq\e_1$. Iterating the procedure gives $\e_k>0$,
	$\chi_k:]-\e_k,\e_k[\to]0,\infty[$ such that
\begin{align*}
	\varphi^{(k)}(t)=\sigma t^{m-k} \chi_k(t)
\end{align*}
for every $|t|<\e_k$. Finally observe that $0<\varphi^{(m)}(0)=\sigma \chi_m(0)$, i.e., $\sigma=1$.
\medskip

\noindent	(ii) Chose $\bar\e>0$ such that for both $i=1,2$, $0\leq k\leq m_i$ and all $0\leq t\leq\bar\e$
	\begin{align*}
		\varphi^{(k)}_i(t) \sim  t^{(m_i-k)}.
	\end{align*}
	Then for all $0\leq s\leq 1$
	\begin{align*}
		\frac{d^k}{d s^k} \varphi_i(\bar\e s)
		 \sim  \bar\e^k (\bar\e s)^{(m_i-k)} = \bar\e^{m_i} s^{(m_i-k)}.
	\end{align*}

~\end{proof}

In order to  formulate our main theorem, adapting Varchenko's notion of height to our setting,  we introduce the height $h$ of the surface by
$$\frac{1}{h}=\frac{1}{m_1}+\frac{1}{m_2}.$$
 Let us also put put $\bar m=m_1\vee m_2=\max\{m_1,m_2\}$ and $m=m_1\wedge m_2=\min\{m_1,m_2\}$.


\begin{thmnr}\label{mainthm}
Let $p>\max\{\frac{10}{3},h+1\}$, $\frac{1}{s'}\geq\frac{h+1}{p}$ and $\frac{1}{s}+\frac{2\bar m+1}{p}<\frac{\bar m+2}{2}$.
Then $\mathcal R^*$ is bounded from $L^{s,t}(\Gamma,d\sigma_\Gamma)$ to $L^{p,t}(\R^3)$ for every $1\leq t\leq\infty$.

If in addition $s\leq p$ or $\frac{1}{s'}>\frac{h+1}{p},$ then $\mathcal R^*$ is even bounded from $L^{s}(\Gamma,d\sigma_\Gamma)$ to $L^{p}(\R^3)$.
\end{thmnr}

\begin{remnrs}\label{remsonnec}
\begin{enumerate}
\item  Notice that the ``critical line''  ${1}/{s'}={(h+1)}/{p}$ and the line ${1}/{s}+{(2\bar m+1)}/{p}={(\bar m+2)}/{2}$ in the $(1/s,1/p)$-plane  intersect at the point  $(1/s_0,1/p_0)$ given by
\begin{align}\label{p0}
\frac 1{s_0}=\frac{3\bar m +m-m\bar m}{4\bar m+2m},\quad \frac 1{p_0}=\frac{\bar m +m}{4\bar m+2m}.
\end{align}
This shows in particular that the point  $(1/s_0,1/p_0)$ lies strictly above (if $m>2$) or on the bi-sectrix $1/s=1/p$  (if $m=2$).
\item The condition $1/{s'}\geq{(h+1)}/{p}$ in the theorem is  necessary and in fact dictated by homogeneity (Knapp box examples).
\item By (i),  the condition
 \begin{align}\label{secondnec}
	\frac{\bar m+2}{2}>\frac{2\bar m+1}{p}+\frac{1}{s}
\end{align}
only plays a role  above the bi-sectrix.  It is  necessary too when $p<s,$  hence, in view of (i), if $m>2,$ and if $m=2,$  it is necessary  with the possible exception of the case where $s_0=p_0=4(\bar m+1)/(\bar m+2),$ for which we do not have an argument. Our proof in Section \ref{necessarycond}  will reflect the fact that for $m_j>2,$ the behavior of the operator must be worse than for the case $m_j=2.$
\item From the first condition in the theorem, we see that $p\ge h+1$ is also necessary. Moreover, we shall  show in Section \ref{necessarycond} that   strong type estimates  are not possible unless $s\leq p$ or ${1}/{s'}>{(h+1)}/{p}.$ The condition $p>{10}/3$ is due to the use of the bilinear method, as this exponent gives the sharp bilinear result for the paraboloid, and it is surely not sharp. Nevertheless, when $h>7/3,$ we obtain the sharp result.

\end{enumerate}

\end{remnrs}

\medskip

A new phenomenon appears in these surfaces. In the case of non-vanishing Gaussian curvature, it is  conjectured that the sharp range is  given by the homogeneity condition  $1/{s'}\geq{(h+1)}/{p}$ (with $h=2/(n-1),$ hence  $h+1=(n+1)/(n-1)$),  and a second condition, $p>{2n}/{(n-1)},$ due to the decay  rate    of the Fourier transform of the surface measure.  A similar result is conjectured for the light cone (cf. Figure \ref{gauss}).
\begin{figure}\begin{center}
  \includegraphics{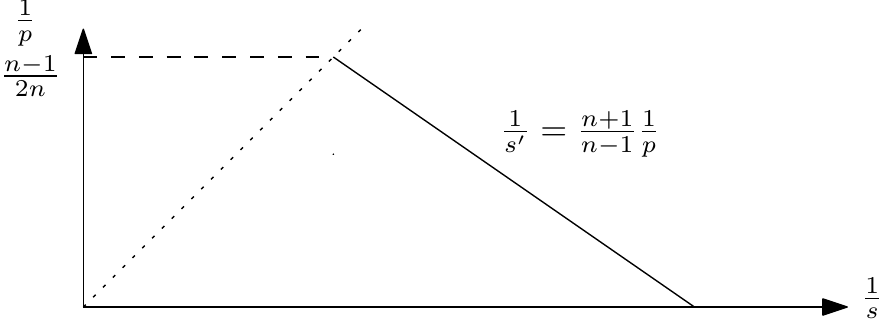}
\caption{\small Conjectured range of $p$ and $q$ for nonvanishing Gaussian curvature}
\label{gauss}
\end{center}\end{figure}
In contrast to this, we show in our theorem that for the class of surfaces $\Gamma$ under consideration  a third condition appears, namely  \eqref{secondnec}.

\medskip
Let us  briefly discuss the different situations that may arise in Theorem \ref{mainthm}, depending on the choice of  $m_1$ and $m_2:$
\medskip

First observe that $1/p_0$ in \eqref{p0} is above the critical threshold $1/{p_c}={3}/{10}$, if $\bar m \leq 2m$. In this case, the new condition ${1}/{s}+{(2\bar m+1)}/{p}={(\bar m+2)}/{2}$ will not show up in our theorem. So for $\bar m \leq 2m$, we are in the situation of either Figure \ref{sitII} (if $h\leq{7}/{3}$, i.e., $h+1\leq{10}/{3}$) or of Figure \ref{sitI} (if $h>{7}/{3}$). Notice that in the last case our theorem is sharp.\\
\begin{figure}[ht]\begin{center}
  \includegraphics{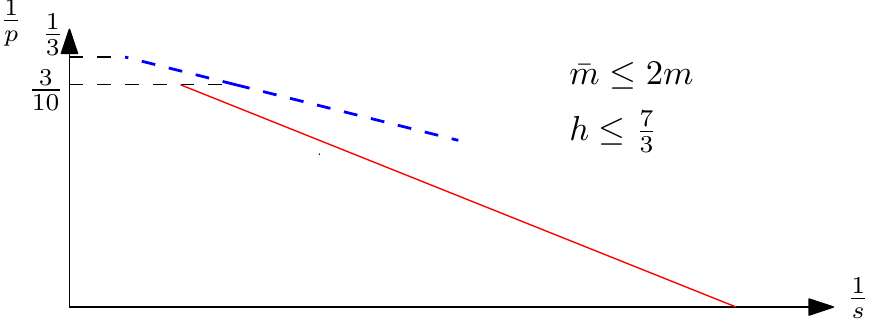}
\caption{\small Range of $p$ and $s$ in Theorem \ref{mainthm}}
\label{sitII}
\end{center}\end{figure}
\begin{figure}[ht]\begin{center}
  \includegraphics{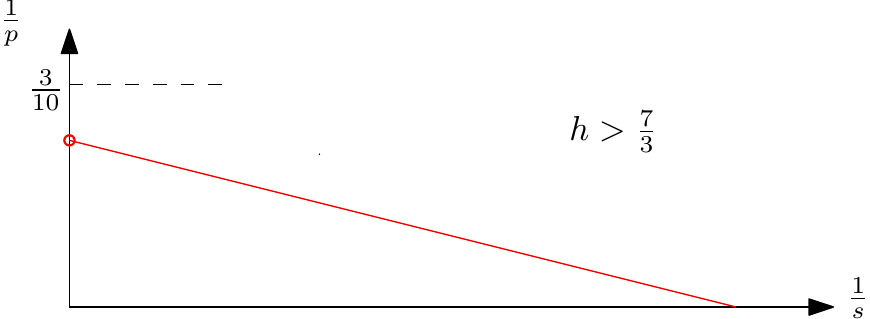}
\caption{\small Range of $p$ and $s$ in Theorem \ref{mainthm}}
\label{sitI}
\end{center}\end{figure}
It might also be interesting to compare $p_0$ not only with the condition $p>{10}/{3}$, which is due to the bilinear method, but with the conjectured range $p>3$. We always have $p_0\geq 3$, while we have $p_0=3$ only if $m_1=m_2$, i.e.,  a reasonable conjecture is that the new condition  \eqref{secondnec}  should always appear for inhomogeneous surfaces with $m_1\neq m_2$.
\medskip
In the case $\bar m>2m$, our new condition might be visible. Observe next that the line ${1}/{s}+{(2\bar m+1)}/{p}={(\bar m+2)}/{2}$ intersects the ${1}/{p}$-axis where $p=p_1={(4\bar m+2)}/{(\bar m+2)}$. Thus there are two subcases:

\medskip
 For $\bar m <7$ we have $p_1<{10}/{3}$, corresponding to Figure \ref{sitIII}, and our new condition appears.

\begin{figure}[ht]\begin{center}
  \includegraphics{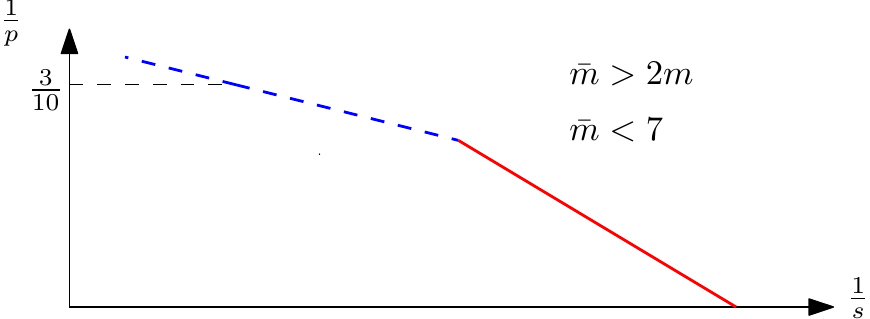}
\caption{\small Range of $p$ and $s$ in Theorem \ref{mainthm}}
\label{sitIII}
\end{center}\end{figure}

\medskip
For $\bar m \geq 7$ we may either have $p_0\geq p_1\geq h+1$ (which is equivalent to $\bar m m\geq 3\bar m+m$) and thus Figure \ref{sitIV} applies, or $p_0< p_1< h+1$ (which is equivalent to $\bar m m< 3\bar m+m$), and  we are in the situation of Figure \ref{sitI}; here again the new condition becomes relevant.
Observe that in the two last  mentioned cases, i.e.,  for $\bar m \geq 7$, our theorem is always sharp.


\begin{figure}[ht]\begin{center}
  \includegraphics{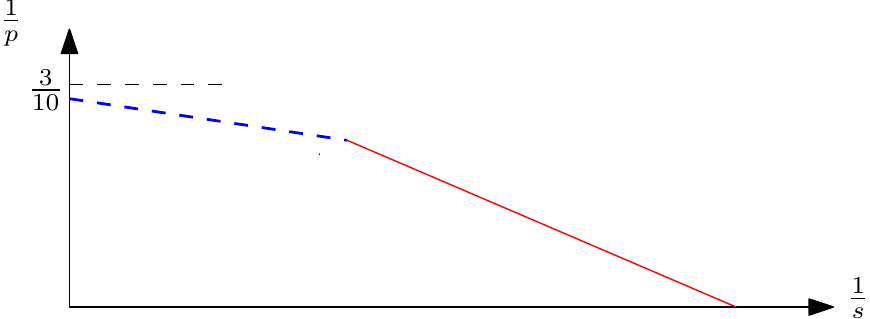}
\caption{\small Range of $p$ and $s$ in Theorem \ref{mainthm}}
\label{sitIV}
\end{center}\end{figure}
\medskip

Further observe that the appearance  of a third condition, besides the classical ones, is natural: Fix $m_1=2$ and let $\bar m=m_2 \to\infty.$ Then the contact order in the second coordinate direction degenerates. Hence, we would expect to find the same $p$-range as for a two dimensional cylinder, which  agrees with the  range for a parabola in the plane, namely $p>4$ (see \cite{F1}, \cite{Z}). Since $h\to2$ as $\bar m=m_2 \to\infty,$ the condition $p>\max\{{10}/{3},h+1\}$ becomes $p>{10}/{3}$ in the limit, which would lead to a larger range than expected. However, the new  extra condition $1/s+{(2\bar m+1)}/{p}<{(\bar m+2)}/{2}$ becomes $p>4$ for $\bar m\to\infty,$ as is to be expected.

\medskip

The restriction problem for the graph of functions $\phi(x)=\xi_1^{m_1}+\xi_2^{m_2}$ (and related surfaces) had  previously already been studied by E. Ferreyra and M. Urciuolo \cite{FU}, however by simpler methods, which led to weaker results than ours.  In their approach, they  made use of the  invariance of this surface under suitable non-isotropic dilations  as well as of the  one-dimensional results for curves.  This allowed them to obtain some results for
$p>4,$  in the  region below the homogeneity line, i.e., for ${1}/{s'}>(h+1)/{p}.$ Our results  are stronger in two ways:   they  include the critical  line and, more importantly, when $h<3,$ we  obtain a larger range for $p.$

As for the points on the critical  line  in the range $p>4,$ let us indicate that these points can in fact  also  be obtained by means of a simple summation argument   involving Lorentz spaces and real interpolation. This can be achieved by means of  summation trick   going back to ideas by Bourgain \cite{Bo85} (see for instance \cite{TVV}, \cite{lee03}). Details are given in Section \ref{summationtrick} of this article.

\subsection{Passage from surface to Lebesgue measure}

We shall always consider  hypersurfaces $S=\{(\eta,\phi(\eta)):\eta\in U\}$ which are  the graphs of  functions  $\phi$ that  are smooth on an open bounded subset $U\subset \R^d$ and continuous on the closure of $U.$  The adjoint of the Fourier restriction  operator associated to $S$ is  then given by
$$
\mathcal R^\ast f(x,t)=\widehat{f\,d\sigma_S}(x,t)= \int_S f(\xi)e^{-i(x,t)\cdot\xi}\,d\sigma_S(\xi), \qquad (x,t)\in\R^d\times \R=\R^{d+1},
$$
where $d\sigma_S=(1+|\nabla\phi(\eta)|^2)^{1/2} \, d\eta$ denotes the Riemannian surface measure of $S.$ Here, $f:S\to \C$ is a function on $S,$ but we shall often identify it with the corresponding function $\tilde f:U\to\C,$ given by $\tilde f(\eta)=f(\eta,\phi(\eta)).$
Correspondingly,  we define
$$
R^*_{\R^d}g(x,t):= \widehat{g d\nu}(x,t):=\int_{U} g(\eta) e^{-i(x\eta+t\phi(\eta))} \, d\eta, \qquad  (x,t)\in\R^d\times \R=\R^{d+1},
$$
for every function $g\in L^1(U)$ on $U.$ We shall occasionally address  $d\nu=d\nu_S$  as the ``Lebesgue measure'' on $S,$  in contrast with  the surface measure $d\sigma=d\sigma_S.$ Moreover, to emphasize which surface $S$  is meant, we shall occasionally also write $R^*_{\R^d}=R^*_{S,\R^d}.$  Observe that if  there is a constant $A$ such that
\begin{align}\label{nablabound}
|\nabla \phi(\eta)|\le A, \qquad \eta\in U,
\end{align}
 (this applies for instance to our class of hypersurfaces $\Gamma,$ since we assume that $m_i\ge 2$), then the Lebesgue measure
 $d\nu$    and the surface measure $d\sigma$  are comparable, up to  some positive multiplicative constants depending only on $A.$ Moreover,  since
\begin{align}
R^\ast f =R^*_{\R^d} \big(\tilde f (1+|\nabla\phi(\eta)|^2)^{1/2}\big),
\end{align}
the $L^q$-norms $\|\tilde f\|_{L^q(d\eta)}$ and  $\|f\|_{L^q(d\sigma_S)}=\|\tilde f(1+|\nabla\phi(\eta)|^2)^{1/2q}\|_{L^q(d\eta)}$  of $\tilde f$ and of $f$ are comparable too.
Throughout  the article, we shall therefore apply the following
\begin{conv}\label{conv}
{\rm
Whenever $|\nabla \phi|\lesssim 1,$ then, with some slight abuse of notation, we shall denote the function $f$ on $S$ and the corresponding function $\tilde f$ on $U$ by the same symbol $f,$  and write $R_{\R^d}^* f$ in place of $R_{\R^d}^* \tilde f.$
 }
\end{conv}

Applying this in particular to the hypersurfaces $\Gamma,$  we see that  Theorem \ref{mainthm} is equivalent to the following
\begin{thmnr}\label{mainthm2}
Let $p>\max\{\frac{10}{3},h+1\}$, $\frac{1}{s'}\geq\frac{h+1}{p}$ and $\frac{1}{s}+\frac{2\bar m+1}{p}<\frac{\bar m+2}{2}$.
Then $R_{\R^2}^*$ is bounded from $L^{s,t}(\Gamma,d\sigma_\Gamma)$ to $L^{p,t}(\R^3)$ for every $1\leq t\leq\infty$.

If in addition $s\leq p$ or $\frac{1}{s'}>\frac{h+1}{p},$ then $R_{\R^2}^*$ is even bounded from $L^{s}(\Gamma,d\sigma_\Gamma)$ to $L^{p}(\R^3)$.
\end{thmnr}

\subsection{Necessary conditions}\label{necessarycond}

The condition $p>h+1$ is in some sense the weakest one.  Indeed, the second condition already implies $p\geq h+1,$ and even  $p>h+1$ when $s<\infty.$ Thus the condition $p>h+1$ only plays a role when the critical  line
${1}/{s'}= {(h+1)}/{p}$ intersects the axis ${1}/{s}=0$ at a point where $p> p_c=10/3$ (cf. Figure \ref{sitI}).

However, the condition $p>h+1$ is necessary as well (although some kind of weak type estimate might hold true at the endpoint).
This can be shown by analyzing the oscillatory integral defined by $R^*_{\R^d} 1$ (compare {\cite{So87} for similar arguments). For the sake of simplicity, we shall do this only for the model case $\phi(\xi)=\xi_1^{m_1}+\xi_2^{m_2}$ (the more general case can be treated by
similar, but technically more involved arguments).

Indeed,  we have the following
 well-known lower bounds on  related oscillatory integrals (for the convenience of the reader, we include a proof):
\begin{lemnr}\label{oscint} Assume that $m\geq2.$		
\begin{enumerate}
 \item If  $1\ll\mu\ll\lambda\ll\mu^m$, then $\left| \int_0^\delta e^{i(\mu\xi-\lambda(\xi^m+\landau(\xi^{m+1}))}{d}\xi \right|
								\ge C_\delta \mu^{-\frac{m-2}{2m-2}}\lambda^{-\frac{1}{2m-2}}$,
							provided $\delta>0$ is sufficiently small.

\item If  $1\ll\mu^m\ll\lambda,$   $0\le \alpha <1$ and  $0\le \beta <1$ , then
	$$\left| \int_0^1
		 e^{i(\mu\xi-\lambda\xi^m)}\xi^{-\alpha}|\log(\xi/2)|^{-\beta} d\xi \right| \gtrsim \lambda^\frac{\alpha-1}{m} (\log\lambda)^{-\beta}.$$
\end{enumerate}
\end{lemnr}

\begin{proof} (i)  Apply the transformation $\xi\mapsto\left({\mu}/{\lambda}\right)^\frac{1}{m-1}\xi$ to obtain
\begin{align*}
	 \int_0^\delta e^{i(\mu\xi-\lambda\xi^m+\landau(\xi^{m+1}))}{d}\xi
	&=
						\int_0^{\delta\left(\frac{\lambda}{\mu}\right)^\frac{1}{m-1}}
	\left(\frac{\mu}{\lambda}\right)^\frac{1}{m-1} e^{i\left(\frac{\mu^m}{\lambda}\right)^\frac{1}{m-1}\big(\xi-\xi^m+\landau((\frac \mu \lambda)^{\frac 1{m-1}}\xi^{m+1})\big)}{d}\xi \\
		&=\int_0^1  +\int_1^{\delta\left(\frac{\lambda}{\mu}\right)^\frac{1}{m-1}},
\end{align*}
where $\left({\mu}/{\lambda}\right)^\frac{1}{m-1}\ll 1$ and  $\left({\mu^m}/{\lambda}\right)^\frac{1}{m-1}\gg1$.
The phase function $\phi(\xi)=\xi-\xi^m+\landau((\frac \mu \lambda)^{\frac 1{m-1}}\xi^{m+1})$ has a unique critical point  at $\xi_0=\xi_0(\mu/\lambda)$ in $[0,1]$ lying very close to $m^{-1/(m-1)}$ , so we may  apply the method of stationary phase to the first integral and find that
$$
\left| \int_0^1  \right|
	\gtrsim  \left(\frac{\mu}{\lambda}\right)^\frac{1}{m-1}
					 \left(\frac{\mu^m}{\lambda}\right)^{-\frac{1}{2}\frac{1}{m-1}}.
$$
Moreover, integration by parts in  the second  integral leads to
\begin{align*}
	\left|\int_1^{\delta\left(\frac{\lambda}{\mu}\right)^\frac{1}{m-1}} \right|
	\lesssim  C_1\left(\frac{\mu}{\lambda}\right)^\frac{1}{m-1}\left(\frac{\mu^m}{\lambda}\right)^{-\frac{1}{m-1}},
\end{align*}
provided $\delta$ is sufficiently small. These estimates imply that
$$
 \left|	 \int_0^1  \left(\frac{\mu}{\lambda}\right)^\frac{1}{m-1}e^{i(\mu\xi-\lambda\xi^m+\landau(\xi^{m+1}))}{d}\xi \right|	\gtrsim \left(\frac{\mu}{\lambda}\right)^\frac{1}{m-1}	  \left(\frac{\mu^m}{\lambda}\right)^{-\frac{1}{2}\frac{1}{m-1}}
	= \mu^{-\frac{m-2}{2m-2}}\lambda^{-\frac{1}{2m-2}}.					 				
$$
\medskip
\noindent (ii)
Apply the change of variables $\xi\mapsto\lambda^{-\frac{1}{m}}\xi$ to obtain
\begin{align*}
	&\left| \int_0^1 e^{i(\mu\xi-\lambda\xi^m)}\xi^{-\alpha}|\log(\xi/2)|^{-\beta}{d}\xi \right|\\
		=& \lambda^{\frac{\alpha-1}{m}}|\log(\lambda^{-\frac 1m})|^{-\beta} \left| \int_0^{\lambda^\frac{1}{m}}
			 e^{i(\mu\lambda^{-\frac{1}{m}}\xi-\xi^m)} \xi^{-\alpha} \Big(1+\frac{|\log(\xi/2)|}{\frac 1m \log\lambda}\Big)^{-\beta}\,  {d}\xi \right|	\\
	\gtrsim& \lambda^{\frac{\alpha-1}{m}}(\log \lambda)^{-\beta}.							
\end{align*}
Notice here that $\lambda^{1/m}\gg 1$ and $\mu\lambda^{-\frac{1}{m}}\ll 1,$ and that, as $\lambda\to\infty,$  the last oscillatory integral tends to $\int_0^\infty e^{-i\xi^m} \xi^{-\alpha} {d}\xi\ne 0$ (which is easily seen).
\end{proof}

Part (ii) of the lemma implies that
\begin{align*}
	\left| \int_0^1 e^{-i(x_1\xi-x_3\xi^{m_1})}d\xi\right| \gtrsim x_3^{-\frac{1}{m_1}}
\end{align*}
for $1\leq x_3<\infty$, $1\ll x_1^{m_1}\ll x_3,$
 and since  $R^*_{\R^d} 1=\widehat {d\nu},$  we find that
\begin{align*}
	\|R^*_{\R^d} 1\|_p^p
	\geq&\int_1^\infty \int_{1\ll x_2\ll x_3^\frac{1}{m_2}} \int_{1\ll x_1\ll x_3^\frac{1}{m_1}}
									|R^*_{\R^d}1 (x_1,x_2,-x_3)|^p d x_1d x_2d x_3	\\
	\gtrsim&	\int_1^\infty \int_{1\ll x_2\ll x_3^\frac{1}{m_2}}d x_2
		\int_{1\ll x_1\ll x_3^\frac{1}{m_1}}d x_1 x_3^{-p\left(\frac{1}{m_1}+\frac{1}{m_2}\right)} d x_3\\ 		 \gtrsim& \int_1^\infty x_3^{(1-p)\left(\frac{1}{m_1}+\frac{1}{m_2}\right)} d x_3	\\
		=& \int_1^\infty x_3^{-\frac{p-1}{h}} d x_3.
\end{align*}
If the adjoint Fourier restriction operator is bounded, the integral has to be finite, thus necessarily
${(p-1)}/{h}>1$, i.e., $p>h+1.$

\medskip

Next, to see that the condition \eqref{secondnec}, i.e.,
$	{(\bar m+2)}/{2}>{(2\bar m+1)}/{p}+{1}/{s},
$
is necessary in Theorem $\ref{mainthm}$, we consider the subsurface
$$\Gamma_0=\{(\xi,\phi(\xi)):\xi\in[0,1]\times[1/2,1/2+\delta]\},
$$
where $\delta>0$ is assumed to be sufficiently small. On this subsurface, the principal curvature in  the $\xi_2$-direction is bounded from below. This means that, after applying  a suitable affine transformation of coordinates, the restriction problem for the surface $\Gamma_0$ is equivalent to the  one for  the surface
 $$\Gamma_{m_1,2}=\{( \xi_1,\xi_2,\xi_1^{m_1}+c(\xi_2^2+\landau(\xi_2^3))):(\xi_1,\xi_2)\in[0,1]\times[0,\delta]\},
 $$
 where $c>0.$
 \medskip

As stated  in Remark \ref{remsonnec}, the condition \eqref{secondnec} only plays a role above the bi-sectrix $1/s=1/p$.  So, assume that $p<s$ (as explained, this excludes only  the case where $m=2$ and $s=p=p_0$). Then we may choose
$\beta<1$ such that $\beta s>1>\beta p.$
Assume that $\mathcal R^\ast_{\Gamma}$ is bounded from $L^{s}(\Gamma)$ to $L^p(\R^3)$, i.e.,  $\mathcal R^\ast_{\Gamma_{m_1,2}}$ is bounded from $L^{s}(\Gamma_{m_1,2})$ to $L^p(\R^3)$. Passing again from the surface measure $d\sigma$ to the ``Lebesgue measure'' $d\nu$ on $\Gamma_{m_1,2},$  define $f(\xi_1,\xi_2)=\xi_1^{-{1}/{s}}\log(\xi_1/2)^{-\beta}\in L^{s}(\Gamma_{m_1,2}, d\nu ).$ Then
\begin{align*}
	|\widehat{f{d}\nu}(x_1,x_2,-t)|
	=& \left|\int_{[0,1]\times[0,\delta]} e^{i(x_1\xi_1+x_2\xi_2-t(\xi_1^{m_1}+c(\xi_2^2+\landau(\xi_2^3)))} \xi_1^{-\frac{1}{s}}\log(\xi_1/2)^{-\beta} {d}(\xi_1,\xi_2)\right|	\\
	=& \left|\int_0^1 e^{i(x_1\xi_1-t \xi_1^{m_1})} \xi_1^{-\frac{1}{s}}\log(\xi_1/2)^{-\beta} {d}\xi_1\right|	
			\  \left|\int_0^\delta e^{i(x_2\xi_2-tc(\xi_2^2+\landau(\xi_2^3))} {d}\xi_2\right|.			
\end{align*}
We estimate the first integral by means of  Lemma \ref{oscint} (ii),  and  for the second one we use  Lemma \ref{oscint} (i) (with $m=2$), which leads to
\begin{align*}
	\infty 	>  \|f\|_{s}^p \gtrsim \|\widehat{f{d}\nu}\|_p^p	
		\gtrsim& \int_N^\infty 	\int_{t^{1/2}}^t \int_1^{t^{\frac 1{m_1}}} t^{-\frac{p}{s'm_1}}  t^{-\frac{p}{2}} (\log t)^{-\beta p}\, dx_1 dx_2 dt \\		\approx& \int_N^\infty   t^{1-\frac{p}{2}}  t^{\frac{1}{m_1}-\frac{p}{s'm_1}}(\log t)^{-\beta p} \, dt,									
\end{align*}
provided $N$ is chosen sufficiently large.
This implies that  necessarily
\begin{equation*}
	 1-\frac{p}{2}+\frac{1}{m_1}-\frac{p}{s'm_1}<-1,
\end{equation*}
which is equivalent to
\begin{equation*}
	\frac{m_1+2}{2}>\frac{2m_1+1}{p}+\frac{1}{s}.
\end{equation*}
Interchanging the roles of $\xi_1$ and $\xi_2$, we obtain the same inequality for $m_2$ and hence for $\bar m=m_1\vee m_2$, and we arrive at  \eqref{secondnec}.


\medskip

Let us finally prove  that on the critical line $1/s'=(h+1)/p$ one cannot have strong-type estimates above the bi-sectrix $1/s=1/p$, i.e., for $s>p$. In this regime, we find  some $1>r>0$ such that $1/s<r<1/p$. Let
\begin{align*}
	f(\xi) = \xi_2^{-m_2/sh}| \log(\xi_2/2)|^{-r} \chi_{\{\xi_1^{m_1}\leq \xi_2^{m_2}\}}(\xi).
\end{align*}
It is easy to check that $f\in L^s(\Gamma)$ since $1<rs$.
Now assume that $1\ll x_j^{m_j}\ll t$ for $j=1,2,$  more precisely choose $N\gg1$ and assume that  $N^2\leq Nx_j^{m_j}\leq t$ for $j=1,2$. Then
\begin{align*}
	(R^*_{\R^d}f)(x_1,x_2,-t) &=
	\int_0^1  e^{-i(x_2\xi_2-t\xi_2^{m_2})} \xi_2^{-m_2/sh} |\log(\xi_2/2)|^{-r} \\
	&\qquad\times \int_0^{\xi_2^{m_2/m_1}}	 e^{-i(x_1\xi_1-t\xi_1^{m_1})}d\xi_1 d\xi_2.
\end{align*}
Since $x_1^{m_1}\ll t$ is equivalent to $(\xi_2^{m_2/m_1}x_1)^{m_1}\ll t\xi_2^{m_2}$,  Lemma \ref{oscint} (ii) gives
\begin{align*}
	\left|\int_0^{\xi_2^{m_2/m_1}}	 e^{-i(x_1\xi_1-t\xi_1^{m_1})}d\xi_1 \right|
	= \xi_2^\frac{m_2}{m_1} \left|\int_0^{1}	e^{-i(\xi_2^{m_2/m_1}x_1\eta-t\xi_2^{m_2}\eta^{m_1})}d\eta \right|
	\gtrsim t^{-\frac{1}{m_1}}.
\end{align*}
Applying Lemma \ref{oscint} once more, we obtain
\begin{align*}
	\left|(R^*_{\R^d}f)(x_1,x_2,-t)\right|
	\gtrsim t^{-\frac{1}{m_1}} t^{\frac{1}{sh}-\frac{1}{m_2}} \log^{-r}(t/2)
	= t^{-\frac{1}{p}\left(1+\frac{1}{h}\right)} \log^{-r}(t/2),
\end{align*}
where we made use of $1/s'=(h+1)/p$. Thus we get
\begin{align*}
	\|R^*_{\R^d}f\|_p^p \gtrsim&
	\int_{N^2}^\infty 	\int_{N^\frac{1}{m_2}}^{(t/N)^\frac{1}{m_2}} \int_{N^\frac{1}{m_1}}^{(t/N)^\frac{1}{m_1}} t^{-1-\frac{1}{h}} \log^{-rp}(t/2) \, dx_1 dx_2 dt	\\
	\approx& \int_{N^2}^\infty 	t^{-1} \log^{-rp}(t/2) \,  dt =\infty,
\end{align*}
since $rp<1$.

\bigskip

Let us finish this subsection by adding a few  more observations and remarks.

(a) First, observe that $\Gamma_0$ is a subset of
$$
	\Gamma_1=\{(\xi,\phi(\xi))\in \Gamma:|\xi|\sim 1\}.
$$

(b)  One can use the  dilations $(\xi_1,\xi_2)\mapsto (r^{1/m_1}\xi_1,r^{1/m_2}\xi_2), \, r>0,$ in order to decompose $Q=[0,1]\times [0,1]$ into ``dyadic annuli'' which, after rescaling, reduces the  restriction problem in many situations  to  the one  for  $\Gamma_1$ (this kind of approach has been used intensively  in \cite{ikm}, \cite{IM-uniform},  as well as in   \cite{FU}).

Indeed, on the one hand, any restriction estimate on $\Gamma$ clearly implies the same  estimate also for the subsurface
$\Gamma_1.$  On the other hand, the estimates for the  dyadic pieces sum  up below the sharp critical  line (this has been the approach in  \cite{FU}), i.e.,  when  $1/{s'}>{(h+1)}/{p}.$ Moreover, in many situations one may apply Bourgain's summation trick  in a similar way as described in Section \ref{summationtrick} in order to establish weak-type estimates also when $(1/s,1/p)$ lies on the critical line, i.e., when $1/{s'}={(h+1)}/{p}.$ However, we shall not pursue this approach here, since it would not give too much of a simplification for us and since our approach (outlined in the next subsection)  seems to lead to even  somewhat sharper  result. Moreover, it seems useful and more systematic to understand   bilinear restriction estimates for quite general pairs of pieces of our surfaces $\Gamma,$ and not only the ones which would arise from $\Gamma_1.$

 \medskip
(c) On $\Gamma_1$, one of the  two principal curvatures may vanish, but not both. Notice also that by  dividing  $\Gamma_1$ into  a finite number  of pieces lying in sufficiently small angular sectors  and applying a suitable affine transformation to each of them, we may reduce to surfaces of the form
$$
\Gamma_{m,2}=\{( \xi_1,\xi_2,\psi_m(\xi_1)+\xi_2^2+\landau(\xi_2^3)):\xi_1,\xi_2\in[0,1]\},
$$
 where $\psi_m(\xi_1)\sim \xi_1^m$  as before, with $m=m_1$ or $m=m_2$  (compare also our previous discussion of necessary conditions).  Applying then a further  dyadic decomposition in $\xi_1,$  we  see that we may essentially  reduce to subsurfaces on which $\xi_1\sim\e,$ with $\e>0$  a small dyadic number.
 Note that on these  we have non-vanishing Gaussian curvature,  but the lower bounds of the curvature  depend on $\e>0$. A rescaling then leads to surfaces of the form
 $$
 P_T=\{(\xi_1,\xi_2,\xi_1^2+\xi_2^2+\landau(\xi_1^3+T^{-1}\xi_2^3)): \xi_1\in[0,1],\xi_2\in[0,T]\},
 $$
 with  $T=\e^{-\frac{m}{2}}\gg 1.$
 A prototype of such a situation would be the part of the standard  paraboloid  lying above a very long-stretched  rectangle.
  Although Fourier restriction estimates for the paraboloid  have been studied extensively, the authors are not aware of any results that would give  the right control on the dependence on the parameter $T\gg 1$. Indeed, one can prove  that  the following  lower bound for the adjoint restriction operator $R_T^*=R^*_{P_T,\R^2}$
associated to Lebesgue measure on  $P_T$  holds true for all $s$ and $p$ for which $(1/s,1/p)$ lies  within  the shaded region in Figure \ref{conjecture}\begin{align}\label{pconjecture}
   \|R^*_{T}\|_{L^s(P_T)\to L^p(\R^3)} \gtrsim T^{\left(\frac{1}{p}-\frac{1}{s}\right)_+},
   \end{align}
and a reasonable conjecture is that also the reverse inequality essentially holds true, maybe up to an extra factor $T^\delta$, i.e., that
\begin{align}\label{pconjecture2}
   \|R^*_{T}\|_{L^s(P_T)\to L^p(\R^3)} \le C_\delta T^{\delta+\left(\frac{1}{p}-\frac{1}{s}\right)_+},
\end{align}
for every $\delta>0.$

\begin{figure}[ht]
\centering
  \includegraphics[width=0.5\textwidth,height=0.18\textwidth]{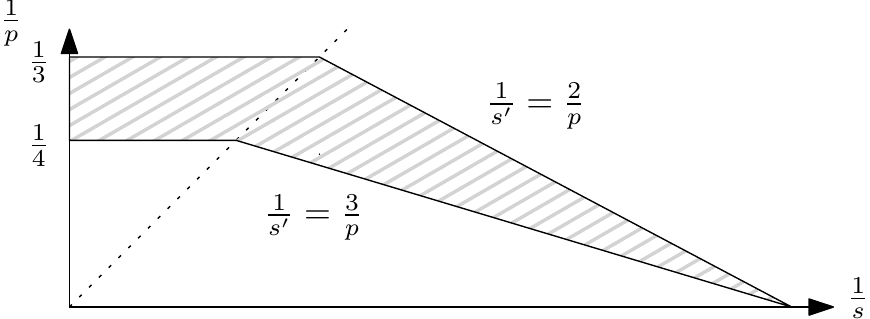}
\caption{Region  on which \eqref{pconjecture} is valid}
\label{conjecture}       
\end{figure}



We give some hints why \eqref{pconjecture} holds true and why the inverse inequality (with $\delta$-loss) seems a  reasonable conjecture.
Let $d\nu_T$ denote the ``Lebesgue measure'' on  $P_T$. Then by Lemma \ref{oscint}
\begin{align*}
	|\widehat{d\nu_T}(x_1,x_2,t)|\gtrsim& t^{-\frac{1}{2}}
	\int_0^T e^{i(x_2\xi_2+t[\xi_2^2+\landau(T^{-1}\xi_2^3)])} d\xi_2 \\
	=& T t^{-\frac{1}{2}}
	\int_0^1 e^{i(x_2T\eta+tT^2[\eta^2+\landau(\eta^3)])} d\eta
	\gtrsim t^{-1},
\end{align*}
provided $x_1\ll t$ and $x_2\ll Tt$ (we may arrange matters in the preceding reductions so that the error term  $\landau(\eta^3)$ is small compared to $\eta^2$).
Hence, since we assume that $p>3,$
$$
   \|\widehat{d\nu_T}\|_{L^p(\R^3)} \gtrsim T^\frac{1}{p}.
$$
Obviously $\|1\|_{L^s(P_T, d\nu_T)}=T^\frac{1}{s}$, so we see that
\begin{align*}
	\|R^*_{T}\|_{L^s(P_T)\to L^p(\R^3)} \gtrsim T^{\left(\frac{1}{p}-\frac{1}{s}\right)}.
\end{align*}
Restricting $P_T$ to  the region where $\xi_2\leq1$, we see that   also
$	\|R^*_{T}\|_{L^s(P_T)\to L^p(\R^3)} \gtrsim 1,$
and combining these two  lower bounds gives \eqref{pconjecture}.
\medskip

On the other hand,  from Remark 4, (2.4) in \cite{FU} we easily obtain by an obvious rescaling-argument  that for $ 1/{s'}=3/p$ and  $p>4$ (hence $1/p<1/s$), we have
$$ \|R^*_{T}\|_{L^s(P_T,d\nu_T)\to L^p(\R^3)}\le C,$$
uniformly in $T.$
It is conjectured that for the  entire  paraboloid  $\mathcal P=\{(\xi_1,\xi_2,\xi_1^2+\xi_2^2): (\xi_1,\xi_2)\in\R^2\},$ the adjoint restriction operator $R^*_{\mathcal P,\R^2}$ is bounded for $1/{s'}=2/p$ and $p>3$ (hence $1/p<1/s$). It would be reasonable to expect the same kind of behaviour for suitable perturbations of the paraboloid, and subsets of those, such as $P_T$ (maybe with an extra factor $T^\delta,$ for any $\delta>0$). By complex interpolation, the previous estimate in combination with the latter  conjectural estimate   would  lead to
$$ \|R^*_T\|_{L^s(P_T, d\nu_T)\to L^p(\R^3)}\le C_\delta T^\delta,$$
for every $\delta>0,$ provided that  $1/p<1/s$ and $2/p<1/{s'}<3/p.$ In combination with a trivial application of
H\" older's inequality this leads to the conjecture \eqref{pconjecture2}
$$
   \|R^*_{T}\|_{L^s(P_T)\to L^p(\R^3)} \le C_\delta T^{\delta+\left(\frac{1}{p}-\frac{1}{s}\right)_+},
$$
for every $\delta>0,$ provided $(1/s,1/p)$ lies  within  the shaded region in Figure \ref{conjecture}.

\subsection{The strategy of the approach}
We will study certain bilinear operators. For a suitable pair of subsurfaces $S_1,S_2\subset S$ (we will be more specific on this point later), we seek to establish  bilinear estimates
\begin{align*}
	\|R_{\R^2}^* f_1 R_{\R^2}^*f_2\|_{L^p(\R^3)} \leq
	C_p C(S_1,S_2) \, \|f_1\|_{L^2(S_1)} \|f_2\|_{L^2(S_2)},
\end{align*}
for functions $f_1,\, f_2$ supported in $S_1$ and $S_2,$ respectively.

For hypersurfaces with nonvanishing Gaussian curvature and principal  curvatures of the same sign, the sharp estimates of this type, under the appropriate transversality assumption,  appeared in \cite{T2} (after previous partial results in \cite{TVV}, \cite{TV1}). For the light cone in any dimension, the analogous results were established in \cite{W2}, \cite{T4} (improving on earlier  results  in \cite{Bo3} and \cite{TV1}). For the case of principal curvatures of different sign, or with a smaller number of non-vanishing principal curvatures, sharp bilinear  results are also known \cite{lee05}, \cite{v05}, \cite{LV}.
\medskip

What is crucial for  us is to know how the constant $C(S_1,S_2)$  explicitly depends on the pair of surfaces  $S_1$ and $S_2,$ in order to be   able to sum all the bilinear estimates  that we obtain for pairs of pieces of our given surface, to pass to a linear estimate. Classically, this is done by proving a bilinear estimate for one "generic" class  of subsurfaces. For instance,  if $S$ is the paraboloid,  then other pairs of subsurfaces can be reduced to it  by  means suitable affine transformations and homogeneous rescalings. However,  general surfaces do not come with such kind of  self-similarity under these  transformations, and it  is one of the new features of this article that we will establish  very precise bilinear estimates.
\medskip

The bounds on the constant $C(S_1,S_2)$  that we establish will  depend on the size of the domains and local principal curvatures of the subsurfaces, and we shall have to keep track of these  during the whole proof. In this sense, many of the lemmas are generalized, quantitative versions of well known results from classical bilinear theory.
\medskip

The pairs of subsurfaces that we would like to discuss are pieces of the surface sitting over two dyadic rectangles and  satisfying certain separation or "transversality" assumptions. However, such a rectangle might touch one of the axis, where some principle curvature is vanishing. In this case we will decompose dyadically a second time.
But even on these smaller sets, we do not have the correct "transversality" conditions; we first have to find a proper rescaling such that the scaled subsurfaces allow to run the bilinear machinery.
\medskip

The following section will begin with the bilinear argument to provide us with a very general bilinear result for  sufficiently "good" pairs of surfaces. In the subsequent section, we construct a suitable scaling in order  to apply this general result to our situation. After rescaling and several additional arguments, we pass to a global bilinear estimate and finally proceed to the linear estimate.

\medskip

A few more remarks on the notion will be useful:
as mentioned before, it is very important to know precisely how the constants depend on the specific choice of subsurfaces. Moreover, there will appear other constants,  depending possibly  on $m_1,m_2,p,q$, or other quantities, but not explicitly on the choice of subsurfaces. We will not keep track of such  types of constants, since it would even set a false focus and distract the reader. Instead we will simply use the symbol $\lesssim$ for an inequality involving one of these constants of minor importance. To be more precise on this, later we introduce a family of pairs of subsurfaces $\s_0$. Then for quantities $A,B:\s_0\to\R$ the inequality $A\lesssim B$ means there exists a constant $C>0$ such that
$A(S_1,S_2)\leq C B(S_1,S_2)$ uniformly for all $(S_1,S_2)\in\s_0$.
\medskip

Moreover, we will also use the notation $A\sim  B$ if $A\lesssim B$ and $B\lesssim A$. We will even use this notation for vectors, meaning their entries are comparable in each coordinate.
Similarly, we write $A\ll B$ if there exists a constant $c>0$ such that
$A(S_1,S_2)\leq c B(S_1,S_2)$ for all $(S_1,S_2)\in\s_0$ and $c$ is "small enough" for our purposes. This  notion of  being "sufficiently small" will in general  depend on the situation and further constants, but the choice will  be uniform in the sense that it will  work for all pairs of subsurfaces in the class $\s_0$.

The inner product of two vectors $x,y$  will usually be denoted by $xy$ or $x\cdot y,$ occasionally also by $\langle x,y\rangle.$ 

\section{General bilinear theory}

\subsection{Wave packet decomposition}
We begin with what is basically a well known result, although we need a more quantitative version (cf. \cite{T3}, \cite{lee05}).

\begin{lemnr}\label{wave packets}
Let $U\subset\R^d$ be an open and bounded subset, and let $\phi\in C^\infty(U,\R).$
We assume that there exist  constants $\kappa>0$ and  $D\leq {1}/{\kappa}$ such that $\|\partial^\alpha\phi\|_\infty \le A_\alpha\kappa D^{2-|\alpha|}$ for all $\alpha\in\N^d$ with $|\alpha|\geq2$.
Then for every $R\geq1$ there exists a wave packet decomposition adapted to $\phi$ with tubes of radius
$R/D=R'$ and length $R^2/(D^2\kappa)=R'^2/\kappa$, where we have put  $R=R'D.$

More precisely,  consider the   index sets  $\Y=R'\Z^d$ and  $\V=R'^{-1}\Z^d \cap U$,  and  define for $w=(y,v)\in \Y\times \V= \W$ the tube
\begin{align}\label{wavepack}
	T_w= \Big\{(x,t)\in\R^d\times\R: |t|\leq \frac{R'^2}{\kappa},~|x-y+t\nabla\phi(v)|\leq R'\Big\}.
\end{align}
Then, given any function $f\in L^2(U),$  there exist functions (wave packets) $\{p_w\}_{w\in \W}$ and coefficients $c_w\in\C$ such that  $R^*_{\R^d}f$ can be decomposed as
$$
R^*_{\R^d}f(x,t)=\sum\limits_{w\in\W} c_w p_w(x,t)
$$
for every $t\in\R$ with  $|t|\leq{R'^2}/{\kappa},$ in such a way that  the following hold true:
\begin{enumerate}
\renewcommand{\labelenumi}{(P\arabic{enumi})} 
	\item $p_w=R^*_{\R^d}(\FT_{\R^d}^{-1} (p_w(\cdot,0))).$
	\item $\supp \FT_{\R^{d+1}} p_w \subset B((v,\phi(v)),2/R').$
	\item $p_w$ is essentially supported in $T_w,$  i.e.,
					$|p_w(x,t)|\leq C_N R'^{-\frac{d}{2}} \left(1+\frac{|x-y+t\nabla\phi(v)|}{R'}\right)^{-N}$
					for every $N\in\N.$
				In particular,  $\|p_w(\cdot,t)\|_2 \lesssim 1$.
	\item For all $W\subset \W$, we have
					$ \|\sum\limits_{w\in W} p_w(\cdot,t)\|_2\lesssim |W|^\frac{1}{2}$.
	\item $\|c\|_{\ell^2} \lesssim \|f\|_{L^2}$.
\end{enumerate}
\renewcommand{\labelenumi}{(\roman{enumi})} 
Moreover, the constants arising explicitly (such as the $C_N$) or implicitly in these estimates  can be chosen to depend only on the constants  $A_\alpha$ but no further on the function $\phi,$  and also not on the other quantities $R,D$ and $\kappa$  (such constants will be called admissible).
\end{lemnr}


\begin{figure}[ht]
\centering
  \includegraphics[width=0.5\textwidth,height=0.4\textwidth]{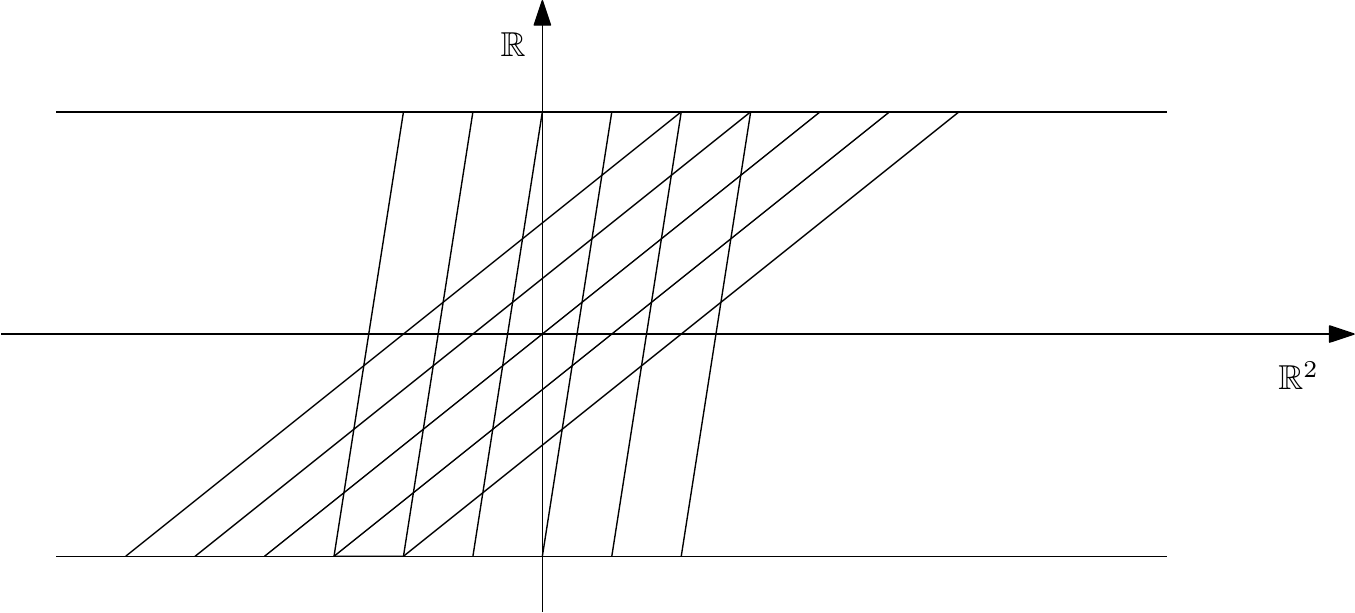}
\caption{The tubes $T_w$ fill a horizontal strip}
\label{horizonstrip}       
\end{figure}

\begin{remnrs} \label{reformulationrem}  (i) Notice that no bound is required on $\nabla\phi$ at this stage; however, such  bounds will become important later (for instance in  (iii)).

(ii) Denote by $N(v)$ the normal vector at $(v,\phi(v))$ to the graph of $\phi$ which is given by  $N(v)=(-\nabla\phi(v),1)$. Since $R'^2/\kappa\ge R',$ we may  thus re-write
$$T_w=(y,0)+\Big\{tN(v):|t|\leq\frac{R'^2}{\kappa}\Big\}+\landau(R').
$$
Moreover,
$$|x-y+t\nabla\phi(v)|=|(x,t)-(y,0)-tN(v)|\geq \dist((x,t),T_w).$$ It is then easily seen that  (P3) can be re-written as
$$
|p_w(z)|\leq C_N R'^{-\frac{d}{2}} \left(1+\frac{\dist(z,T_w)}{R'}\right)^{-N}
$$
for all $z\in\R^{d+1}$ with $|\langle z, e_{d+1}\rangle|\leq R'^2/\kappa,$ where $e_{d+1}$ denotes  the last vector of the canonical basis of $\R^{d+1}.$  This justifies the statement that ``$p_w$ is essentially supported in $T_w$.''

(iii) Notice further that we can re-parametrize the wave packets by lifting $\V$ to $\tilde\V=\{(v,\phi(v)):v\in\V\}\subset S$. If we now assume that $\|\nabla\phi\|\lesssim 1$, then we have $|(v,\phi(v))-(v',\phi(v'))|\sim  |v-v'|$, and  thus $\tilde\V$ becomes an $R'^{-1}$-net in $S$.
Finally, we shall identify a parameter $y\in\R^d$ with the point $(y,0)$ in the hyperplane $\R^d\times\{0\}$.
\end{remnrs}

\begin{proof} We will basically follow the proof by Lee \cite{lee05}; the only new feature consists in elaborating the precise  role of the constant $\kappa$.\\
 Let $\psi,\hat\eta\in C_0^\infty(B(0,1))$ be chosen in a  such a way that for
$\eta_y(x)=\eta(\frac{x-y}{R'})$, $\psi_v(\xi)=\psi(R'(\xi-v))$
we have $\sum\limits_{v\in\V}\psi_v=1$ on $U$ and $\sum\limits_{y\in\Y} \eta_y=1$.
We also chose a slightly bigger function  $\tilde\psi\in C_0^\infty(B(0,3))$ such that $\tilde\psi=1$ on $B(0,2)\supset\supp\psi+\supp\hat\eta,$ and put $\tilde\psi_v(\xi)=\tilde\psi(R'(\xi-v))$.
Then the functions
$$
F_{(y,v)}=\FT_{\R^d}^{-1}(\widehat{\psi_vf}\eta_y)=(\psi_v f)\ast \check\eta_y, \quad y\in\Y, v\in\V,
$$
 are  essentially well-localized in both position and momentum/frequency  space. Define
$q_w=R_{\R^d}^*(F_w)$, $w=(y,v)\in\W$; up to a certain factor $c_w,$ which will be determined later, these are already the announced wave packets, i.e., $q_w=c_w p_w$.\\
Since $f=\sum\limits_{w\in W} F_w $, we then have the decomposition  $R_{\R^d}^*f=\sum\limits_{w\in W} q_w.$
Let us concentrate on property (P3) - the other properties are then rather easy to establish.    Since $\supp F_{(y,v)}\subset B(v,2/R')$, we have for every $w=(y,v)\in\W$
\begin{align*}
	q_w(x,t) =& \int e^{-i(x\xi+t\phi(\xi))} F_w(\xi) \,d\xi
	= \int e^{-i(x\xi+t\phi(\xi))} F_w(\xi)\tilde\psi_v(\xi)\,d\xi \\
	=&(2\pi)^{-d} \iint e^{-i\big((x-z)\xi+t\phi(\xi)\big)}\tilde\psi_v(\xi)\, d\xi
						\, \widehat {F_w}(z) \,dz	\\
	=& (2\pi)^{-d} R'^{-d} \iint e^{-i\big((x-z)(\frac{\xi}{R'}+v)+t\phi(\frac{\xi}{R'}+v)\big)}
	\tilde\psi(\xi) \,d\xi	\,  \widehat {F_w}(z)\, dz	\\
	=& (2\pi)^{-d} R'^{-d}\int K(x-z,t) \widehat {F_w}(z)\, dz,					
\end{align*}
with the kernel
\begin{align} \nonumber
	K(x,t)=\int e^{i\big(x(\frac{\xi}{R'}+v)+t\phi(\frac{\xi}{R'}+v)\big)}\tilde\psi(\xi)\, d\xi.
\end{align}
We claim that
\begin{align}\label{oszint2}
	|K(x,t)| \lesssim \left(1+\frac{|x+t\nabla\phi(v)|}{R'}\right)^{-N}
\end{align}
for every $N\in\N$. To this end, we shall estimate the oscillatory integral
\begin{align}\nonumber
	K_\lambda = \int e^{i\lambda\Phi(\xi)}\tilde\psi(\xi)\,d\xi,
\end{align}
with phase
\begin{align}\nonumber
	\Phi(\xi) = \frac{x(\frac{\xi}{R'}+v)+t\phi(\frac{\xi}{R'}+v)}{1+R'^{-1}|x+t\nabla\phi(v)|},
\end{align}
where we have put $\lambda=1+R'^{-1}|x+t\nabla\phi(v)|$.
In order to prove \eqref{oszint2}, we may assume that  $|x+t\nabla\phi(v)|\gg R'$. Then integrations by parts will lead to  $|K_\lambda|\lesssim \lambda^{-N}$ for all $N\in\N,$ hence  to \eqref{oszint2}, provided we can show that
\begin{align}
	|\nabla\Phi(\xi)|\sim & 1 ,\qquad\text{for all } \xi, \label{gradphi} \\
		\text{and}\quad \|\partial^\alpha\Phi\|_\infty \lesssim& 1 \quad\text{for all }\alpha\geq 2,
		\label{alphaphi}
\end{align}
and that the constants in these estimates are admissible. But,
\begin{align*}
	\frac{|t||\nabla\phi(\frac{\xi}{R'}+v)-\nabla\phi(v)|}{|x+t\nabla\phi(v)|}	
	\ll& \frac{|t||\nabla\phi(\frac{\xi}{R'}+v)-\nabla\phi(v)|}{R'} 	
	\lesssim \frac{|t|}{R'^2} \|\phi''\|_\infty \\
	\leq& \frac{1}{\kappa} \|\phi''\|_\infty \leq 1,
\end{align*}
for every $\xi\in\supp\tilde\psi$, hence
\begin{align}
	|t||\nabla\phi(\frac{\xi}{R'}+v)-\nabla\phi(v)| \ll |x+t\nabla\phi(v)|.\nonumber
\end{align}
Thus
\begin{align*}
	|\nabla\Phi(\xi)| = \frac{|x+t\nabla\phi(\frac{\xi}{R'}+v)|}{R'+|x+t\nabla\phi(v)|}	
	=& \frac{|x+t\nabla\phi(v)-t[\nabla\phi(v)-\nabla\phi(\frac{\xi}{R'}+v)]|}{R'+|x+t\nabla\phi(v)|} \\
	\sim & \frac{|x+t\nabla\phi(v)|}{R'+|x+t\nabla\phi(v)|} \sim  1,
\end{align*}
which verifies  \eqref{gradphi}.
And,
for $|\alpha|\geq2$ we have
\begin{align*}
	|\partial^\alpha\Phi(\xi)| \leq& |t R'^{-|\alpha|}(\partial^\alpha\phi)(\frac{\xi}{R'}+v)|	
\overset{}{\lesssim} \frac{R'^2}{\kappa} R'^{-|\alpha|} \kappa D^{2-|\alpha|}		
	\leq (D R')^{2-|\alpha|} = R^{2-|\alpha|} \leq 1,
\end{align*}
which gives \eqref{alphaphi}. It is easily checked that the constants in these  estimates can be chosen to be admissible.
Following the proof in \cite{lee05}, we conclude that
\begin{align*}
	|q_w(x,t)| \lesssim&  R'^{-d}\int\left| K(x-z-y,t)	\widehat {F_w}(z+y) \right|d z		\\
	=& R'^{-d}\int\left|  K(x-z-y,t)\eta\left(\frac{z}{R'}\right)\widehat{\psi_v f}(z+y) \right|d z	\\
	\lesssim& \left(1+\frac{|x-y+t\nabla\phi(v)|}{R'}\right)^{-N} \mathop{M} (\widehat{\psi_v f})(y),
\end{align*}
where $M$ denotes the Hardy-Littlewood maximal operator. Thus, by  choosing  $c_w=c_{(y,v)}={R'}^\frac{d}{2} M(\widehat{\psi_vf})(y),$ we obtain (P3).

Properties (P1) and (P2) follow from the definition of the wave packets.
From (P2) and (P3) we can deduce (P4). For (P5), we refer to \cite{lee05}.
\end{proof}

In view of  our previous remarks, it is easy to re-state Lemma \ref{wave packets} in a more coordinate-free way. For any given hyperplane $H=n^\perp\subset\R^{d+1}$, with $n$  a unit vector (so that $\R^{d+1}=H+\R n$), define the partial Fourier (co)-transform
$$
\FT_H^{-1} f(\xi+tn)= \int_H f(x+tn)e^{ix\cdot\xi} \,dx, \quad \xi\in H, t\in \R.
$$
Moreover, if $U\subset H$ is open and bounded, and if $\phi_H\in C^\infty(U,\R)$ is given,  then  consider the smooth hypersurface  $S=\{\eta+\phi_H(\eta) n: \eta\in U\}\subset \R^{d+1},$ and define the corresponding Fourier extension operator
$$
R^*_{H}f(x+tn)= \int_{U} f(\eta) e^{-i(x\eta+t\phi_H(\eta))} \, d\eta= \int_{U} f(\eta) e^{-i<x+tn,\eta+\phi_H(\eta)n>} \, d\eta,
$$
for $(x,t)\in H\times \R, f\in L^2(U).$ Notice that $R^*_{\R^d}$ corresponds to the special case $H=\R^d\times \{0\},$ and  thus by means of a suitable rotation, mapping $e_{d+1}$ to $n,$ we immediately obtain the  following


\begin{kornr}[Wave packet decomposition]\label{wave packets2}
 Let $U\subset H$ be an open and bounded subset, and let $\phi_H\in C^\infty(U,\R).$  We assume  that there are constants $\kappa>0$ and $D\leq{1}/{\kappa}$ such that $\|\phi_H^{(l)}\|_\infty \le A_l\kappa D^{2-l}$ for every $l\in\N$ with $l\ge 2,$ where $\phi_H^{(l)}$ denotes the total derivative of $\phi_H$ of order $l,$ and in addition that $\|\phi'\|_\infty\le A.$
Then for every $R\geq1$ there exists a wave packet decomposition adapted to $S$ and the decomposition  of $\R^{d+1}$ into $\R^{d+1}=H+\R n,$  with tubes of radius
$R/D=R'$ and length $R^2/(D^2\kappa)=R'^2/\kappa,$ where  $R=R'D.$
\smallskip

More precisely, there exists an $R'$-lattice $\Y$ in $H$ and  an $R'^{-1}$-net $\V$ in $S$ such that the following hold true:
 if we denote by $ \W$ the index set  $\W=\Y\times \V$ and associate to $w=(y,v)\in \Y\times \V= \W$ the tube-like set
\begin{align}\label{wavepack2}
	T_w=y+\Big\{tN(v):|t|\leq\frac{R'^2}{\kappa}\Big\}+B(0,R'),
\end{align}
then for  every given  function $f\in L^2(U)$  there exist functions (wave packets) $\{p_w\}_{w\in \W}$ and coefficients $c_w\in\C$ such that  for every $x=x'+tn\in \R^{d+1}$  with $|t|\leq\frac{R'^2}{\kappa}$ and  $x'\in H,$  we may decompose $R_H^*f(x)$ as
$$
R_H^*f(x)=\sum\limits_{w\in\W} c_w p_w(x),
$$
 in such a way that  the following hold true:

\begin{enumerate}
\renewcommand{\labelenumi}{(P\arabic{enumi})} 
	\item $p_w= R_H^*(\FT_H^{-1} (p_w|_H))$,
	\item $\supp\FT_{\R^{d+1}} p_w\subset B(v,R'^{-1})$ and $\supp\FT_H (p_w(\cdot+tn))\subset B(v',\landau(R'^{-1}))$,
					where $v'$ denotes the orthogonal projection of $v\in S$ to $H$.
	\item $p_w$ is essentially supported in $T_w$, i.e.,
					$|p_w(x)|\leq C_N R'^{-1} \left(1+\frac{\dist(x,T_w)}{R'}\right)^{-N}$.
	\item For all $W\subset \W$, we have
					$ \|\sum\limits_{w\in W} p_w(\cdot+tn)\|_{L^2(H)}\lesssim |W|^\frac{1}{2}$.
		\item $\|c\|_{\ell^2} \lesssim \|f\|_{L^2}$.
\end{enumerate}

Moreover, the constants arising  in these estimates  can be chosen to depend only on the constants  $A_l$ and $A,$ but no further on the function $\phi_H,$ and also not on the other quantities $R,D$ and $\kappa$ (such constants will be called admissible).
\end{kornr}

Notice that, unlike as in Lemma \ref{wave packets}, we may here  choose an $R'^{-1}$-net  in $S$ in place of an $R'$ -lattice in $H$ for the parameter set $\V,$ because of our assumed   bound on $\phi_H'.$

It will become important that under suitable additional assumptions on the position of a given hyperplane $H$, we may re-parametrize a given smooth hypersurface  $S=\{(\xi,\phi(\xi)):\xi \in U\}$  (where $U$ is an open subset of $\R^d$) also in the form
$$
S=\{\eta+\phi_H(\eta) n: \eta\in U_H\},
$$
where $U_H$ is an open subset of $H$ and  $\phi_H\in C^\infty(U_H,\R).$

\begin{lemnr}[Re-parametrization]\label{reparametrize}
Let $H_1=n_1^\perp$ and $H_2=n_2^\perp$ be two hyperplanes in $\R^{d+1},$ where $n_1$ and $n_2$ are given unit vectors. Let $K=H_1\cap H_2,$ and choose  unit vectors $h_1,h_2$ orthogonal to $K$ such that
$
H_1=K+\R h_1, \quad H_2=K+\R h_2.
$
Let $U_1\subset H_1$ be an open bounded subset such that for every $x'\in K,$ the section $U_1^{x'}=\{u\in\R: x'+uh_1\subset U_1\}$ is an (open) interval, and let $\phi_1\in C^\infty(U_1,\R)$ satisfying the assumptions of Corollary \ref{wave packets2}. Setting $B=\kappa D^2$ and $r=D^{-1},$  an equivalent way to state this is that
 there are constants $B,r>0$ such that $Br\leq 1$, $\|\phi_1'\|_\infty\le A$ and $\|\phi_1^{(l)}\|_\infty \le A_l B r^{l}$ for every $l\in\N$ with $l\ge2.$ Denote by $S$ the hypersurface
$$
S=\{\eta+\phi_1(\eta) n_1: \eta\in U_1\}\subset \R^{d+1},
$$
and again by $v\mapsto N(v)$ the corresponding unit normal field on $S.$

Assume  furthermore that the vector  $n_2$ is  transversal to $S,$ i.e.,   $|\langle n_2, N(v)\rangle|\ge a>0$ for all $v\in S.$
Then there exist an open  bounded subset  $U_2\subset H_2$ such that for every $x'\in K,$ the section $U_2^{x'}=\{s\in\R: x'+sh_2\in U_2\}$ is an  interval,   and a function $\phi_2\in C^\infty(U_2,\R)$ so that we may re-write
\begin{align}
S=\{\xi+\phi_2(\xi )n_2: \xi\in U_2\}. \label{Srepar}
\end{align}
Moreover, the derivatives of $\phi_2$ satisfy estimates of the same form as those of  $\phi_1,$ up to multiplicative constants which are admissible, i.e.,  which depend only on the constants $A_l,A$ and $a.$

Finally,  given any $f_1\in L^2(U_1),$ there exists a unique function $f_2\in L^2(U_2)$ such that
\begin{align}
R^*_{H_1}f_1=R^*_{H_2} f_2,\label{Rrepar}
\end{align}
and $\|f_1\|_2\sim \|f_2\|_2,$  where the constants in these estimates are admissible.

\end{lemnr}

 \begin{proof} Assume that \eqref{Srepar} holds true. Then, given any point $\eta+\phi_H(\eta)n_1\in S,$ with $\eta=x'+uh_1\in U_1, x'\in K, $ we find  some $\xi=x'+sh_2\in U_2$ such that
\begin{align}\label{su2}
 x'+uh_1+\phi_1(x'+uh_1)n_1=x'+sh_2+\phi_2(x'+sh_2)n_2,
 \end{align}
 which shows that necessarily
\begin{align}\label{su}
 s=\langle uh_1+x'+\phi_1(x'+uh_1)n_1,h_2\rangle.
 \end{align}
 Let us  therefore define the mapping $G:U_1\to H_2$  by
 $$
 G( x'+uh_1)=x'+\langle uh_1+x'+\phi_1(x'+uh_1)n_1,h_2\rangle h_2.
 $$
 Moreover, fixing an orthonormal basis $E_1,\dots ,E_{d-1}$ of $K$ and extending this by  the vector $h_1$ respectively $h_2$ in order to obtain bases of $H_1$ and $H_2$ and working in the corresponding coordinates, we may  assume without loss of generality that $U_1$ is an open subset  of $\R^{d-1}\times \R,$  since $\dim K=d-1,$  and that
 $G$  is a  mapping   $G: U_1\to \R^{d-1}\times \R,$ given by
 $$
 G( x',u)=(x',g(x',u)),
 $$
 where
 $$
 g(x',u)=\langle x'+ uh_1+\phi_1(x',u)n_1,h_2\rangle.
 $$
 To show that $G$ is a diffeomorphism onto its image $U_2=G(U_1),$ observe that
 $$
 \pa_u G( x',u)=(0,\pa_u g( x',u))=(0,\langle h_1+ \pa_u\phi_1(x',u)n_1,h_2\rangle).
 $$
 On the other hand, the vector
 $$
 N_0= -\pa_u \phi_1(x',u ) h_1-\sum_{j=1}^k \pa_{x_j} \phi(x',u) E_j +n_1
 $$
 is normal to $S$ at the point $x'+uh_1+\phi_1(x'+uh_1)n_1$  (here $x'=\sum_{j=1}^{d-1} x_j E_j$), and  $|N_0|\sim 1.$
 Thus, our  transversality assumption implies that
\begin{align}\label{transcon}
 | \langle-\pa_u \phi_1(x',u ) h_1+n_1, n_2\rangle|\gtrsim a >0.
 \end{align}
 But, $\{h_j,n_j\}$ forms an orthonormal basis of $K^\perp$ for $j=1,2,$ and thus, rotating all these vectors by an angle of $\pi/2,$  we see that \eqref{transcon} is equivalent to  $|\langle\pa_u \phi_1(x',u ) n_1+h_1,h_2\rangle|\gtrsim a >0,$ so that
 $$
 | \pa_u g( x',u)|\gtrsim a>0.
 $$
Given the special form of $G,$ this also implies that
$$
|\det G'(x',u)|=| \pa_u g( x',u)|\gtrsim a>0.
$$
Consequently,  for $x'$ fixed, the mapping $u\mapsto g(x'.u)$ is a diffeomorphism from the interval $U_1^{x'}$ onto an open interval $U_2^{x'},$ and thus $G$ is bijective  onto its image $U_2,$ in fact even a  diffeomorphism, and $U_2$ fibers into the intervals $U_2^{x'}.$ Indeed,  the inverse mapping  $F=G^{-1}:U_2\to U_1$  of $G$ is also of the form
$$
F(x',s)=(x',f(x',s)),
$$
where
\begin{align}\label{gf}
g(x',f(x',s)=s.
\end{align}
In combination with \eqref{su2} this shows that \eqref{Srepar} holds indeed true,   with
\begin{align}\label{phi2}
\phi_2(x',s)=f(x',s)\langle h_1, n_2\rangle+ \phi_1(F(x',s)) \langle n_1,n_2\rangle.
\end{align}
 Moreover,
if $f_1\in L^2(U_1),$ then, by \eqref{su2} and a change of coordinates,
\begin{align*}
&R^*_{H_1}f_1(y)= \iint_{U_1} f_1(x',u) e^{-i\langle y,x'+uh_1+\phi_1(x',u)n_1\rangle} \, dx' du\\
&=\iint_{U_2} f_1(F(x',s))|\det F'(x',s)|\,  e^{-i\langle y,x'+sh_2+\phi_2(x',s)n_2\rangle} \, dx' ds,
\end{align*}
so that \eqref{Rrepar} holds true, with
\begin{align}\label{f1-f2}
f_2(x'+sh_2)=f_1\big(x'+f(x',s)h_1\big)\, |\det F'(x',s)|.
\end{align}
Our estimates for derivatives of $F$ show that $|\det F'(x',s)|\sim 1,$  with admissible constants,  so that in particular $\|f_1\|_2\sim \|f_2\|_2.$

\medskip
What remains is the control of the derivatives of $\phi_2.$  This  somewhat technical part of the proof will be based on Fa\`a di Bruno's theorem and is deferred to the   Appendix (see Subsection \ref{faa}).
\end{proof}

\medskip
We shall from now on restrict ourselves to dimension $d=2.$ The following lemma will deal with the separation of tubes along certain types of curves, for a  special class of 2-hypersurfaces. It will later be applied to intersection curves of two hypersurfaces.


\begin{lemnr}[Tube-separation along intersection curve]\label{separationlemma}
Let $\Y,\V,\W,R,T_w$ be as in Corollary \ref{wave packets2}. Moreover assume that  $\phi\in C^\infty(U,\R)$, $U\subset\R^2,$ such that $\partial_i^2\phi(x)\sim \kappa_i$ for all $x\in U,i=1,2,$ and $\partial_1\partial_2\phi=0$. Define $\kappa=\kappa_1\vee\kappa_2$.
Let $\gamma=(\gamma_1,\gamma_2)$ be a curve in $U$ with such that  $|\dot{\gamma}_i|\sim  1$ for $i=1,2.$
Then for all pairs of points $v_1,v_2\in {\rm im}(\gamma)+\landau(R'^{-1})$ such that  $v_1-v_2={j}/{R'}$, where  $j\in \Z^2$ and $|j|\gg 1,$ the following separation condition holds true (again with constants in these estimates which are admissible in the obvious sense):
\begin{align*}
	|\nabla\phi(v_1)-\nabla\phi(v_2)| \sim  |j| \frac{R'}{R'^2/\kappa}.
\end{align*}
\end{lemnr}

\begin{proof}
Choose $t_1,t_2$ such that $v_i=\gamma(t_i)+\landau(R'^{-1})$. Then
\begin{align}\label{mws10jan14}
	|\nabla\phi(v_i)-\nabla\phi(\gamma(t_i))|
	\leq \|\phi''|_U\|_\infty |v_i-\gamma(t_i)|
	\lesssim \frac{\kappa}{R'}.
\end{align}
Therefore
\begin{align*}
	\frac{|j|}{R'} =& |v_1-v_2|
	=|\gamma(t_1)-\gamma(t_2)|+\landau(R'^{-1})
	\sim  |\dot{\gamma}_1||t_1-t_2|+|\dot{\gamma}_2||t_1-t_2|+\landau(R'^{-1})\\
	\sim & |t_1-t_2|+\landau(R'^{-1}),
\end{align*}
and since $|j|\gg1$, we see that $|t_1-t_2| \sim  {|j|}/{R'}.$
By our assumptions on $\phi$ and \eqref{mws10jan14}, we thus see that there exist $s_1$ and $s_2$ lying between $t_1$ and $t_2$ such that
\begin{align*}
	|\nabla\phi(v_1)-\nabla\phi(v_2)|  \geq &|\nabla\phi(\gamma(t_1))-\nabla\phi(\gamma(t_2))| - \kappa\landau(R'^{-1})\\
	\sim  &\big( |\partial_1^2\phi(\gamma(s_1))\dot{\gamma}(s_1)|
						+|\partial_2^2\phi(\gamma(s_2))\dot{\gamma}(s_2)|\big)\,|t_1-t_2|
						- \kappa \landau(R'^{-1})\\	
	\sim  & (\kappa_1 + \kappa_2)\frac{|j|}{R'} + \kappa \landau(R'^{-1})\\	
	\sim  & 	|j| \frac{\kappa}{R'},			
\end{align*}
where we used again that $|j|\gg1$.
\end{proof}

\subsection{A bilinear  estimate for normalized hypersurfaces }\label{generalbilinear}
In this section, we shall work under the following
\medskip

\noindent\textsc{General Assumptions:}
Let $\phi\in C^\infty(\R^2)$ such that $\partial_1\partial_2\phi\equiv0$, and let
$$S_j=\{(\eta,\phi(\eta):\eta\in U_j\}, \quad U_j=r^{(j)}+[0,d^{(j)}_1]\times[0,d^{(j)}_2],\qquad j=1,2,
$$
where $r^{(j)}\in\R^2$ and  $d^{(j)}_1,d^{(j)}_2>0.$ We assume that
the principal curvature of $S_j$ in the direction of  $\eta_1$  is comparable to  $\kappa^{(j)}_1>0,$ and in the direction of  $\eta_2$    to  $\kappa^{(j)}_2>0,$    up to some fixed  multiplicative constants.
 We then put for $j=1,2,$
  \begin{align}\nonumber
 &\kappa^{(j)}=\kappa^{(j)}_1\vee\kappa^{(j)}_2, \qquad \bar\kappa_i=\kappa^{(1)}_i\vee\kappa^{(2)}_i\\
&\bar\kappa=\bar\kappa_1\vee\bar\kappa_2=\kappa^{(1)}\vee\kappa^{(2)} \label{players}\\
&\bar d_i=d_i^{(1)}\vee d^{(2)}_i, \qquad D=\min\limits_{i,j} d_i^{(j)}.\nonumber
\end{align}
The vector field  $N=(-\nabla\phi,1)$ is normal to $S_1$ and $S_2,$ and thus  $N_0={N}/{|N|}$ is a unit normal field to these hypersurfaces.
We make the  following additional assumptions:
\begin{enumerate}
\item For all $i,j=1,2$ and all $\eta\in U_j$, we have
\begin{align}\label{13jan2}
|\partial_i\phi(\eta)-\partial_i\phi(r^{(j)})|
\lesssim\kappa_i^{(j)} d_i^{(j)} \quad
\mbox{and} \quad\bar\kappa_i\bar d_i \lesssim 1
\end{align}
(notice that the first inequality follows already  from our earlier  assumptions).
\item For all $\eta\in U_1\cup U_2$ and for all $\alpha\in\N^2,|\alpha|\geq2,$ we have
			  $|\partial^\alpha\phi(\eta)|\lesssim \bar\kappa D^{2-|\alpha|}$.
\item  For $i=1,2,$ i.e., with respect to  both variables, the  following separation condition holds true:
\begin{align}\label{13jan1}
|\partial_i\phi(\eta^1)-\partial_i\phi(\eta^2)|\sim 1\quad \mbox{for all } \eta^j\in U_j, \ j=1,2.
\end{align}

		
\end{enumerate}
The set of all pairs $(S_1, S_2)$ of hypersurfaces satisfying these  properties will be denoted by $\s_0$ (note that it  does depend on the constants hidden by the symbols $\lesssim$ and $\sim$).

\medskip
The main goal of this chapter will be to establish a local, bilinear  Fourier extension  estimate on suitable cuboids adapted to the wave packets.

\begin{thmnr}\label{Eathm}
Assume that $5/{3}\leq p\leq2.$ Let us choose $r\in\R^2$ such that $r=r^{(j)},$ if $\kappa^{(j)}=\kappa^{(1)}\wedge\kappa^{(2)}.$ Then for every $\alpha>0$ there exist constants $C_\alpha,\gamma_\alpha>0$ such that for every pair $S=(S_1,S_2)\in\s_0$, every parameter $R\geq1$ and all functions $f_j\in L^2(S_j)$, $j=1,2$ we have
\begin{align}\label{bilin}
\|R_{\R^2}^*f_1\,  R_{\R^2}^*f_2\|_{L^p(Q^0_{S_1,S_2} (R))} \leq C_\alpha
	R^\alpha  (\kappa^{(1)}\kappa^{(2)})^{\frac{1}{2}-\frac{1}{p}}D^{3-\frac{5}{p}} \log^{\gamma_\alpha}(C_0(S))\,
	\|f_1\|_2 \|f_2\|_2,			
\end{align}
where
\begin{align}\label{Qscaled}
	Q^0_{S_1,S_2}(R)=\left\{x\in\R^3:~ |x_i+\partial_i\phi(r)x_3|\leq \frac{R^2}{D^2\bar\kappa},i=1,2,~|x_3|\leq\frac{R^2}{D^2(\kappa^{(1)}\wedge\kappa^{(2)})} \right\},
\end{align}
with
\begin{align}\label{C0}
 C_0(S)=\frac{\bar d_1^2\bar d_2^2}{D^4} (D[\kappa^{(1)}\wedge\kappa^{(2)}])^{-\frac{1}{p}}
							(D\kappa^{(1)} D\kappa^{(2)})^{-\frac{1}{2}}.
\end{align}
\end{thmnr}
Notice that $C_0(S)\gtrsim 1.$
\begin{remnr} If  $\kappa^{(1)}=\kappa^{(2)}=\bar\kappa$, then $r$ is not well defined. But in this case the two sets
	$Q^0_{S_1,S_2}(R;j)=\{x\in\R^3:~ |x_i+\partial_i\phi(r^{(j)})x_3|\leq \frac{R^2}{D^2\bar\kappa},i=1,2,~|x_3|\leq\frac{R^2}{D^2\bar\kappa} \}$, $j=1,2$, essentially  coincide.
	Indeed,  since
$|\nabla\phi(r^{(1)})-\nabla\phi(r^{(2)})|\sim  1$ (due to the transversality assumption (iii)),
an easy geometric consideration 
shows that
\begin{equation*}
	a Q^0_{S_1,S_2}(R;1)\subset Q^0_{S_1,S_2}(R;2) \subset b Q^0_{S_1,S_2}(R;1),
\end{equation*}
for some constants $a,b$ which do not depend on $R$ and  the class $\s_0$ from which $S=(S_1,S_2)$ is taken.
\end{remnr}

By applying a suitable affine transformation whose linear part fixes the points of $\R^2\times \{0\},$  if necessary,\color{black} we may assume without loss of generality  that $r=0$ and $\nabla\phi(r)=0$. Notice that conditions (i)-(iii) and the conclusion of the theorem are invariant under such affine transformations.

In fact, we we shall  then prove estimate \eqref{bilin} in the  theorem on the even larger cuboid
\begin{align}\label{largecuboid}
	Q_{S_1,S_2}(R)=\left\{x\in\R^3:~ |x_{i_0}|\leq \frac{R^2}{D^2\bar\kappa},~\|x\|_\infty\leq\frac{R^2}{D^2(\kappa^{(1)}\wedge\kappa^{(2)})} \right\},
\end{align}
for an appropriate choice of the coordinate direction $x_{i_0}, i_0\in\{1,2\},$ in which the cuboid has  smaller side length. Later we shall need to combine different cuboids which may possibly have their smaller side lengths in different directions. Then  it will become necessary to restrict to their intersection, which leads to \eqref{Qscaled}.

Indeed,  we shall see that there will be two directions in which the  side length of the cuboids are   dictated  by the length of the wave packets, and one  remaining third direction for which we shall have more freedom in choosing the side length.

Observe also that  $\bar\kappa_i\bar d_i \lesssim 1$, and thus we may even assume without loss of generality that
\begin{align}\label{13feb1748}
	\bar\kappa_i\bar d_i\ll 1\quad\text{for all }i=1,2,
\end{align}
simply by decomposing  $S_1$ and $S_2$ into a finite number of subsets for which the side lengths of corresponding rectangles $U_j$  are sufficiently small fractions of the given $d_i^{(j)}.$

For $\eta^j\in U_j$ define
\begin{align*}
\phi_1(\eta)=\phi(\eta-\eta^2)+\phi(\eta^2),\qquad  \eta\in\eta^2+U_1\\
\phi_2(\eta)=\phi(\eta-\eta^1)+\phi(\eta^1),\qquad \eta\in\eta^1+U_2.
\end{align*}
 The set  $((\eta^2,\phi(\eta^2))+S_1)\cap((\eta^1,\phi(\eta^1))+S_2)=\mathop{\rm graph}(\phi_1)\cap\mathop{\rm graph}(\phi_2)$
will be called an {\it intersection curve} of $S_1$ and $S_2$. It agrees with  the graph of $\phi_1$ (or $\phi_2$) restricted to the set where $\psi=\phi_1-\phi_2=0$.
On this set, the normal field $N_j(\eta)=(-\nabla\phi_j(\eta),1)$ forms the conical set
$$\Gamma_j=\{sN_j(\eta)|s\in\R,\psi(\eta)=0\}.$$
In the sequel, we shall use the abbreviation $j+1\modd 2=2,$ if $j=1,$ and $j+1\modd 2=1,$ if $j=2.$


\begin{lemnr}\label{propertylemma}
Let $(S_1,S_2)\in\s_0$. Assume that $\nabla\phi(r)=0$ for some $r\in S_1\cup S_2$ and $\bar\kappa_i\bar d_i\ll 1$. Then the following hold true:
\begin{enumerate}
		\item[(a)] $D\kappa_i^{(j)} \ll 1$ for all $i,j=1,2$.
		\item[(b)] $|\nabla\phi(\eta)|\lesssim 1$ for all $(\eta,\phi(\eta))\in S_1\cup S_2$.
		\item[(c)] The unit normal fields on $S_1$ and $S_2$ are transversal, i.e.,
					\begin{align}\label{transversality1}
						|N_0(\eta^1)-N_0(\eta^2)|\sim  1\quad\text{for all }(\eta^j,\phi(\eta^j))\in S_j.
					\end{align}
		\item[(d)] $N_j$ and $\Gamma_{j+1\modd 2}$ are transversal for $j=1,2$ and for
				any choice of intersection curve  of $S_1$ and $S_2$.
		\item[(e)] If $\gamma$ is a parametrization by arclength $t$ of the projection of an intersection curve of $S_1$ and $S_2$ to the first two coordinates $\eta\in\R^2,$ then $|\dot{\gamma}_1|\sim 1\sim |\dot{\gamma}_2|$.	
\end{enumerate}
\end{lemnr}

\begin{proof} We shall denote by $\eta=\tilde x\in\R^2$ the projection of a point in $x\in \R^3$ to its first two coordinates.
(a) is clear since $D=\min\limits_{i,j=1,2} d_i^{(j)}$. To prove (b), notice that for any $x,x'\in S_1\cup S_2$ we have $|\nabla\phi(\tilde x)-\nabla\phi(\tilde x')|\lesssim 1$: if $x$ and $x'$ belong to different hypersurface  $S_j$, we apply condition (iii) from the definition of $\s_0$, and if $x$ and $x'$ are in the same hypersurface  $S_j$, we use condition (a). Thus we have $|\nabla\phi(\tilde x)|=|\nabla\phi(\tilde x)-\nabla\phi(r)|\lesssim1$ for all $x\in S_1\cup S_2$.\\
This gives $|N(\tilde x)|=\sqrt{1+|\nabla\phi(\tilde x)|^2} \sim  1$ for all $x\in S_1\cup S_2,$ which already implies
 the transversality of the normal fields: 
$$
|N_0(\eta^1)-N_0(\eta^2)|\sim  |N(\eta^1)-N(\eta^2)|
	=|\nabla\phi(\eta^1)-\nabla\phi(\eta^2)|\sim  1
$$
for all $(\eta^j,\phi(\eta^j))\in S_j, \ j=1,2.$

We shall prove (e) first, since (e) will be needed for the proof of (d). It suffices to prove  that $|\partial_i\psi(\eta)|\sim 1$ for all $\eta$ such that $\eta-\eta^j\in U_{j+1\modd 2}$, $\eta^j\in U_j,$ since the tangent to the curve $\gamma$ at any point $\gamma(t)$ is  orthogonal to $\nabla\psi(\gamma(t)).$
But, in view of \eqref{13jan1},
$$
	|\partial_i\psi(\eta)| = |\partial_i\phi(\eta-\eta^2)-\partial_i\phi(\eta-\eta^1)|
	{\sim } 1.	
$$
For the  claim (d), since the claim is symmetric in $j\in\{1,2\},$ it suffices to show that $N_1$ and $\Gamma_2$ are transversal. Since we have
\begin{align*}
	|N_1(\eta)-N_1(\eta')|=|\nabla\phi_1(\eta)-\nabla\phi_1(\eta')|
	\lesssim \kappa^{(1)}_1 d^{(1)}_1 + \kappa^{(1)}_2 d^{(1)}_2 \ll 1
\end{align*}
for all $\eta,\eta'\in U_1+\eta^2$, whereas $|N_1(\eta)|\sim  1$ for all $\eta\in U_1+\eta^2$,
it is even enough to show that
$N_1(\eta)$ and the tangent space $T_{N_2(\eta)}\Gamma_2$ of $\Gamma_2$ at the point $N_2(\eta)$ are transversal.
Since  $\gamma$ is a parametrization by arclength of the zero set of $\psi$, the tangent space of $\Gamma_2$ at the point $N_2(\eta)$ for $\eta=\gamma(t)$ is spanned by $N_2(\eta)$ and $(-D^2\phi_2(\eta)\;\dot{\gamma}(t),0)$,  where $D^2\phi_2$ denotes the Hessian matrix of $\phi_2.$  But, recalling that we assume that
$\partial_1\partial_2\phi\equiv 0,$ we see that the vectors
 $N_2(\eta)$ and $\frac{1}{\kappa^{(2)}}(-D^2\phi_2(\eta)\;\dot{\gamma}(t),0)$  form an ``almost''  orthonormal frame for the  tangent space $T_{N_2(\eta)}\Gamma_2,$ and thus the  transversality can be checked by estimating the volume $V$  of the parallelepiped spanned  by $N_1(\eta)$ and these two vectors, which is given by
 \begin{align*}
	V&= \left|\begin{array}{ccc} -\partial_1\phi_1(\eta) & -\partial_2\phi_1(\eta) & 1	\\
									-\partial_1\phi_2(\eta) & -\partial_2\phi_2(\eta) & 1	\\
									\frac{1}{\kappa^{(2)}}\partial_1^2\phi_2(\eta)\dot{\gamma}_1(t) &
									 \frac{1}{\kappa^{(2)}}\partial_1^2\phi_2(\eta)\dot{\gamma}_2(t) & 0 \end{array}\right|	\\
	&= \frac{1}{\kappa^{(2)}} \big|-\partial_1^2\phi_2(\eta)\dot{\gamma}_1(t)\partial_2\psi(\eta)
														 +\partial_2^2\phi_2(\eta)\dot{\gamma}_2(t)\partial_1\psi(\eta)\big|.							
\end{align*}
Since $\psi\circ\gamma=0$ by definition, we have
$ \partial_1\psi(\eta)\dot{\gamma}_1(t) + \partial_2\psi(\eta)\dot{\gamma}_2(t),$ hence
\begin{align*}
\partial_2\psi(\eta)=& -\partial_1\psi(\eta)\frac{\dot{\gamma}_1(t)}{\dot{\gamma}_2(t)}.
\end{align*}
Thus
\begin{align*}
		V=& \frac{|\partial_1\psi(\eta)|}{\kappa^{(2)}|\dot{\gamma}_2(t)|}
			(\partial_1^2\phi_2(\eta)\dot{\gamma}_1^2(t)	+\partial_2^2\phi_2(\eta)\dot{\gamma}_2^2(t))	 \\
	\sim & |\partial_1\phi(\eta-\eta^2)-\partial_1\phi(\eta-\eta^1)|
									\frac{\kappa^{(2)}_1 + \kappa^{(2)}_2}{\kappa^{(2)}}		
	\sim  1,
\end{align*}
finishing the proof.
\end{proof}

%
%



We now come to the introduction of the wave packets that we shall  use in the proof of Theorem \ref{Eathm}. Let us  assume without loss of generality  that
\begin{align}\label{2feb1351}
	\kappa^{(1)}\leq \kappa^{(2)},
\end{align}
  i.e., $r=r^{(1)}$ and $\nabla\phi(r^{(1)})=0$.
\medskip


Next, since  $S_1$ is horizontal at $(r^{(1)}, \phi(r^{(1)})),$ we may use  the wave packet decomposition from Corollary \ref{wave packets2}, with normal $n_1$ and hyperplane $H_1$ given by
\begin{align*}
	n_1=(0,0,1)\qquad\text{and}\qquad H_1=\R^2\times\{0\},
\end{align*}
 in order to decompose
   \begin{align}\label{wp1}
  R^*_{\R^2}f_1=R_{H_1}^*f_1=\sum_{w_1\in\W_1}c_{w_1} p_{w_1}, \quad w_1\in\W_1,
  \end{align}
 into wave packets  $p_{w_1}, w_1\in \W_1,$ of length ${R'^2}/\kappa^{(1)},$ directly by means of Lemma \ref{wave packets}. By $T_{w_1},w_1\in \W_1,$ we denote the associated set of tubes.  Recall that this decomposition is valid on the set
   $P_1=\R^2\times[-\frac{R'^2}{\kappa^{(1)}},\frac{R'^2}{\kappa^{(1)}}].$

   \medskip
   Let us next turn to  $S_2$ and $R^*_{\R^2} f_2.$
If we would keep the same coordinate system for $S_2$, we would have to truncate even further in $x_3$-direction, since $\frac{R'^2}{\kappa^{(2)}}\leq\frac{R'^2}{\kappa^{(1)}}$.
However, by  {\eqref{13jan1} we have for $\eta\in U_2$ and both $i=1$ and $i=2$ that
$$|\langle e_i, N(\eta) \rangle|=|\partial_i\phi(\eta)|
=|\partial_i\phi(\eta)-\partial\phi(r^{(1)})| \sim  1.$$
This means  that we may  apply  Lemma \ref{reparametrize} to $S_2$  in order to re-parametrize $S_2$ by an open subset (denoted again by $U_2$) of the hyperplane $H_2=n_2^\perp$ given by
\begin{align*}
	n_2=e_{i_0}
	\qquad\text{and}\qquad
	H_2=\{n_2\}^\bot = \{e_{i_0}\}^\bot.
	\end{align*}
We may thus  replace the function $f_2$ by a function (also denoted by $f_2$)  on $U_2$ of comparable $L^2$-norm, and replace $R^*_{\R^2} f_2$ by $R^*_{H_2} f_2$ in the subsequent arguments.
\medskip

Next, applying Corollary \ref{wave packets2}, now with $H=H_2,$  then for $i_0=1,$ as well as for $i_0=2,$  we may  decompose
 \begin{align}\label{wp2}
  R_{H_2}^*f_2=\sum_{w_2\in\W_2}c_{w_2} p_{w_2}, \quad w_2\in\W_2,
  \end{align}
 on the set
$$P_2=\left\{x\in\R^3: |\langle x,n_2\rangle |\leq\frac{R'^2}{\kappa^{(2)}} \right\}
= \left\{x\in\R^3: |x_{i_0}|\leq\frac{R'^2}{\kappa^{(2)}} \right\}$$ by means of wave packets of of length ${R'^2}/\kappa^{(2)}.$ The associates set of tubes will be denoted by  $T_{w_2},w_2\in \W_2.$

In order to  decide how to chose $i_0$, we observe that for $\eta\in U_1,$  our definitions \eqref{players} in combination with the estimates \eqref{13jan2} and \eqref{13feb1748} show that
\begin{align*}
	|\partial_i\phi(\eta)-\partial_i\phi(r^{(1)})|
\lesssim \kappa_i^{(1)} d_i^{(1)}
	\leq \frac{\kappa_i^{(1)}}{\bar\kappa_i} \bar\kappa_i\bar d_i
	\ll \frac{\kappa_i^{(1)}}{\bar\kappa_i}.
\end{align*}
Notice that  the wave packets associated  to $S_1$ are roughly pointing in the direction of  $N(r^{(1)})=(0,0,1).$  More precisely, if  we project a  wave packet pointing in direction of  $N(\eta)$, $\eta\in U_1,$ to the coordinate $x_i, i=1,2,$ then by the previous estimates we see that  we obtain an interval of length comparable to
\begin{align}\label{projectlength}
	|\langle e_i,\frac{R'^2}{\kappa^{(1)}} N(\eta)\rangle|			
	= \frac{R'^2}{\kappa^{(1)}} |\partial_i\phi(\eta)|	
	=& \frac{R'^2}{\kappa^{(1)}} |\partial_i\phi(\eta)-\partial_i\phi(r^{(1)})|
	\ll   \frac{R'^2}{\bar\kappa}\frac{\bar\kappa}{\kappa^{(1)}}\frac{\kappa_i^{(1)}}{\bar\kappa_i}.	
\end{align}
Let us therefore choose  $i_0$ so that
\begin{align*}
	\frac{\kappa^{(1)}_{i_0}}{\bar\kappa_{i_0}}
	= \frac{\kappa_1^{(1)}}{\bar\kappa_1}\wedge\frac{\kappa_2^{(1)}}{\bar\kappa_2}.
\end{align*}
Then
$$
\bar\kappa \frac{\kappa_{i_0}^{(1)}}{\bar\kappa_{i_0}}=(\bar\kappa_1\vee\bar\kappa_2)
				\left(\frac{\kappa_1^{(1)}}{\bar\kappa_1}\wedge\frac{\kappa_2^{(1)}}{\bar\kappa_2}\right)	
				 \le\kappa_1^{(1)}\vee\kappa_2^{(1)} = \kappa^{(1)},
$$
and thus by \eqref{projectlength} and \eqref{2feb1351}
\begin{align*}
	|\langle e_{i_0},\frac{R'^2}{\kappa^{(1)}} N(\eta)\rangle|
	\ll \frac{R'^2}{\bar\kappa}= \frac{R'^2}{\kappa^{(2)}}.
\end{align*}
This means that  the geometry fits well: the wave packets associated to  $S_1$ do not turn too much into the direction of $x_{i_0}:$  projected to this coordinate, their length is smaller  than the length of the wave packets associated to $S_2,$ which are essentially pointing in the direction of the $i_0$-th coordinate axis (cf. Figure \ref{qofr1}).

\begin{figure}
\begin{center}  \includegraphics{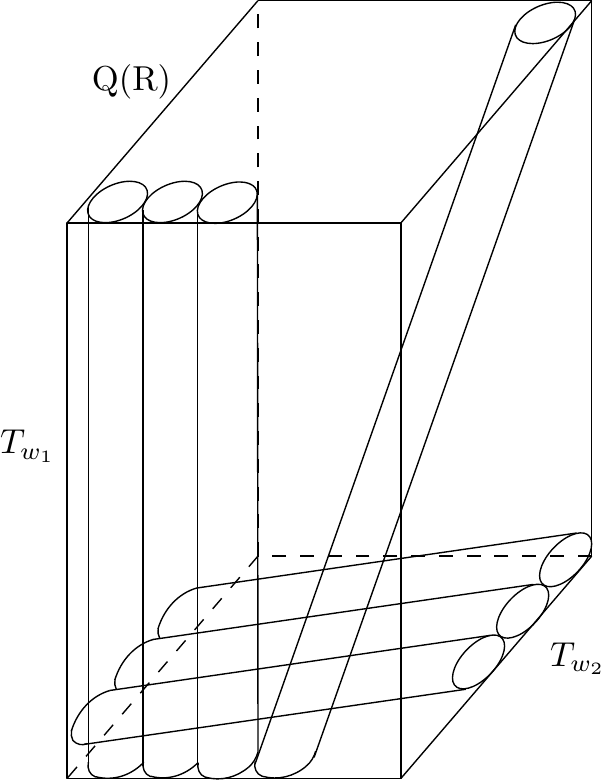} \end{center}
\caption{The wave packets filling the cuboid $Q_{S_1,S_2}(R)$}
\label{qofr1}
\end{figure}

\medskip
However, for the remaining coordinate direction $x_i$, $i\in\{1,2\}\setminus\{i_0\}$, we cannot guarantee such a behaviour.
But notice that by \eqref{2feb1351}
\begin{align*}
	&P_1\cap P_2=\bigg(\R^2\times\left[-\frac{R'^2}{\kappa^{(1)}},\frac{R'^2}{\kappa^{(1)}}\right]\bigg)
	\cap \left\{x\in\R^3: |x_{i_0}|\leq\frac{R'^2}{\kappa^{(2)}} \right\} \\
	&= \left\{(x_1,x_2)\in\R^2: |x_{i_0}|\leq\frac{R'^2}{\kappa^{(2)}} \right\}
										 \times\left[-\frac{R'^2}{\kappa^{(1)}},\frac{R'^2}{\kappa^{(1)}}\right]	\\
	&\supset  \left\{x\in\R^3:|x_{i_0}|\leq \frac{R'^2}{\bar\kappa},\
																		\|x\|_\infty \leq \frac{R'^2}{\kappa^{(1)}\wedge\kappa^{(2)}} \right\}		\\
	&=  Q_{S_1,S_2}(R),
\end{align*}
i.e., on the cuboid $Q_{S_1,S_2}(R)$ we may apply our development into   wave packets for the wave packets associated to  the hypersurface $S_1,$ as well as those associated to  $S_2.$
\medskip

For every $\alpha>0$,  let us denote by $E(\alpha)$ the following statement:
\medskip

{\noindent $E(\alpha):$}
There exist constants $C_\alpha>0$ and $\gamma_\alpha>0$ such that for all pairs $ S=(S_1,S_2)\in\mathcal{S}_0,$ all  $R\geq 1$ and all $ f_j\in L^2(U_j), \,j=1,2$ (which we may also regard as functions on $S_j$) the following estimate holds true:
\begin{align}\label{Ealest}
	\|R^*_{H_1}f_1 R^*_{H_2}f_2\|_{L^p(Q_{S_1,S_2}(R))} &\leq C_\alpha R^\alpha \log^{\gamma_\alpha}(1\!\!+\!\!R)\\
&\times (\kappa^{(1)}\kappa^{(2)})^{\frac{1}{2}-\frac{1}{p}}D^{3-\frac{5}{p}} \log^{\gamma_\alpha}(C_0(S))  \|f_1\|_2 \|f_2\|_2	\nonumber		.
\end{align}
Here, $C_0(S)$ denotes the constant defined in Theorem \ref{Eathm}.
\medskip

Our goal will be to show that $E(\alpha)$ holds true for every $\alpha>0,$ which would prove Theorem \ref{Eathm}. To this end, we shall apply the method of  induction on scales.
\medskip

Observe that the intersection of two  of the transversal tubes $T_{w_1}, w_1\in \W_1,$ and $T_{w_2}, w_2\in \W_2,$ will always be contained in a cube of side length $\landau(R').$
Let us therefore decompose $\R^3$ by means of a grid of side length $R'$ into cubes $q$ of the same side length, and  let $\{q\}_{q\in\mathcal Q}$ be a family of such cubes covering $Q_{S_1,S_2}(R).$ By $c_q$ we shall denote the center of the cube $q.$  Choose  $\chi\in\s(\R^3)$ with $\supp\hat\chi\subset B(0,1)$
 and $\hat\chi(0)=1/(2\pi)^n$, and put  $\chi_q(x)=\chi\left(\frac{x-c_q}{R'}\right).$  Poisson's summation formula then implies that  $\sum\chi_q=1$ on $\R^3,$ so that in particular we may assume that  $\sum_{q\in\mathcal Q}\chi_{q}=1$ on $Q_{S_1,S_2}(R)$.

Notice  that our approach  slightly differs from the standard usage of induction on scales, where
 $\chi_q$ is chosen to be  the characteristic function of $q,$ and not a smoothened version of it.  The price we shall have to pay is that some arguments will become  a bit more technical,  but the compact  Fourier support of the functions $\chi_q$ will become crucial  later.

For a given index set $W_j\subset\W_j, j=1,2, $ of wave packets (compare \eqref{wp1}, \eqref{wp2}), we denote by \begin{align*}
	W_j(q)=\{w_j\in W_j: T_{w_j}\cap R^\delta q\neq\emptyset\}
\end{align*}
the collection of  all the tubes of type $j$ passing through (a slightly thickened)  cube $q.$ Here,
 $\delta>0$ is a small parameter which will be fixed later,  and $R^\delta q$ denotes  the dilate  of q by  the factor $R^\delta$ having the same center $c_q$ as $q$.

 Let us denote by $\mathcal N$ the set  $\mathcal N=\{2^n|n\in\N\}\cup\{0\}.$ In order to count the magnitude of the number of wave packets $W_j$ passing through a given cube $q,$ we introduce the sets
\begin{align*}
	Q^\mu=\{q:|W_j(q)|\sim\mu_j,j=1,2\},\quad \mu=(\mu_1,\mu_2)\in\mathcal N^2.
\end{align*}
Obviously the $Q^\mu$ form a partition of the family of all cubes $q\in\mathcal Q.$
For $w_j\in W_j$, we further introduce the set of all cubes in $Q^\mu$ close to  $T_{w_j}$:
\begin{align*}
	Q^\mu(w_j)=\{q\in Q^\mu: T_{w_j}\cap R^\delta q\neq\emptyset\}.
\end{align*}
Finally, we determine the number of such cubes by means of the sets
\begin{align*}
	W_j^{\lambda_j,\mu} =\{w_j\in W_j:|Q^\mu(w_j)|\sim\lambda_j\}, \quad\lambda_j,\mu_1,\mu_2\in \mathcal N.
\end{align*}
For every fixed $\mu$, the family $\{W_j^{\lambda_j,\mu}\}_{\lambda_j\in\mathcal N}$ forms a partition of $W_j$.

We are now in a position to reduce the statement $E(\alpha)$ to a formulation in terms of wave packets.

\subsection{Reduction to a wave packet formulation}

Following basically a standard pigeonholing argument in combination with (P5), the estimate in $E(\alpha)$ can easily be reduced to a bilinear estimate for sums of wave packets (modulo an increase of the exponent $\gamma_\alpha$ by 5). It is in this reduction that  some power of the logarithmic  factor $\log (C_0(S))$ will appear, and we shall have to be a bit more precise than usually in order to identify  $C_0(S)$ as the expression given by \eqref{C0}.
\begin{lemnr}\label{E(a)reduction}
Let $\alpha>0$. Assume there are  constants $C_\alpha,\gamma_\alpha>0$ such that for all $(S_1,S_2)\in\mathcal{S}_0$ (parametrized by the open subsets $U_j\subset H_j$) the following estimate is satisfied:

Given any any two   families of wave packets $\{p_{w_1}\}_{w_1\in\mathcal W_1}$ and  $\{p_{w_2}\}_{w_2\in\mathcal W_2}$ associated to $S_1$ respectively $S_2$  as in the wave packet decomposition Corollary \ref{wave packets2}, where the   $p_{w_j},j=1,2,$ satisfy uniformly the estimates in (P2) -- (P5), then for
 all $R\geq 1$, all  $\lambda_j,\mu_j\in\mathcal N$ and all  subsets $W_j\subset\W_j,j=1,2,$  we have (with admissible constants)
 \begin{align}\label{E(a)reduction_eq}
	&\|\prod_{j=1,2}\sum_{w_j\in W_j^{\lambda_j,\mu}} p_{w_j} \sum_{q\in Q^\mu} \chi_q\|_{L^p(Q_{S_1,S_2}(R))}	\nonumber\\
	\leq& C_\alpha R^\alpha \log^{\gamma_\alpha}(1+R)
			(\kappa^{(1)}\kappa^{(2)})^{\frac{1}{2}-\frac{1}{p}}D^{3-\frac{5}{p}} \log^{\gamma_\alpha}(C_0(S))
			|W_1|^\frac{1}{2} |W_2|^\frac{1}{2}.
\end{align}
Then $E(\alpha)$ holds true.
\end{lemnr}

\begin{proof}
In order to show $E(\alpha),$ we may  assume without loss of generality  that $\|f_j\|_2=1$, $j=1,2$. Let us abbreviate $C_0(S)=C_0.$

First observe that for fixed $q$ and $v_j$ the number of $y_j$ such that the tube $T_{(y_j,v_j)}$ passes through $R^\delta q$ is bounded by $R^{c\delta}$, whereas the total number of $v_j\in V_j$ is bounded by
\begin{align}\label{Vbound2}
	|V_j|\sim R'^2 |U_j| \leq R^2\frac{\bar d_1\bar d_2}{D^2}.
\end{align}
Thus we have
$$
	|W_j(q)| \leq R^{2+c\delta} \frac{\bar d_1\bar d_2}{D^2} 	
	{\leq} R^{2+c} \frac{\bar d_1^2\bar d_2^2}{D^4} (D[\kappa^{(1)}\wedge\kappa^{(2)}])^{-\frac{1}{p}}
	(D\kappa^{(1)} D\kappa^{(2)})^{-\frac{1}{2}}=	R^{c'} C_0,						
$$
where we have  used property  (i) of Lemma \ref{propertylemma}. Consequently  $Q^\mu=\emptyset,$ if $\mu_j\gg R^{c'} C_0$ for some $j.$
Similarly, the number of cubes $q$ of side length $R'$ such that $R^\delta q$ intersects with a tube $T_{w_j}$ of length $R'^2/{\kappa^{(j)}}$ is bounded by $R^{c\delta} R'/\kappa^{(j)}=R^{1+c\delta}/D\kappa^{(j)}$. Since $D\leq \bar d_1, \bar d_2$, this implies $$
	|Q^\mu(w_j)|\leq \frac{R^{1+c\delta}}{D\kappa^{(j)}}	
	\leq R^{c'} \frac{\bar d_1^4\bar d_2^4}{D^8} (D[\kappa^{(1)}\wedge\kappa^{(2)}])^{-\frac{2}{p}}
								(D\kappa^{(1)} D\kappa^{(2)})^{-1}	
	= R^{c'} C_0^2,							
$$
and thus $W_j^{\lambda_j,\mu}=\emptyset,$ if $\lambda_j\gg  R^{c'} C_0^2.$
For $C\ge 0$ let us put  $\mathcal N(C)=\{\nu\in\mathcal N|\nu\lesssim C\}.$ Since $C_0\gtrsim 1,$  we then  see that
$$\mathcal Q= \bigcup\limits_{\mu_1,\mu_2\in\mathcal N(R^{c'}C_0^2)}Q^\mu,
$$
and for every fixed $\mu,$
$$ W_j=\bigcup\limits_{\lambda_j\in\mathcal N(R^{c'}C_0^2)} W_j^{\lambda_j,\mu}.
$$
These decompositions in combination with our assumed estimate \eqref{E(a)reduction_eq} imply that
\begin{align*}
	&\big\|\prod_{j=1,2}\sum_{w_j\in W_j} p_{w_j}\big\|_{L^p(Q_{S_1,S_2}(R))}	\\
	\leq& \sum_{\lambda_1,\lambda_2,\mu_1,\mu_2\in\mathcal N(R^{c'}C_0^2) }
				\big\|\prod_{j=1,2}\sum_{w_j\in W_j^{\lambda_j,\mu}} p_{w_j} \sum_{q\in Q^\mu} \chi_q\big\|_{L^p(Q_{S_1,S_2}(R))}	\\
	\leq& C_\alpha R^\alpha \log^4(R^{c'}C_0^2) \log^{\gamma_\alpha}(1+R)
				(\kappa^{(1)}\kappa^{(2)})^{\frac{1}{2}-\frac{1}{p}}D^{3-\frac{5}{p}} \log^{\gamma_\alpha}(C_0)
				|W_1|^\frac{1}{2} |W_2|^\frac{1}{2}
\end{align*}
for every $W_j\subset\W_j,j=1,2,$  hence
\begin{align} \label{E(a)reduction_eq1}
	\|\prod_{j=1,2}\sum_{w_j\in W_j} p_{w_j} \|_{L^p(Q_{S_1,S_2}(R))}
	&\leq C_\alpha R^\alpha \log^{\gamma_\alpha+4}(1+R)  \\
			&\times(\kappa^{(1)}\kappa^{(2)})^{\frac{1}{2}-\frac{1}{p}}D^{3-\frac{5}{p}} \log^{\gamma_\alpha+4}(C_0)
			|W_1|^\frac{1}{2} |W_2|^\frac{1}{2}.\nonumber
\end{align}

Recall  next that $R^*f_j=\sum\limits_{w_j\in\W_j}c_{w_j}p_{w_j}.$
Introduce the subsets $W_j^k=\{w_j\in\W_j:|c_{w_j}|\sim2^{-k}\},$  which allow to partition  $\W_j=\bigcup\limits_{k\in\N} W_j^k.$
We fix some $k_0,$ whose precise  value will be determined later. Then
\begin{align*}
	&\| \sum_{k>k_0}\sum_{w_1\in W_1^k}\sum_{w_2\in\W_2} c_{w_1}p_{w_1} c_{w_2}p_{w_2} \|_{L^p(Q_{S_1,S_2}(R))}\\
	&\leq |Q_{S_1,S_2}(R)|^\frac{1}{p} \sum_{k>k_0}
					\|\sum_{w_1\in W_1^k}\sum_{w_2\in\W_2} c_{w_1}p_{w_1} c_{w_2}p_{w_2}\|_\infty.
\end{align*}
The wave packets $p_{w_j}$ are well separated with respect to the parameter $y_j,$ and  by (P4), their $L^\infty$- norm is   of order $\landau(R'^{-1}).$ Moreover,  by  \eqref{Vbound2}  the number of
$v_j$'s is bounded by $R^2{\bar d_1\bar d_2}/{D^2}.$ Furthermore,  $|c_{w_1}|\lesssim 2^{-k}$ for every $w_1\in W_1^k$, and  by (P6) we have
$|c_{w_2}|\leq\| \{c_{w_2}\}_{w_2\in\W_2}\|_{\ell^2} \lesssim \|f_2\|_2=1$. Combining all this information, we may estimate
\begin{align*}
	&\| \sum_{k>k_0}\sum_{w_1\in W_1^k}\sum_{w_2\in\W_2} c_{w_1}p_{w_1} c_{w_2}p_{w_2} \|_{L^p(Q_{S_1,S_2}(R))}	\\
	\lesssim& \left(\frac{R'^6}{[\kappa^{(1)}\wedge\kappa^{(2)}]\kappa^{(1)}\kappa^{(2)}}\right)^\frac{1}{p}
			\frac{\bar d_1^2\bar d_2^2}{D^4} R^4  R'^{-2}  \sum_{k>k_0} 2^{-k}		\\
	\sim& R^{\frac{6}{p}+2}  C_0\, D^{3-\frac{5}{p}}
				(\kappa^{(1)}\kappa^{(2)})^{\frac 12-\frac{1}{p}}2^{-k_0}.
\end{align*}
If we now choose
$k_0=\log_2 C_0 + \log R^{\frac{6}{p}+2},$
 then we obtain
\begin{align} \label{12sep1543}
	\| \sum_{k>k_0}\sum_{w_1\in W_1^k}\sum_{w_2\in\W_2} c_{w_1}p_{w_1} c_{w_2}p_{w_2} \|_{L^p(Q_{S_1,S_2}(R))}
	\lesssim
	 D^{3-\frac{5}{p}} (\kappa^{(1)}\kappa^{(2)})^{\frac{1}{2}-\frac{1}{p}}.
	 	 \end{align}
In a  similar way we  also get
\begin{align}\label{12sep1544}
	\| \sum_{k_1\leq k_0}\sum_{w_1\in W_1^{k_1}}
	\sum_{k_2>k_0}\sum_{w_2\in W_2^{k_2}} c_{w_1}p_{w_1} c_{w_2}p_{w_2} \|_{L^p(Q_{S_1,S_2}(R))}
	 \lesssim D^{3-\frac{5}{p}} (\kappa^{(1)}\kappa^{(2)})^{\frac{1}{2}-\frac{1}{p}}.
\end{align}

The remaining terms can simply be estimated by
\begin{align*}
	&\| \sum_{k_1,k_2=1}^{k_0}\sum_{w_1\in W_1^{k_1}}\sum_{w_2\in\W_2^{k_2} }
														c_{w_1}p_{w_1} c_{w_2}p_{w_2} \|_{L^p(Q_{S_1,S_2}(R))}		\\
	\lesssim& \sum_{k_1,k_2=1}^{k_0} 2^{-k_1-k_2}
	\|\sum_{w_1\in W_1^{k_1}}\sum_{w_2\in\W_2^{k_2}} c_{w_1}2^{k_1}p_{w_1} c_{w_2}2^{k_2}p_{w_2} \|_{L^p(Q_{S_1,S_2}(R))}.
\end{align*}
Since $|c_{w_j}2^{k_j}|\sim1$ for $w_j\in W_j^{k_j}$, it is appropriate to apply \eqref{E(a)reduction_eq1} to the modified wave packets $\tilde p_{w_j} = c_{w_j}2^{k_j}p_{w_j}:$
\begin{align*}
	&\| \sum_{k_1,k_2=1}^{k_0}\sum_{w_1\in W_1^{k_1}}\sum_{w_2\in\W_2^{k_2} }
														c_{w_1}p_{w_1} c_{w_2}p_{w_2} \|_{L^p(Q_{S_1,S_2}(R))}	\\
	\leq& C_\alpha R^\alpha \log^{\gamma_\alpha+4}(1+R)
			(\kappa^{(1)}\kappa^{(2)})^{\frac{1}{2}-\frac{1}{p}}D^{3-\frac{5}{p}} \log^{\gamma_\alpha+4}(C_0)
			\sum_{k_1,k_2=1}^{k_0} 2^{-k_1-k_2}  |W_1^{k_1}|^\frac{1}{2} |W_2^{k_2}|^\frac{1}{2}.
\end{align*}
But observe that by (P5)
\begin{align*}
	\sum_{k_1,k_2=1}^{k_0} 2^{-k_1-k_2}  |W_1^{k_1}|^\frac{1}{2} |W_2^{k_2}|^\frac{1}{2}
	\leq& k_0 \left(\sum_{k_1=1}^{k_0}|W_1^{k_1}|2^{-2k_1}
								  \sum_{k_2=1}^{k_0}|W_2^{k_2}|2^{-2k_2} \right)^\frac{1}{2}	\\
	\lesssim& k_0 \left(\sum_{k_1=1}^{k_0}\sum_{w_1\in W_1^{k_1}}|c_{w_1}|^2
								 			\sum_{k_2=1}^{k_0}\sum_{w_2\in W_2^{k_2}}|c_{w_2}|^2 \right)^\frac{1}{2}	\\
	\lesssim& k_0 \|f_1\|_2 \|f_2\|_2 = k_0,								 			
\end{align*}
and thus
\begin{align}\label{12sep1545}
	&\| \sum_{k_1,k_2=1}^{k_0}\sum_{w_1\in W_1^{k_1}}\sum_{w_2\in W_1^{k_2}} c_{w_1}p_{w_1} c_{w_2}p_{w_2} \|_{L^p(Q_{S_1,S_2}(R))}		\nonumber\\
	&\hskip0.5cm\lesssim C_\alpha R^\alpha \log^{\gamma_\alpha+5}(1+R)
			(\kappa^{(1)}\kappa^{(2)})^{\frac{1}{2}-\frac{1}{p}}D^{3-\frac{5}{p}} \log^{\gamma_\alpha+5}(C_0).
\end{align}
Combining \eqref{12sep1543} -- \eqref{12sep1545}, we find that
\begin{align*}
	\|R^*_{H_1}f_1 R^*_{H_2}f_2\|_{L^p(Q_{S_1,S_2}(R))}
	=& \| \prod_{j=1,2} \sum_{w_j\in\W_j}c_{w_j}p_{w_j}\|_{L^p(Q_{S_1,S_2}(R))}		\\
	\lesssim& C_\alpha R^\alpha \log^{\gamma_\alpha+5}(1+R)
			(\kappa^{(1)}\kappa^{(2)})^{\frac{1}{2}-\frac{1}{p}}D^{3-\frac{5}{p}} \log^{\gamma_\alpha+5}(C_0),
\end{align*}
which verifies $E(\alpha).$
\end{proof}

\subsection{Bilinear estimates for sums of wave packets}

Let $v_j\in\V_j$, $j=1,2$, and define the ($\landau(1/R')$ thickened) ``intersection'' of  the transversal hypersurfaces $S_1$ and $S_2$ by
\begin{align*}
	\Pi_{v_1,v_2}= (v_1+S_2)\cap(v_2+S_1) +\landau(R'^{-1}).
\end{align*}
For any subset $W_j\subset\W_j$, let
\begin{align*}
	W_j^{\Pi_{v_1,v_2}}=\{w_j'\in W_j: v'_j+v_{j+1}\in\Pi_{v_1,v_2}\}
\end{align*}
(where $j+1$ is to  be interpreted mod 2 as before, i.e., we shall use the short hand notation $j+1=j+1\mod 2$ in the sequel whenever $j+1$ appears as an index), and  denote by
\begin{align*}
	V_j=\{v_j'\in\V_j:(y'_j,v'_j)\in W_j\ \mbox{ for some } y'_j\in \mathcal Y\}
\end{align*}
 the $\V$- projection of $W_j$. Further let
\begin{align*}
	 V_j^{\Pi_{v_1,v_2}}=&\{v_j'\in\V_j: (y'_j,v'_j)\in W_j^{\Pi_{v_1,v_2}} \mbox{  for some }y_j'\in\Y_j\}	 \nonumber\\
	=&\{v_j'\in\V_j:\mbox{  there is some } y_j'\in\Y_j\mbox{ s.t. } (y'_j,v'_j)\in W_j\text{ and }v_j'+v_{j+1}\in\Pi_{v_1,v_2} \}.
\end{align*}



\begin{lemnr}\label{L1L2}
Let $W_j\subset\W_j$, $j=1,2$. Then
\begin{align}
	\big\| \sum_{w_1\in W_1}\sum_{w_2\in W_2}p_{w_1}p_{w_2}\big\|_{L^1(Q_{S_1,S_2}(R))}
	\leq& \frac{R'^2}{\sqrt{\kappa^{(1)}\kappa^{(2)}}} |W_1|^\frac{1}{2}|W_2|^\frac{1}{2}	
	\label{L1} \\
	\big\| \sum_{w_1\in W_1}\sum_{w_2\in W_2}p_{w_1}p_{w_2}\big\|_{L^2(Q_{S_1,S_2}(R))}
	\lesssim& R'^{-\frac{1}{2}} \min\limits_j\sup\limits_{v_1,v_2}
				 |V_j^{\Pi_{v_1,v_2}}|^\frac{1}{2} |W_1|^\frac{1}{2}|W_2|^\frac{1}{2}.	\label{L2b}
\end{align}
\end{lemnr}

\begin{proof}
We shall closely  follow the arguments in  \cite{LV}, in particular the proof of Lemma 2.2, with only slight modifications.

The first estimate is easy. Using Hölder's inequality, we see that
\begin{align*}
	\big\|\sum_{w_1\in W_1}\sum_{w_2\in W_2}p_{w_1}p_{w_2}\big\|_{L^1(Q_{S_1,S_2}(R))}
	&\leq \prod_{j=1,2} \|\sum_{w_j\in W_j} p_{w_j}\|_{L^2(Q_{S_1,S_2}(R))}	\\
	&\hskip-1.5cm\leq \prod_{j=1,2} \left(\int_{-R'^2/\kappa^{(j)}}^{R'^2/\kappa^{(j)}}
											\|\sum_{w_j\in W_j} p_{w_j}(\cdot+tn_j)\|_{L^2(H_j)}^2d t \right)^\frac{1}{2}	\\
	&\hskip-1.5cm \lesssim \prod_{j=1,2} \frac{R'}{\sqrt{\kappa^{(j)}} }
											|W_j|^\frac{1}{2},								\end{align*}
where we have used (P4) in the last estimate.
The second one is more involved.
We write
\begin{align*}
	&\big\| \sum_{w_1\in W_1}\sum_{w_2\in W_2}p_{w_1}p_{w_2}\big\|_{L^2(Q_{S_1,S_2}(R))}^2\\
	&= \sum_{w_1\in W_1}\sum_{w_2\in W_2}\sum_{v'_1\in V_1}\sum_{v'_2\in V_2}
		\Big\langle p_{w_1}\sum_{y'_2\in Y_2(v'_2)}p_{w'_2}\,,\, p_{w_2}\sum_{y'_1\in Y_1(v'_1)}p_{w'_1} \Big\rangle,
\end{align*}
where  $Y_j(v'_j)=\{y\in\Y_j|(y,v'_j)\in W_j\}$ (recall that  $V_j$ is $\V$- projection of $W_j$). Since for $j=1,2$ the Fourier transform of $ \sum\limits_{y'_{j+1}\in Y_{j+1}(v'_{j+1})}p_{w'_{j+1}}p_{w_j}$ is supported in a ball of radius $\landau(R'^{-1})$ centered at $v'_{j+1}+v_j$, we may assume that the intersection of these two balls  is non-empty,  and thus
$$
	v'_1+v_2=v'_2+v_1+\landau(R'^{-1}).
$$
Especially
$$
	v'_{j+1}+v_j \in \Pi_{v_1,v_2}
$$
and
$$
	v'_j\in V_j^{\Pi_{v_1,v_2}},\quad j=1,2.
$$
This implies that
\begin{align*}
	&\big\| \sum_{w_1\in W_1}\sum_{w_2\in W_2}p_{w_1}p_{w_2}\big\|_{L^2(Q_{S_1,S_2}(R))}^2	\\
	&\leq \sum_{w_1\in W_1}\sum_{w_2\in W_2}	
				\sum_{v'_1\in V_1^{\Pi_{v_1,v_2}}} \sum_{\underset{v'_2=v'_1+v_2-v_1+\landau(R'^{-1})}{v'_2}}
		\int_{\R^3} |p_{w_1}p_{w_2}|\, dx\\
				&\hskip8cm\times\Big\|\sum_{y'_2\in Y(v'_2)}p_{w'_2}\Big\|_\infty
				 \Big\|\sum_{y'_1\in Y(v'_1)}p_{w'_1}\Big\|_\infty.
\end{align*}
Observe that there are at most $\landau(1)$ possible choices for $v'_2$ such that $$v'_2=v'_1+v_2-v_1+\landau(R'^{-1}).$$ Since the wave packets $p_{w_j}$ are essentially supported in the tubes $T_{w_j}$, which are well separated with respect to the parameter $y,$  the sum in $y'_j$ can be replaced by the  supremum, up to  some multiplicative constant. Since $T_{w_1}$ and $T_{w_2}$ satisfy the transversality condition \eqref{transversality1}, $p_{w_1}p_{w_2}$ decays rapidly away from the intersection $T_{w_1}\cap T_{w_2}$, i.e.,
\begin{align*}
	\int_{\R^3} |p_{w_1}p_{w_2}| d x
	\lesssim \int_{\R^3} R'^{-2} \left(1+\frac{|x|}{R'}\right)^{-N}  d x
	= R' \int_{\R^3} (1+|x|)^{-N}  d x \sim  R'.
\end{align*}
We  thus obtain
\begin{align}\label{L2detail}
	\| \sum_{\overset{w_1\in W_1}{w_2\in W_2}}p_{w_1}p_{w_2}\|_{L^2(Q_{S_1,S_2}(R))}^2
	\lesssim& R' |W_1|\ |W_2|\ \sup_{v_1,v_2}|V_1^{\Pi_{v_1,v_2}}|
		\prod_{j=1,2}\sup\limits_{w'_j\in W_j} \|p_{w'_j}\|_\infty	\\
	\lesssim& R'^{-1} |W_1|\ |W_2|\ \sup_{v_1,v_2}|V_1^{\Pi_{v_1,v_2}}|	\nonumber.
\end{align}
Repeating the same computation with the roles of $v'_1$ and $v'_2$ interchanged gives \eqref{L2b}.
\end{proof}
~\\

\subsection{Basis of the induction on scales argument}
In order to start our induction on scales, we need to establish a base case estimate which will respect the form of our estimate  \eqref{E(a)reduction_eq}. This will require a somewhat  more sophisticated approach than what is done  usually, based on the following
\begin{lemnr}\label{Vbound}
Let $V_j\subset\V_j$. Then $\min\limits_j\sup\limits_{v_1\in\V_1,v_2\in\V_2} |V_j^{\Pi_{v_1,v_2}}|\lesssim R$.
\end{lemnr}

\begin{proof}
Define the graph mapping $\Phi:U_1\cup U_2\to S_1\cup S_2,\Phi(x)=(x,\phi(x))$.
If $v'_j=\Phi(x'_j)\in V_j^{\Pi_{v_1,v_2}}$, then
$v'_j+v_{j+1}\in \Pi_{v_1,v_2},$
and for $x_{j+1}=\Phi^{-1}(v_{j+1})$ we have
$x'_j+x_{j+1}\in \gamma(I)+\landau(R'^{-1})$, where $\gamma:I\to [U_1+x_2]\cap[U_2+x_1]\subset\R^2$ is a parametrization by
arclength of the projection to the  $(x_1,x_2)$-space of the intersection curve $\Pi_{v_1,v_2}.$ Recall from Lemma \ref{propertylemma} (v) that our assumptions imply that then  $\gamma$ will be  close to a diagonal, i.e.,  $|\dot{\gamma}_i|\sim  1$, $i=1,2$.

For all $t,t'\in I$, we have $\gamma(t),\gamma(t')\in[U_1+x_2]\cap[U_2+x_1]$, hence
\begin{align*}
	\min\limits_j d_i^{(j)}
	\geq |\gamma_i(t)-\gamma_i(t')|
	\geq \min\limits_{t''\in I}|\dot{\gamma_i}(t'')||t-t'|\sim  |t-t'|.
\end{align*}
This implies
$|I|=\sup\limits_{t,t'\in I} |t-t'| \lesssim \min\limits_{i,j} d_i^{(j)}=D,$  hence $L(\gamma) \lesssim D,$
and thus 
\begin{align*}
	|V_j^{\Pi_{v_1,v_2}}|&\sim
		|\Phi^{-1}(V_j^{\Pi_{v_1,v_2}})|	\\
	&\leq |\{x'_j\in\Phi^{-1}(\V_j): x'_j\in \gamma(I)-x_{j+1}+\landau(R'^{-1}) \}|	\\
	&\lesssim  L(\gamma)/(R'^{-1})	\\
	&\lesssim D R' =R,
\end{align*}
since $\Phi^{-1}(\V_j)$ is an $R'^{-1}$-grid in $U_j$.
\end{proof}


\begin{kornr}\label{inductionstart}
E(1) holds true, provided ${4}/{3}\leq p\leq 2.$
\end{kornr}
\begin{proof}
Due to Lemma \ref{E(a)reduction}, it is enough to show the corresponding estimate for wave packets
\eqref{E(a)reduction_eq} with $\alpha=1$.
But, estimating $|V_j^{\Pi_{v_1,v_2}}|$ on the right-hand side of  \eqref{L2b} in  Lemma \ref{L1L2} by means of  Lemma  \ref{Vbound}, we obtain
\begin{align*}
	\big\|\prod_{j=1,2}\sum_{w_j\in W_j^{\lambda_j,\mu}} p_{w_j} \sum_{q\in Q^\mu} \chi_q\big\|_{L^2(Q_{S_1,S_2}(R))}
	\lesssim& R'^{-\frac{1}{2}} R^\frac{1}{2} |W_1|^\frac{1}{2}|W_1|^\frac{1}{2}.
\end{align*}
Interpolating this with the corresponding $L^1$-estimate that we obtain from \eqref{L1}, we arrive at
\begin{align*}
	\big\|\prod_{j=1,2}\sum_{w_j\in W_j^{\lambda_j,\mu}} p_{w_j} \sum_{q\in Q^\mu} \chi_q\big\|_{L^p(Q_{S_1,S_2}(R))}
	\lesssim& (\kappa^{(1)}\kappa^{(2)})^{\frac{1}{2}-\frac{1}{p}} R'^{\frac{5}{p}-3}
				 R^{1-\frac{1}{p}} |W_1|^\frac{1}{2}|W_1|^\frac{1}{2}	\\
	\leq&(\kappa^{(1)}\kappa^{(2)})^{\frac{1}{2}-\frac{1}{p}}D^{3-\frac{5}{p}}R
	 |W_1|^\frac{1}{2}|W_1|^\frac{1}{2},
\end{align*}
provided $\frac{4}{3}\leq p\leq 2.$
\end{proof}

\subsection{Further decompositions}

In a next step,  by some slight modification  of  the usual approach, we  introduce a further decomposition of the cuboid $Q_{S_1,S_2}(R)$ defined in \eqref{largecuboid} into smaller cuboids $b$ whose dimensions are those of  $Q_{S_1,S_2}(R)$ shrunk by a factor $R^{-2\delta},$  i.e., all of the $b$'s will be   translates of $Q_{S_1,S_2}(R^{1-\delta}).$  Here, $\delta>0$ is a  sufficiently small parameter to be chosen later. Since
$$\frac{R'^2}{\bar\kappa}R^{-2\delta} = \frac{R^{1-2\delta}R'}{D\bar\kappa} \geq R^{1-2\delta} R',
$$
the smallest side length of $b$ is still much larger than the side length $R^{\delta}R'$ of the thickened cubes $R^\delta q$ introduced at the end of Section \ref{generalbilinear}. Observe further that the number of  cuboids $b$ into which $Q_{S_1,S_2}(R)$  will be decomposed is of the order $R^{c\delta}.$ \footnote{Here and in the subsequent   considerations, $c$ will  denote some  constant which is independent of $R$ and $S_1,S_2,$ but whose precise value may vary  from line to line.}
\medskip

If $\mu\in \mathcal N^2$ is a fixed pair of dyadic numbers, and if  $w_j\in W_j,$  then we assign  to $w_j$ a cuboid  $b(w_j)$  in such a way that $b(w_j)$
contains a maximal number of $q$'s from $Q^\mu(w_j)$ among all the cuboids $b.$  We say that $b\sim w_j,$ if $b$ is contained in $10b(w_j)$ (the cuboid having the same center as  $b(w_j)$ but scaled by a factor $10$).
Notice that if  $b\not\sim w_j$, then this does not  necessarily mean that there are  only few  cubes
$q\in Q^\mu(w_j)$ contained in $b$ (since the cuboid  $b(w_j)$ may not be unique), but it does imply  that there are many cubes $q$ lying ``away'' from $b$. To be more  precise, if  $b\not\sim w_j,$ then
\begin{align}\label{bconcentration}
	|\{q\in Q^\mu(w_j)| q\cap 5b =\emptyset\}|
	\geq |\{q\in Q^\mu(w_j)| q\subset b(w_j)\}|
	\gtrsim R^{-c\delta} |Q^\mu(w_j)|,
\end{align}
since only $\landau(R^{2\delta})$ cuboids $b$ meet $T_{w_j}.$

For a fixed $b$, we can decompose any given set $W_j\subset \W_j$ into $W_j^{\not\sim b}=\{w_j\in W_j:b\not\sim w_j\}$ and $W_j^{\sim b}=\{w_j\in W_j:b\sim w_j\}$. Thus we have
\begin{align}\label{IandII}
	\big\|\prod_{j=1,2}\sum_{w_j\in W_j^{\lambda_j,\mu}} p_{w_j} \sum_{q\in Q^\mu} \chi_q\big\|_{L^p(Q_{S_1,S_2}(R))}
	\leq& \sum_b \big\|\prod_{j=1,2}\sum_{w_j\in W_j^{\lambda_j,\mu}} p_{w_j} \sum_{q\in Q^\mu} \chi_q\big\|_{L^p(b)}	\\
	=& I + II + III	\nonumber,																				
\end{align}
where
\begin{align*}
I =& \sum_b \big\|\prod_{j=1,2}\sum_{w_j\in W_j^{\lambda_j,\mu,\sim b}} p_{w_j} \sum_{q\in Q^\mu} \chi_q\big\|_{L^p(b)},	\\
II =& \sum_b\big \|\sum_{w_1\in W_1^{\lambda_1,\mu,\not\sim b}} p_{w_1}	
								\sum_{w_2\in W_2^{\lambda_2,\mu}} p_{w_2} \sum_{q\in Q^\mu}\chi_q\big\|_{L^p(b)}, 	\\
	III =& \sum_b\big \|\sum_{w_1\in W_1^{\lambda_1,\mu,\sim b}} p_{w_1}
						\sum_{w_2\in W_2^{\lambda_2,\mu,\not\sim b}} p_{w_2} \sum_{q\in Q^\mu}\chi_q\big\|_{L^p(b)}. 	
\end{align*}

As usual in the bilinear approach, Part I, which comprises the terms of highest density of wave packets over the cuboids $b,$  will  be handled by means of  an inductive argument. 
 The treatment of Part II  (and analogously of Part III) will be based on a combination of geometric and combinatorial arguments.   It is only here that  the very  choice of the  $b(w_j)$ will become crucial.


\begin{lemnr}\label{estofI}
Let $\alpha>0$, and assume  that $E(\alpha)$ holds true. Then
\begin{align*}
	I \leq C_\alpha R^{\alpha(1-\delta)} \log^{\gamma_\alpha}(1+R)
	(\kappa^1\kappa^2)^{\frac{1}{2}-\frac{1}{p}}D^{3-\frac{5}{p}} \log^{\gamma_\alpha}(C_0(S))
			|W_1|^\frac{1}{2} |W_2|^\frac{1}{2}.
\end{align*}
\end{lemnr}

\begin{proof}
To shorten notation, write $C_1= C_\alpha(\kappa^1\kappa^2)^{\frac{1}{2}-\frac{1}{p}} D^{3-\frac{5}{p}} \log^{\gamma_\alpha}(C_0(S))$. Recall the reproducing  formula (P1) in Corollary \ref{wave packets2}: $p_{w_j}=R^*_{H_j}(\FT_{H_j}(p_{w_j}|_{H_j})).$
Since  every cuboid $b$ is  a translate of $Q_{S_1,S_2}(R^{1-\delta}),$  and since  a translation of $R^*_{H_j}g$  corresponds  to a modulation of the function  $g,$ we see that  $E(\alpha)$ implies
\begin{eqnarray*}
I &=& \sum_b \big\|\prod_{j=1,2}\sum_{w_j\in W_j^{\lambda_j,\mu,\sim b}} p_{w_j} \sum_{q\in Q^\mu}  \chi_q\big\|_{L^p(b)}	\\
&\leq& \sum_b \big\|\prod_{j=1,2} R^*_{H_j}(\sum_{w_j\in W_j^{\lambda_j,\mu,\sim b}}  \FT_{H_j}(p_{w_j}|_{H_j} )\big\|_{L^p(b)}	\\
&\leq& C_1 (R^{1-\delta})^\alpha	\log^{\gamma_\alpha}(1+R^{1-\delta})
   \sum_b\prod_{j=1,2}\|\sum_{w_j\in W_j^{\lambda_j,\mu,\sim b}} \FT_{H_j}(p_{w_j}|_{H_j})\|_{L^2(H_j)}\\
  &\leq& C_1 R^{\alpha(1-\delta)}	\log^{\gamma_\alpha}(1+R)\sum_b \prod_{j=1,2} |W_j^{\lambda_j,\mu,\sim b}|^\frac{1}{2}.
\end{eqnarray*}
In the last estimate, we have made use of property (P4).
Moreover, using H\"older's inequality, we obtain
\begin{eqnarray*}
	\sum_b \prod_{j=1,2} |W_j^{\lambda_j,\mu,\sim b}|^\frac{1}{2}
	\leq \prod_{j=1,2} \left( \sum_b  |W_j^{\lambda_j,\mu,\sim b}| \right)^\frac{1}{2},
\end{eqnarray*}
where, due to Fubini's theorem (for sums),
\begin{eqnarray*}
	\sum_b  |W_j^{\lambda_j,\mu,\sim b}|
	= \sum_b |\{w_j\in W_j^{\lambda_j,\mu}:w_j\sim b\}|	
	= \sum_{w_j\in W_j^{\lambda_j,\mu}} |\{b:\ b\sim w_j\}|	
	\lesssim |W_j|.
\end{eqnarray*}
In combination, these estimates yield
\begin{eqnarray*}
	I \leq C_1 R^{\alpha(1-\delta)}	\log^{\gamma_\alpha}(1+R)
					\prod_{j=1,2} |W_j|^\frac{1}{2}.
\end{eqnarray*}
\end{proof}

\subsection{The geometric argument}
We next turn to the estimation of $II$ and $III.$  A crucial tool will be
the following lemma, which  is a variation of Lemma 2.3 in \cite{LV}.

\begin{lemnr}\label{geometric}
Let $\lambda_j,\mu_j\in \mathcal N$, $W_j\subset\W_j$, $v_j\in\V_j$, $j=1,2$, and let $b$ and $q_0$ be cuboids from our collections such that $q_0\cap 2b\neq\emptyset$. If we define $W_j^{\lambda_j,\mu,\not\sim b}(q_0)=W_j^{\lambda_j,\mu,\not\sim b}\cap W_j(q_0)$, then
 \begin{itemize}
\item[(i)]  \quad $\lambda_1 \mu_2 \ \left|\left[W_1^{\lambda_1,\mu,\not\sim b}(q_0)\right]^{\Pi_{v_1,v_2}}\right|
						\lesssim R^{c\delta} |W_2|	$
\item[(ii)] \quad$\lambda_2 \mu_1 \ \left|\left[W_2^{\lambda_2,\mu,\not\sim b}(q_0)\right]^{\Pi_{v_1,v_2}}\right|
						\lesssim R^{c\delta} |W_1|.$
 \end{itemize}	
\end{lemnr}

\begin{figure}
\begin{center}  \includegraphics{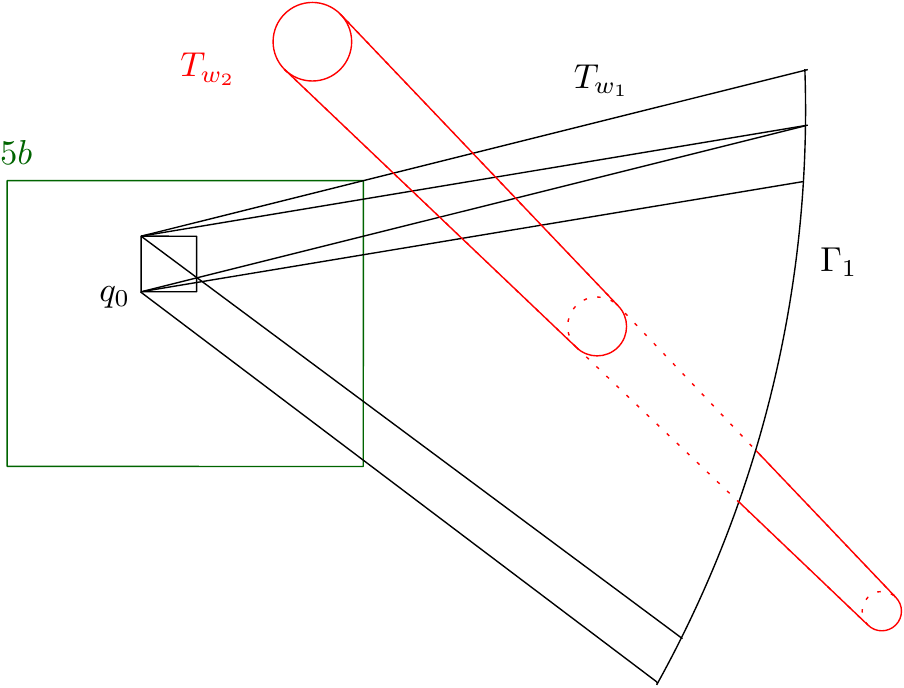} \end{center}
\caption{The geometry in Lemma \ref{geometric}}
\label{cone}       
\end{figure}

\begin{proof}
\newcommand{\ww}{\left[W_1^{\lambda_1,\mu,\not\sim b}(q_0)\right]^{\Pi_{v_1,v_2}}}
We only show (i), the proof of  (ii) being analogous. Set
$$
\Gamma_1 = \bigcup\Big\{T_{w_1}:w_1\in\ww\Big\}\setminus 5b,\quad Q^\mu_{\Gamma_1}=\{q\in Q^\mu| R^\delta q\cap \Gamma_1\neq\emptyset \}.
$$
 Since we have seen in Lemma \ref{propertylemma} (iv) that $T_{w_2}$ is transversal to $\Gamma_1,$  we have
\begin{align}\label{2ndtransversality}
|Q^\mu_{\Gamma_1}\cap Q^\mu(w_2)|\lesssim R^{c\delta}.
\end{align}
Due to the separation of the tube directions, the sets $T_{w_1}\setminus 5b$ do not overlap too much. To be more precise, we claim that for all cubes $q\in Q^\mu_{\Gamma_1}$
\begin{align}\label{16sep1744}
	\Big|\Big\{w_1\in\ww: R^\delta q\cap T_{w_1}\setminus5b\neq\emptyset\big\}\Big|\lesssim R^{c\delta}.
\end{align}
Indeed, let $w_1,w'_1\in\ww$ and $x\in R^\delta q\cap T_{w_1}\setminus5b$, $x'\in R^\delta q\cap T_{w'_1}\setminus5b$. The definition of $W_1(q_0)$ means that we find $x_0\in R^\delta q_0\cap T_{w_1}$ and $x'_0\in R^\delta q_0\cap T_{w'_1}$; then we may write
\begin{align}\label{16sep1730}
	x=x_0+|x-x_0|N(v_1) + \landau(R')\text{ and }x'=x'_0+|x'-x'_0|N_0(v'_1) + \landau(R').
\end{align}
Furthermore we have
\begin{align}\label{16sep1731}
	\big||x-x_0|-|x'-x'_0|\big|\leq|x-x'|+|x_0-x'_0| = \landau(R^{c\delta}R').
\end{align}
Since $T_{w_1}$ has length ${R'^2}/{\kappa^{(1)}}$, so that   the length of $b$ in the direction of $T_{w_1}$ is at least $R^{-2\delta}{R'^2}/{\kappa^{(1)}}$,   and since $x_0\in R^\delta q_0\subset 4b$ but $x\notin 5b$, we conclude that
\begin{align}\label{16sep1732}
	R^{-2\delta}\frac{R'^2}{\kappa^{(1)}} \leq |x-x_0|.
\end{align}
Applying Lemma \ref{separationlemma}, and making consecutively use of the estimates \eqref{16sep1732}, \eqref{16sep1731}, \eqref{16sep1730} and  again \eqref{16sep1731}, we obtain
\begin{eqnarray*}
	R'|v_1-v'_1|
	&\lesssim& \frac{R'}{\kappa^{(1)}}|N(v_1)-N(v'_1)| \\
	&\lesssim& R^{2\delta}R'^{-1} |x-x_0|\,|N_0(v_1)-N_0(v'_1)|	\\
	&\lesssim& R^{2\delta}R'^{-1} \big||x-x_0|N_0(v_1)-|x'-x'_0|N_0(v'_1)\big|+\landau(R^{c\delta})\\
	&\lesssim& R^{2\delta}R'^{-1}\big(|x-x'|+|x'_0-x_0|\big)+\landau(R^{c\delta})	\\
	&=&	\landau(R^{c\delta}).
\end{eqnarray*}
Recall also that the direction of a tube $T_{w_1}$ with $w_1=(y_1,v_1)$ depends only on $v_1,$  and thus the set of all these directions  corresponding to the set
$$\Big\{w_1\in\ww: R^\delta q\cap T_{w_1}\setminus5b\Big\}$$
has cardinality $\landau(R^{c\delta})$. But, for a fixed direction $v_1$, the number of  parameters $y_1$ such that the tube $T_{(y_1,v_1)}$ passes through $R^{c\delta}q_0$ is bounded by $\landau(R^{c\delta})$ anyway,  and thus \eqref{16sep1744} holds true.

Recall  next from \eqref{bconcentration} that  for $w_1\not\sim b$ we have
$
R^{-c\delta}|Q^\mu(w_1)|\lesssim  	|\{q\in Q^\mu(w_1): q\cap 5b =\emptyset\}|  .
$
Since  for $w_1\in W_1^{\lambda_1,\mu}$ we have  $|Q^\mu(w_1)|\sim\lambda_1,$ we may thus estimate
\begin{eqnarray*}
	& & R^{-c\delta}\lambda_1 \left|\ww\right| 		\\
	&\lesssim& R^{-c\delta} \sum_{w_1\in\ww} |Q^\mu(w_1)|		\\
&\lesssim& \sum_{w_1\in\ww} |\{q\in Q^\mu(w_1): q\cap 5b =\emptyset\}|	\\
	&\leq& \sum_{w_1\in\ww} |\{q\in Q^\mu: R^\delta q\cap T_{w_1}\neq\emptyset,\
														R^\delta q\cap 5b =\emptyset\}|	\\
	&\leq& \sum_{w_1\in\ww} |\{q\in Q^\mu: R^\delta q\cap (T_{w_1}\setminus 5b)\neq\emptyset\}|	\\
	&=& \sum_{q\in Q^\mu} \Big|\Big\{w_1\in\ww: R^\delta q\cap (T_{w_1}\setminus 5b)\neq\emptyset\Big\}\Big|			 \\
	&=&  R^{c\delta} |Q^\mu_{\Gamma_1}|,
\end{eqnarray*}
where we have used \eqref{16sep1744} in the last estimate.
But,  by \eqref{2ndtransversality}, we  also have
\begin{eqnarray*}
\mu_2	|Q^\mu_{\Gamma_1}|
	= \sum_{q\in Q^\mu_{\Gamma_1}} |W_2(q)|
\leq \sum_{w_2\in W_2} |Q^\mu_{\Gamma_1}\cap Q^\mu(w_2)|		
	\lesssim R^{c\delta} |W_2|,
\end{eqnarray*}
and combining this with the previous estimate we  arrive at  the desired estimate in (i).
\end{proof}


\begin{lemnr}\label{estofII}
Let $0<\delta<\frac{1}{4}$. Then
\begin{align}\nonumber
	II &= \sum_b \|\sum_{w_1\in W_1^{\lambda_1,\mu,\not\sim b}} p_{w_1}	
		\sum_{w_2\in W_2^{\lambda_2,\mu}} p_{w_2} \sum_{q\in Q^\mu}\chi_q\|_{L^p(b)}\\
	&\leq C_\alpha R^{c\delta} (\kappa^1\kappa^2)^{\frac{1}{2}-\frac{1}{p}}
	D^{3-\frac{5}{p}} |W_1|^\frac{1}{2} |W_2|^\frac{1}{2} \label{2feb1515}	
\end{align}
and
\begin{align}\nonumber
	III &= \sum_b \|\sum_{w_1\in W_1^{\lambda_1,\mu,\sim b}} p_{w_1}
	\sum_{w_2\in W_2^{\lambda_2,\mu,\not\sim b}} p_{w_2} \sum_{q\in Q^\mu}\chi_q\|_{L^p(b)}\\
	\leq& C_\alpha R^{c\delta} (\kappa^1\kappa^2)^{\frac{1}{2}-\frac{1}{p}}
				D^{3-\frac{5}{p}} |W_1|^\frac{1}{2} |W_2|^\frac{1}{2}. \label{2feb1515b}	
\end{align}
\end{lemnr}
\begin{proof}
We will only  prove the first inequality; the proof of second one works in a similar way.
Since the  number of  $b$'s  over which we sum in  \eqref{2feb1515}  is of the order   $R^{c\delta}$, it is enough to show that for every  fixed $b$
\begin{align}
	\|\sum_{w_1\in W_1^{\lambda_1,\mu,\not\sim b}} p_{w_1}
						\sum_{w_2\in W_2^{\lambda_2,\mu}} p_{w_2} \sum_{q\in Q^\mu}\chi_q\|_{L^p(b)}\le C_\alpha R^{c\delta} (\kappa^1\kappa^2)^{\frac{1}{2}-\frac{1}{p}}
				D^{3-\frac{5}{p}} |W_1|^\frac{1}{2} |W_2|^\frac{1}{2}.
\end{align}
For $p=1$, we apply \eqref{L1} from Lemma \ref{L1L2}:
\begin{align*}
	&\|\sum_{w_1\in W_1^{\lambda_1,\mu,\not\sim b}} p_{w_1}
				\sum_{w_2\in W_2^{\lambda_2,\mu}} p_{w_2} \sum_{q\in Q^\mu}\chi_q\|_{L^1(b)}	\\
	\lesssim& \|\sum_{w_1\in W_1^{\lambda_1,\mu,\not\sim b}} p_{w_1}
						\sum_{w_2\in W_2^{\lambda_2,\mu}} p_{w_2}\|_{L^1(Q_{S_1,S_2}(R))}	\\
	\leq& \frac{R'^2}{\sqrt{\kappa^{(1)}\kappa^{(2)}}} |W_1|^\frac{1}{2}|W_1|^\frac{1}{2}.										
\end{align*}
For $p=2$, we claim that
\begin{align}\label{L2claim}
	\|\sum_{w_1\in W_1^{\lambda_1,\mu,\not\sim b}} p_{w_1}
						\sum_{w_2\in W_2^{\lambda_2,\mu}} p_{w_2} \sum_{q\in Q^\mu}\chi_q\|_{L^2(b)}^2
	\lesssim C_\alpha R^{c\delta} R'^{-1} |W_1||W_2|.						
\end{align}
The  desired inequality \eqref{2feb1515} will then follow by means of  interpolation with the previous  $L^1$-estimate - notice here that $R^{\frac{5}{p}-3} \leq 1$ since $\frac{5}{3}\leq p$.
\medskip

To prove \eqref{L2claim}, recall that the side lengths of $b$ are of the form $(R'^2/\kappa^{(j)})R^{-2\delta} =
R'/(D\kappa^{(j)})R^{1-2\delta}\geq R'R^{1-2\delta}$, $j\in\{1,2\}$. If $q\cap 2b=\emptyset$, then for $x\in b$ we have $|x-c_q|\geq \inf\limits_{y\notin 2b}|x-y|=\mathop{d}(x,(2b)^c)\geq R'R^{1-2\delta}$.
Therefore for every $x\in b$
\begin{align}\nonumber
	\big|\sum_{q\in Q^\mu,q\cap 2b=\emptyset} \chi_q(x)\big|
	\leq& C_N\underset{2^l\geq R^{1-2\delta}}{\sum_{l\in\N}}\ \underset{|x-c_q|\sim R'2^l}{\sum_{q\in Q^\mu}}
							\left(1+\frac{|x-c_q|}{R'}\right)^{-N-2}	\\
	\lesssim& C_N\underset{2^l\geq R^{1-2\delta}}{\sum_{l\in\N}} |\{q:|x-c_q|\sim R'2^l\}|\, 2^{-(N+2)l}
	 \nonumber \\
	\sim 	&C_N \underset{2^l\geq R^{1-2\delta}}{\sum_{l\in\N}} 2^{-Nl}	
	\sim  C_N R^{-(1-2\delta)N}			
	= C_{\delta,N'} R^{-N'}.	\label{sumoverq}
\end{align}
The last step requires that $\delta<1/2.$
Choosing $N$ sufficiently large, we see that by Lemma \ref{L1L2} and Lemma \ref{Vbound}
\begin{align*}
	&\Big\|\sum_{w_1\in W_1^{\lambda_1,\mu,\not\sim b}} p_{w_1}
						\sum_{w_2\in W_2^{\lambda_2,\mu}} p_{w_2}
						\sum_{q\in Q^\mu,q\cap 2b=\emptyset}\chi_q\Big\|_{L^2(b)}^2		\\	
	\lesssim& \Big\|\sum_{w_1\in W_1^{\lambda_1,\mu,\not\sim b}} p_{w_1}
						\sum_{w_2\in W_2^{\lambda_2,\mu}} p_{w_2} \Big\|_{L^2}^2		\,\cdot\,
						\Big\| \sum_{q\in Q^\mu,q\cap 2b=\emptyset}\chi_q \Big\|_{L^\infty(b)}^2	\\
	\lesssim& C_{\delta,N'} R'^{-1} |W_1|\,|W_2| \min\limits_j\sup\limits_{v_1,v_2}|V_j^{\Pi_{v_1,v_2}}| R^{-2N'} \\
	\lesssim& C_{\delta,N'} R'^{-1} |W_1|\,|W_2| R^{1-2N'}	\\
	\lesssim& C'_{\delta,N''} R'^{-1} |W_1|\,|W_2| R^{-N''}.						
\end{align*}
Thus it is enough to consider the sum over the set $Q^\mu_b=\{q\in Q^\mu: q\cap 2b\neq\emptyset\}$. For fixed $w_1,w_2,$ we split this set into the subsets
$ Q^\mu_b(w_1,w_2)=Q^\mu_b\cap Q^\mu(w_1)\cap Q^\mu(w_2),$ $ Q^\mu_b\cap Q^\mu(w_1)\setminus Q^\mu(w_2)$ and
$$ Q^\mu_b\setminus Q^\mu(w_1)
		=\big(Q^\mu_b\cap Q^\mu(w_2)\setminus Q^\mu(w_1) \big)
			\cup \big(Q^\mu_b\setminus(Q^\mu(w_2)\cap Q^\mu(w_1))\big).
$$
Except for the  first set, the contributions by the other subsets can be treated in the same way,  since they are all special cases of the following situation:
\medskip

Let $Q_0=Q_0(w_1,w_2)\subset Q^\mu_b$ such that there exists an $j\in \{1,2\}$ with $R^\delta q\cap T_{w_j}=\emptyset$ for all $q\in Q_0$. Then
\begin{align}\label{notononeT}
	\|\sum_{w_1\in W_1^{\lambda_1,\mu,\not\sim b}} p_{w_1}
						\sum_{w_2\in W_2^{\lambda_2,\mu}} p_{w_2} \sum_{q\in Q_0}\chi_q\|_{L^2(b)}^2
	\lesssim C_\alpha R^{c\delta} R'^{-1} |W_1||W_2|.					
\end{align}
Notice that the right-hand side is just what we need  for \eqref{L2claim}.

For the proof of \eqref{notononeT}, assume without loss of generality that $j=1$. Let $q\in Q_0.$ Then $T_{w_1}\cap R^\delta q=\emptyset,$ and for all $x\in(R^\delta/2) q$ we have
$(R^\delta/2) R' \leq \dist(x,(R^\delta q)^c) \leq \dist(x,T_{w_1})$.
Thus for every $x\in Q_{S_1,S_2}(R)$, we have $\dist(x,T_{w_1})\geq (R^\delta/2) R'$ or $x\notin (R^\delta/2) q$.
In the first case, we have
\begin{align}\label{11sep1057}
	|p_{w_1}(x)|\leq& C_N R'^{-1}\left(1+\frac{\dist(x,T_{w_1})}{R'}\right)^{-2N}	\\
	\leq&  C'_N R'^{-1} R^{-\delta N}  \left(1+\frac{\dist(x,T_{w_1})}{R'}\right)^{-N}.	\nonumber
\end{align}
One the other hand, in the second case, where  $x\notin (R^\delta/2) q,$ we have $(R^\delta/2) R'\leq|x-c_q|$. Using the rapid decay of the Schwartz function $\phi$ we then see that
\begin{align}\label{11sep1058}
	|\chi_q(x)|=|\chi\left(\frac{x-c_q}{R'}\right)|
	\leq C_N \left(\frac{|x-c_q|}{R'}\right)^{-N}
	\leq C'_N R^{-\delta N}.
\end{align}
Applying an argument similar to the one used  in \eqref{sumoverq}, we even obtain
\begin{align*}
	|\sum_{q\in Q_0}\chi_q(x)|\leq C''_N R^{-\delta N}
\end{align*}
for all $x\notin(R^\delta/2) q$.
To summarize, we obtain for  that for every $x\in Q_{S_1,S_2}(R)$
\begin{align}\label{11sep59}
	\big|p_{w_1}\sum_{q\in Q_0(w_1,w_2)}\chi_q\big|(x)\leq C(N,\delta) R'^{-1} R^{-\delta N} \left(1+\frac{\dist(x,T_{w_1})}{R'}\right)^{-N}.
\end{align}
This means  that the expression $p_{w_1}\sum_{q\in Q_0(w_1,w_2)}\chi_q$ can not only be estimated in the same way  as the original wave packet $p_{w_1},$ but we even obtain an improved estimate because of an additional factor $R^{-\delta N}$.
If we replace $p_{w_1}$ by $p_{w_2}$ on the left-hand side, we obtain in a similar way just the standard wave packet  estimate
\begin{align}\label{11sep60}
\big|p_{w_2}\sum_{q\in Q_0(w_1,w_2)}\chi_q\big|(x)\lesssim \|p_{w_2}\|_\infty \lesssim R'^{-1}\left(1+\frac{\dist(x,T_{w_2})}{R'}\right)^{-N},
\end{align}
without an additional factor.

We can now finish the proof of \eqref{notononeT}, basically by following  the ideas of  the proof of  the estimate \eqref{L2b} in Lemma \ref{L1L2}. The crucial argument was the fact that the Fourier transform of $p_{w'_{j+1}}p_{w_j}$ is supported in $v'_{j+1}+v_j+\landau(R'^{-1})$.
Since $\supp\widehat{\chi_q}=\supp\hat\chi(R'\cdot)\subset B(0,R'^{-1})$, the Fourier support of
$p_{w_1}p_{w_2}\sum\limits_{q\in Q_0(w_1,w_2)}\chi_q$ remains essentially the same.  It is at this  point that we need  that the functions $\chi_q$ have  compact Fourier support.
The modified wave packets $p_{w_i}\sum\limits_{q\in Q_0(w_1,w_2)}\chi_q$ are still well separated with respect to the  parameter $y_i,$ for fixed direction $v_i,$ thanks to \eqref{11sep59} and \eqref{11sep60}.
Thus the argument from Lemma \ref{L1L2} applies, and by the analogue of
\eqref{L2detail} we obtain
\begin{eqnarray*}
	&&\big\| \underset{w_2\in W_2^{\lambda_2,\mu}}
								{\underset{w_1\in W_1^{\lambda_1,\mu,\not\sim b} }{\sum}}p_{w_1}
								p_{w_2}\sum_{q\in Q_0(w_1,w_2)}\chi_q\big\|_{L^2(b)}^2
	\lesssim R' |W_1|\ |W_2|\ \min_j\sup_{v_1,v_2}|V_j^{\Pi_{v_1,v_2}}| 	\\
	&&\times  \sup\limits_{w_1\in W_1,w'_2\in W_2}
						\|p_{w'_2}\sum_{q\in Q_0(w_1,w'_2)}\chi_q\|_\infty
		\sup\limits_{w'_1\in W_1,w_2\in W_2}
						\|p_{w'_1}\sum_{q\in Q_0(w'_1,w_2)}\chi_q\|_\infty	\\
	&\lesssim&  C_{\delta,N'}
					 R'^{-1} |W_1|\ |W_2|\ \min_j\sup_{v_1,v_2}|V_j^{\Pi_{v_1,v_2}}|  R^{-N'}	\\
	&\lesssim& C_{\delta,N'} R'^{-1} R^{1-N'} |W_1|\ |W_2|.
\end{eqnarray*}
In the second inequality, we have made use of
\eqref{11sep59} and \eqref{11sep60}, and  the last one is based on   Lemma \ref{Vbound}.
This concludes the proof of \eqref{notononeT}.

\medskip
What remains to be controlled are the contributions
 by the cubes $q$  from $Q^\mu_b(w_1,w_2)$. Notice that the kernel $K(q,q')=\chi_q(x)\chi_{q'}(x)$ satisfies Schur's test condition
\begin{align*}
	\sup_{q}\sum_{q'} \chi_q(x)\chi_{q'}(x) \lesssim \sum_{q'}\chi_{q'}(x) \lesssim 1,
\end{align*}
with a constant not depending in $x.$
Let us put $f_q=\underset{w_2\in W_2^{\lambda_2,\mu}(q)}
	{\underset{w_1\in W_1^{\lambda_1,\mu,\not\sim b}(q) }{\sum}} p_{w_1}p_{w_2},$  and observe that for $w_1\in W_1^{\lambda_1,\mu,\not\sim b}$ and $w_2\in W_2^{\lambda_2,\mu},$ we have $q\in Q^\mu_b(w_1,w_2)$ if and only if $q\in Q^\mu_b$ and  $w_1\in W_1^{\lambda_1,\mu,\not\sim b}(q)$ and $w_2\in W_2^{\lambda_2,\mu}(q).$ Then we see that we may  estimate				
\begin{eqnarray*}
	&&\|\underset{w_2\in W_2^{\lambda_2,\mu}}
								{\underset{w_1\in W_1^{\lambda_1,\mu,\not\sim b} }{\sum}} p_{w_1}p_{w_2}
								\sum_{q\in Q^\mu_b(w_1,w_2)}\chi_q\|_{L^2(b)}^2		
	= \|\sum_{q\in Q^\mu_b}\chi_q f_q \|_{L^2(b)}^2			\\
	&=& \int_b \big|\sum_{q,q'\in Q^\mu_b}\chi_q\chi_{q'}f_q f_{q'}\big | d x	
	= \int_b \big|\sum_{q,q'\in Q^\mu_b} K(q,q')f_{q} f_{q'}\big| d x  	\\
	&\lesssim & \int_b \sum_{q\in Q^\mu_b} |f_q|^2 \, d x					
	= \sum_{q\in Q^\mu,q\cap2b\neq\emptyset} \|\underset{w_2\in W_2^{\lambda_2,\mu}(q)}
						{\underset{w_1\in W_1^{\lambda_1,\mu,\not\sim b}(q)}{\sum}} p_{w_1}p_{w_2} \|_{L^2(b)}^2.
\end{eqnarray*}
			Invoking also Lemma \ref{L1L2} and Lemma \ref{geometric} (i), we thus obtain
\begin{eqnarray*}		
&&\|\underset{w_2\in W_2^{\lambda_2,\mu}}
								{\underset{w_1\in W_1^{\lambda_1,\mu,\not\sim b} }{\sum}} p_{w_1}p_{w_2}
								\sum_{q\in Q^\mu_b(w_1,w_2)}\chi_q\|_{L^2(b)}^2\\		
	&\lesssim& \sum_{q\in Q^\mu,q\cap2b\neq\emptyset} R'^{-1}
				| W_1^{\lambda_1,\mu,\not\sim b}(q)|\  |W_2^{\lambda_2,\mu}(q)|
			\sup_{v_1,v_2}\big|[W_1^{\lambda_1,\mu,\not\sim b}(q)]^{\Pi_{v_1,v_2}}\big|		\\
	&\lesssim&  	R^{c\delta}	R'^{-1}   \sum_{q\in Q^\mu,q\cap2b\neq\emptyset}
								| W_1^{\lambda_1,\mu}(q)|\ |W_2(q)| \frac{|W_2|}{\lambda_1\mu_2}	\\
	&\lesssim&  R^{c\delta}R'^{-1} \sum_{w_1\in W_1^{\lambda_1,\mu}}
													|Q^\mu(w_1)| \frac{|W_2|}{\lambda_1}	\\
	&\lesssim&  R^{c\delta}R'^{-1} |W_1|\ |W_2|.
\end{eqnarray*}
This completes the proof of estimate \eqref{2feb1515}, hence of Lemma \ref{estofII}.
\end{proof}

\subsection{Induction on scales}
We can now easily complete the proof of Theorem \ref{Eathm} by  following standard arguments.
\begin{kornr}\label{inductionstep}
There exist  constants $c,\delta_0>0$ such that $c\delta_0>1$ and such that the following holds true:

Whenever $\alpha>0$ is such that $E(\alpha)$ holds true, then $E(\max\{\alpha(1-\delta),c\delta\})$ holds true for every $\delta$ such that  $0<\delta<\delta_0.$
\end{kornr}
\begin{proof}
Let us put $\delta_0=1/4.$ Then the previous Lemmas \ref{estofI} and \ref{estofII} imply that
\begin{align*}
	&\big\|\prod_{j=1,2}\sum_{w_j\in W_j^{\lambda_j,\mu}} p_{w_j} \sum_{q\in Q^\mu} \chi_q\big\|_{L^p(Q_{S_1,S_2}(R))}
	\leq I+II+III	\\		
	\lesssim& \big(C_\alpha R^{\alpha(1-\delta)}\log^{\gamma_\alpha}(1+R)+ C_\delta R^{c\delta}\big)
			(\kappa^1\kappa^2)^{\frac{1}{2}-\frac{1}{p}}D^{3-\frac{5}{p}} \log^{\gamma_\alpha}(C_0)
			|W_1|^\frac{1}{2} |W_2|^\frac{1}{2}	\\
	\lesssim& C_{\alpha,\delta} R^{\alpha(1-\delta)\vee c\delta} \log^{\gamma_\alpha}(1+R)
			(\kappa^1\kappa^2)^{\frac{1}{2}-\frac{1}{p}}D^{3-\frac{5}{p}} \log^{\gamma_\alpha}(C_0)
			|W_1|^\frac{1}{2} |W_2|^\frac{1}{2}
\end{align*}
whenever $\delta<\delta_0,$ where $C_0=C_0(S)\gtrsim 1$  is defined in \eqref{C0}. By Lemma \ref{E(a)reduction}, this estimate implies  $E(\alpha(1-\delta)\vee c\delta).$
Finally, by  simply increasing the constant $c,$ if necessary, we may also ensure  that $c\delta_0>1.$
 \end{proof}


\begin{kornr}\label{offcubicBL}
 $E(\alpha)$ holds true  for every $\alpha>0.$
 \end{kornr}

 This completes also the proof of Theorem \ref{Eathm}.
\begin{proof}
Define inductively the sequence $\alpha_0=1$, $\alpha_{j+1}=c\alpha_j/(c+\alpha_j),$ which is decreasing and converges to $0.$ It therefore suffices to prove that  $E(\alpha_j)$  is valid for every  $j\in\N$.
But, by Corollary \ref{inductionstart},  $E(\alpha_0)=E(1)$ does hold  true.  Moreover,  Corollary \ref{inductionstep}  shows that $E(\alpha_j)$  implies  $E(\alpha_{j+1}),$  for if we choose  $\delta=\alpha_j/(c+\alpha_j),$ then $\delta<1/c<\delta_0$ and
  $\alpha_j(1-\delta)=c\delta=\alpha_j c/(c+\alpha_j)=\alpha_{j+1},$ and thus we may conclude  by induction.
\end{proof}


\section{Scaling}\label{Scaling}

For the proof of our main theorem, we shall have to perform a kind of Whitney type decomposition of $S\times S$ into  pairs of patches of hypersurfacec $(S_1,S_2)$ and prove very  precise  bilinear restriction estimates for those.  In  order to reduce these estimates to Theorem \ref {generalbilinear}, we shall need to  rescale  simultanously  the hypersurfaces $S_1,S_2$ for each such pair $(S_1,S_2)$ in a suitable way.
To this end, we shall denote here and in the sequel by $R^*_{S_1, S_2}$ the
 bilinear Fourier extension operator
$$
R^*_{S_1,S_2}(f_1,f_2)=R^*_{\R^2}f_1 \cdot R^*_{\R^2}f_2, \qquad f_1\in L^2(U_1), f_2\in L^2(U_2)
$$
associated to any pair of hypersurfaces $(S_1,S_2)$   given as the graphs $S_j=\{(\xi,\phi_j(\xi)):\xi\in U_j\}, j=1,2.$

The following trivial lemma comprises the effect of the type of rescaling that we shall  need.
\begin{lemnr}\label{scalinglemma}
	Let $S_j=\{(\xi,\phi(\xi):\xi\in U_j\}$, where again  $U_j\subset\R^d$ is open and bounded for  $j=1,2$. Let $A\in GL(d,\R)$, $a>0$,  put $\phi^s(\eta)=\frac{1}{a} \phi(A\eta),$ and let
$$
S_j^s=\{(\eta,\phi^s(\eta))|\eta\in U_j^s\}, \quad U_j^s=A^{-1}(U_j),\quad j=1,2.
$$
 For any   measurable subset  $Q^s\subset\R^{d+1}$, we set  $Q=\{x: (\trans Ax',ax_{d+1})\in Q^s \}$.
Assume that the following estimate holds true:
\begin{align*}
	\| R^*_{S^s_1,S^s_2}(g_1,g_2)  \|_{L^p(Q^s)} \leq C_s \|g_1\|_2 \|g_2\|_2
	\quad \text{   for all } g_j\in L^2(U^s_j).
\end{align*}
Then
\begin{align*}
	\| R^*_{S_1,S_2}(f_1,f_2) \|_{L^p(Q)}
	\leq C_s |\det A|^\frac{1}{p'} a^{-\frac{1}{p}} \|f_1\|_2 \|f_2\|_2
	\quad\text{ for all } f_j\in L^2(U_j).
\end{align*}
\end{lemnr}

\medskip

We now return to our model hypersurface (compare \eqref{model}, \eqref{psicond1} and \eqref{psicond1}), which is the graph of
\begin{align*}
\phi(\xi_1,\xi_2)= \psi_1(\xi_1)+ \psi_2(\xi_2)
\end{align*}
on $]0,1[\times ]0,1[,$ where the derivatives of the $\psi_i$ satisfy
\begin{align*}
\psi_i''(\xi_i)&\sim \xi_i^{m_i-2},\\
|\psi_i^{(k)}(\xi_i)|&\lesssim \xi_i^{m_i-k} \quad \mbox{for }\ k\ge 3,
\end{align*}
and where $m_1,m_2\in \R$ are  such that $m_i\ge 2.$
\medskip

We shall apply the preceding lemma   to pairs $S_1=S$ and $S_2=\tilde S$ of patches of this hypersurface  on which the  following assumptions are  met:
\medskip

\noindent\textsc{GENERAL ASSUMPTIONS:}
Let $S=\{(\xi,\phi(\xi)):\xi\in U\}$ and $\tilde S=\{(\xi,\phi(\xi)):\xi\in \tilde U\},$  where $U=r+[0,d_1]\times[0,d_2]$  and  $\tilde U=\tilde r+[0,\tilde d_1]\times[0,\tilde d_2],$  with $r=(r_1,r_2)$ and $\tilde r=(\tilde r_1,\tilde r_2).$

We assume that  for $i=1,2$ we have  $r_i\geq d_i$ and $\tilde r_i\geq \tilde d_i,$  so  that the  principal curvature $\psi_i''$ of $S$ with respect to $\xi_i$  is  comparable to $\kappa_i=r_i^{m_i-2}$,  and that of $\tilde S$ is comparable to $\tilde \kappa_i=\tilde r_i^{m_i-2}.$ We put

 \begin{align}
 &\bar d_i=d_i\vee\tilde d_i,\nonumber\\
 &\kappa=\kappa_1\vee\kappa_2, \quad \tilde \kappa=\tilde\kappa_1\vee\tilde\kappa_2,\nonumber\\
&\bar\kappa_i=\kappa_i\vee\tilde\kappa_i, \quad \bar r_i=r_i\vee\tilde r_i\label{quantities} \\
&\bar\kappa=\kappa\vee\tilde\kappa=\bar\kappa_1\vee\bar\kappa_2,\nonumber\\
&\Delta r_i=r_i - \tilde r_i.\nonumber
\end{align}
In addition, {\it we assume that for every  direction $\xi_i, i=1,2, $  the rectangle $U$ respectively $\tilde U$ on which the  corresponding principal curvature is bigger (which means that its projection to the $\xi_i$-axis is the one  further to the right), has also bigger  length in this direction.} This is easily seen to be equivalent to
\begin{align}\label{maxproduct}
	(\kappa_i d_i)\vee(\tilde\kappa_i\tilde d_i)=\bar\kappa_i\bar d_i.
\end{align}
Last, but not least, we assume the rectangles $U$ and $\tilde U$ are separated with respect to both variables $\xi_i, i=1,2,$ in the following sense:
\begin{align}\label{separationcond}
	{\rm dist}_{\xi_i}(U,\tilde U)=\inf\{|\xi_i-\tilde \xi_i|:\xi\in U,\tilde \xi\in\tilde U\}
	\sim  |\Delta r_i| \sim  \bar d_i.
\end{align}

\bigskip

Given these assumptions, we shall introduce a rescaling as follows: we put
\begin{align}\label{defa}
	a_1=\bar\kappa_2 \bar d_2,\qquad a_2=\bar\kappa_1 \bar d_1,
\end{align}
and
\begin{align}\label{defphiscaled}
	\phi^s(\eta)=\frac{1}{a}\phi(A\eta)
	=\frac{1}{a_1a_2}\phi(a_1\eta_1,a_2\eta_2).
\end{align}
The quantities that arise after this scaling will be denoted by a superscript $s,$  i.e.,
 \begin{align*} &r_i^s=\frac{r_i}{a_i},\quad  d_i^s=\frac{d_i}{a_i},\quad  D^s=\min\{d_1^s,d_2^s,\tilde d_1^s,\tilde d_2^s\},\\
 &\kappa_i^s=\frac{1}{a_1a_2}a_i^2\kappa_i=\frac{a_i}{a_{i+1\, {\rm mod} 2}}\kappa_i,\\
 &U^s=r^s+[0,d^s_1]\times[0,d^s_2],
  \end{align*}
  with corresponding expressions for $\tilde r^s, \tilde d_i^s, \tilde \kappa^s_i$ and $\tilde U^s.$ For later use, recall also  the normal field  $N$ on $S\cup\tilde S$ defined by $N(\xi,\phi(\xi))=(-\nabla\phi(\xi),1)$ and the coresponding
  unit normal field $N_0={N}/{|N|}.$ After scaling, the corresponing normal fields on $S^s\cup \tilde S^s$ will be denoted by  $N^s$ and
$N_0^s.$  With our  choice of scaling, the following lemma holds true:

\begin{lemnr}[Scaling]\label{scaling}
	\begin{enumerate}
		\item For  $i=1,2$ and all $\eta\in U^s$ and $\tilde\eta\in\tilde U^s$ we have
						$|\partial_i\phi^s(\eta)-\partial_i\phi^s(r^s)|\lesssim \kappa_i^s d_i^s \lesssim 1$ and
						$|\partial_i\phi^s(\tilde\eta)-\partial_i\phi^s(\tilde r^s)|
											\lesssim \tilde \kappa_i^s \tilde d_i^s \lesssim 1$.
Moreover, $\bar \kappa_i^s\bar d_i^s=1.$
		\item  $|\partial^\alpha\phi^s(\eta)|\lesssim\kappa^s
							 |d_1^s\wedge d_2^s|^{2-|\alpha|}$ \quad for every  $|\alpha|\geq2$, $\eta\in U^s$;\\
				$|\partial^\alpha\phi^s(\tilde\eta)|\lesssim\tilde\kappa^s
		|\tilde d_1^s\wedge\tilde d_2^s|^{2-|\alpha|}$\quad for every $|\alpha|\geq2$, $\tilde\eta\in\tilde U^s$.	
		\item For $i=1,2,$ i.e., with respect to  both variables, the  separation condition
\begin{align*}
|\partial_i\phi^s(\eta)-\partial_i\phi^s(\tilde\eta)|\sim  1\quad
\mbox{  for all } \eta\in S,\tilde\eta\in\tilde S.
\end{align*}
holds true.
\end{enumerate}
\noindent In particular, the rescaled pair of hypersurfaces $(S^s,\tilde S^s)$ satisfies the general assumptions (i) --(iii)  made in Theorem \ref{Eathm}.

\end{lemnr}

\begin{proof}

Observe first that
$$
\bar d_i^s=\frac{\bar d_i}{a_i}, \qquad\bar\kappa_i^s=\frac{1}{a_1a_2}a_i^2\bar\kappa_i,
$$
and thus, by  the definition of $a_i,$  we see that $\bar \kappa_i^s\bar d_i^s=1.$

Next,  in order to prove (i), observe that for $\eta\in U^s$
\begin{align*}
	|\partial_i\phi^s(\eta)-\partial_i\phi^s(r^s)|\leq \sup_{\eta'\in U} |\partial_i^2\phi^s(\eta')||\eta_i-r^s_i|
	\lesssim \kappa^s_i d^s_i,
\end{align*}
with
$\kappa_i^s d_i^s \le \bar \kappa_i^s\bar d_i^s=1.$

As for (ii), notice  that  also  $\partial_1\partial_2 \phi^s\equiv 0.$
In the unscaled situation, we have for
$k\ge 2$ and every  $\xi\in U$
 \begin{align*}
	|\partial_i^k \phi(\xi)| \lesssim  \xi_i^{m_i-k}
	\sim  \partial_i^2\phi(\xi) \xi_i^{2-k} \sim  \kappa_i \xi_i^{2-k}.
\end{align*}
 Thus, for $\eta\in U^s,$ we find that
\begin{align*}
	|\partial_i^k \phi^s(\eta)|
	&=\frac{1}{a_1a_2} a_i^k| \partial_i^k\phi(A\eta)	|
	\lesssim  \frac{1}{a_1a_2} a_i^k \kappa_i (a_i\eta_i)^{2-k}	\\
	&= \frac{a_i^2}{a_1a_2} \kappa_i \eta_i^{2-k}	
	= \kappa_i^s \eta_i^{2-k}.
\end{align*}
On the other hand, for  $\eta\in U^s$ we have
\begin{align*}
	\eta_i \geq r_i^s =\frac{r_i}{a_i} \geq \frac{d_i}{a_i}= d_i^s \geq d_1^s\wedge d_2^s,
\end{align*}
and thus we conclude that
\begin{align*}
	|\partial_i^k \phi^s(\eta)| \lesssim \kappa^s(d_1^s\wedge d_2^s)^{2-k},  \qquad k\ge 2.
\end{align*}
In the same way, we obtain the corresponding result for  $\eta\in\tilde U^s.$ These estimates imply (ii).

Finally, in order to prove (iii), let  $\xi=(\xi_1,\xi_2)\in U$ and $\tilde\xi=(\tilde\xi-1,\tilde\xi_2)\in\tilde U.$ Then, by \eqref{separationcond}, we see that $|\xi_i-\tilde \xi_i|\sim \bar d_i.$ Moreover, if for instance $r_i<\tilde r_i$  (the other case can be treated  analogously), then by \eqref{separationcond} we even have $r_i+d_i+c\bar d_i\le \tilde r_i,$ for some admissible constant $c>0$ such that $c<1.$ But then $\kappa_i\lesssim |\psi_i''(t)|\le\tilde\kappa_i$ for every $t$ in between $\xi_i$ and $\tilde \xi_i,$ and moreover $\psi_i''(t)\sim \tilde \kappa_i=\bar\kappa_1$ on the subinterval $[\tilde r_i- c\bar d_i/4, \tilde r_i],$ and thus
\begin{align*}
	|\partial_i\phi(\xi)-\partial_i\phi(\tilde\xi)|
	= \left| \int_{\xi_i}^{\tilde\xi_i} \psi_i''(t)d t \right|	
	\sim  \bar\kappa_i\bar d_i
	= a_{i+1 {\rm mod 2}},
\end{align*}
\color{black}
hence
\begin{align}\label{15jul}
	|\partial_i\phi^s(\eta)-\partial_i\phi^s(\tilde\eta)|
	= \frac{|\partial_i\phi(A\eta)-\partial_i\phi(A\tilde\eta)|}{a_{i+1{\rm mod2}}} \sim 1.
\end{align}
\end{proof}

In view of Lemma \ref{scaling}, we may now  apply Theorem \ref{Eathm} to the rescaled phase function $\phi^s.$ According to \eqref{Qscaled}, the scaled cuboids are  given by
$$Q^0_{S^s,\tilde S^s}(R)=\left\{x\in\R^3:~ |x_i+\partial_i\phi^s(r_0^s)x_3|\leq \frac{R^2}{({D^s})^2\bar\kappa^s},i=1,2,~|x_3|\leq\frac{R^2}{({D^s})^2(\kappa^s\wedge\tilde\kappa^s)} \right\},$$
with $r_0^s=r^s$ if $\kappa^s=\kappa^s\wedge\tilde\kappa^s,$  and  $r_0^s=\tilde r^s$ if $\tilde\kappa^s=\kappa^s\wedge\tilde\kappa^s$.
Thus, if $5/3\le p\le 2,$ then  for every $\alpha>0$ we obtain  the following   estimate, valid for every $R\geq1:$
\begin{align*}
 \| R^*_{S^s,\tilde S^s}\|_{L^2\times L^2\to L^p(Q^s_{S^s,\tilde S^s}(R))}
	\leq (\kappa^s\tilde\kappa^s)^{\frac{1}{2}-\frac{1}{p}}
				({D^s})^{3-\frac{5}{p}}
				\log^{\gamma_\alpha}(C_0^s)  C_\alpha R^\alpha,
\end{align*}
with (compare \eqref{C0})
$$
C_0^s=\frac{\bar {d^s_1}^2\bar {d^s_2}^2}{{(D^s)}^4} (D^s[\kappa^s\wedge\tilde\kappa^s])^{-\frac{1}{p}}
							(D^s\kappa^s D^s\tilde\kappa^s)^{-\frac{1}{2}}.
$$
Recall here that
$
R^*_{S,\tilde S}(f_1,f_2)=R^*_{\R^2}f_1 \cdot R^*_{\R^2}f_2, $ if $ f_1\in L^2(U), f_2\in L^2(\tilde U).
$
Scaling back by means of Lemma \ref{scalinglemma}, we obtain
\begin{align}\label{12aug1313}
 \| R^*_{S,\tilde S}\|_{L^2\times L^2\to L^p(Q_{S,\tilde S}(R))}
	\leq& (a_1a_2)^{1-\frac{2}{p}}(\kappa^s\tilde\kappa^s)^{\frac{1}{2}-\frac{1}{p}}
				({D^s})^{3-\frac{5}{p}} \log^{\gamma_\alpha}(C_0^s)  C_\alpha R^\alpha	\nonumber\\
	=& (a_1a_2\kappa^s\cdot a_1a_2\tilde\kappa^s)^{\frac{1}{2}-\frac{1}{p}}
				({D^s})^{3-\frac{5}{p}} \log^{\gamma_\alpha}(C_0^s)  C_\alpha R^\alpha,
\end{align}
where
\begin{align*}
	Q_{S,\tilde S}(R)=& \left\{x\in\R^3:~ |a_ix_i+\partial_i\phi^s(r_0^s)a_1a_2x_3|\leq
	 \frac{R^2}{({D^s})^2\bar\kappa^s},i=1,2,
	~|a_1a_2x_3|\leq\frac{R^2}{({D^s})^2\kappa^s\wedge\tilde\kappa^s} \right\}	\\
	=& \left\{x\in\R^3:~ |x_i+\partial_i\phi(r_0)x_3|\leq \frac{R^2}{a_i({D^s})^2\bar\kappa^s},i=1,2,
	~|x_3|\leq\frac{R^2}{a_1a_2({D^s})^2\kappa^s\wedge\tilde\kappa^s} \right\}.
\end{align*}
But, by \eqref{defa}, we have
\begin{align}\label{kappasquer}
	\bar\kappa^s = \bar\kappa_1^s\vee\bar\kappa_2^s
	= \frac{a_1}{a_2}\bar\kappa_1 \vee \frac{a_2}{a_1}\bar\kappa_2
	= \frac{a_1}{\bar d_1} \vee \frac{a_2}{\bar d_2}
	= \frac{\bar\kappa_2\bar d_2^2 \vee \bar\kappa_1\bar d_1^2}{\bar d_1\bar d_2}
\end{align}
and
\begin{align*}
	D^s=\min\{d_1^s,d_2^s,\tilde d_1^s,\tilde d_2^s\}
		\leq \min\left\{\frac{\bar d_1}{a_1},\frac{\bar d_2}{a_2}\right\}
		= (\bar\kappa^s)^{-1},
\end{align*}
hence
\begin{align*}
	a_i ({D^s})^2\bar\kappa^s
	\leq a_i D^s
	\leq \bar d_i,	\quad i=1,2,	
\end{align*}
and also
\begin{align*}
	a_1a_2({D^s})^2  (\kappa^s\wedge\tilde\kappa^s) \leq D^s a_1a_2
	 \leq a_2\bar d_1 \wedge a_1 \bar d_2
	 =  \bar\kappa_1\bar d_1^2 \wedge \bar\kappa_2\bar d_2^2.
\end{align*}
These estimates imply that
\begin{align}\label{cuboidrescaled}
	Q_{S,\tilde S}(R) \supset Q^1_{S,\tilde S}(R),
\end{align}
if we put
$$
Q^1_{S,\tilde S}(R)=
	\left\{x\in\R^3:~ |x_i+\partial_i\phi(r_0)x_3|\leq \frac{R^2}{\bar d_i},i=1,2,~|x_3|\leq\frac{R^2}
	{ \bar\kappa_1\bar d_1^2 \wedge \bar\kappa_2\bar d_2^2}\right\}.
$$
Moreover, by \eqref{defa} we have
$$
	d_i^s=\frac{d_i}{a_i} = \frac{\bar\kappa_i\bar d_i}{a_1a_2} d_i,
$$
and
$$
	\min\{d_i^s,\tilde d_i^s\}=\frac{\bar\kappa_i}{a_1a_2}\bar d_i\min\{d_i,\tilde d_i\}
	= \frac{\bar\kappa_i d_i\tilde d_i}{a_1a_2}.
$$
Furthermore,
\begin{align}\label{kappas}
	a_1a_2 \kappa^s \sim  a_1a_2\left(\frac{a_2}{a_1}\kappa_2 + \frac{a_1}{a_2}\kappa_1\right)
	= (\bar\kappa_1^2\bar d_1^2 \kappa_2 + \bar\kappa_2^2\bar d_2^2 \kappa_1)
	= \bar\kappa_1\bar\kappa_2
				\left(\bar\kappa_1\bar d_1^2 \frac{\kappa_2}{\bar\kappa_2}
				+ \bar\kappa_2\bar d_2^2 \frac{\kappa_1}{\bar\kappa_1}\right).
\end{align}
Thus the product of the  first two factors on the right-hand side of  \eqref{12aug1313} can be re-written
\begin{align*}
	&(a_1a_2\kappa^s\cdot a_1a_2\tilde\kappa^s)^{\frac{1}{2}-\frac{1}{p}}
				({D^s})^{3-\frac{5}{p}}	\\
	=& (a_1a_2)^{\frac{5}{p}-3}
	(\bar\kappa_1\bar\kappa_2)^{1-\frac{2}{p}}
			\left(\bar\kappa_1\bar d_1^2 \frac{\kappa_2}{\bar\kappa_2}
				+ \bar\kappa_2\bar d_2^2 \frac{\kappa_1}{\bar\kappa_1}\right)^{\frac{1}{2}-\frac{1}{p}}
			\left(\bar\kappa_1\bar d_1^2 \frac{\tilde\kappa_2}{\bar\kappa_2}
				+ \bar\kappa_2\bar d_2^2 \frac{\tilde\kappa_1}{\bar\kappa_1}\right)^{\frac{1}{2}-\frac{1}{p}}
		\min_i (\bar\kappa_i d_i \tilde d_i)^{3-\frac{5}{p}} \\	
	=& (\bar\kappa_1\bar d_1\bar\kappa_2\bar d_2)^{\frac{5}{p}-3}
	(\bar\kappa_1\bar\kappa_2)^{1-\frac{2}{p}}
			\left(\bar\kappa_1\bar d_1^2 \frac{\kappa_2}{\bar\kappa_2}
				+ \bar\kappa_2\bar d_2^2 \frac{\kappa_1}{\bar\kappa_1}\right)^{\frac{1}{2}-\frac{1}{p}}
			\left(\bar\kappa_1\bar d_1^2 \frac{\tilde\kappa_2}{\bar\kappa_2}
				+ \bar\kappa_2\bar d_2^2 \frac{\tilde\kappa_1}{\bar\kappa_1}\right)^{\frac{1}{2}-\frac{1}{p}}
				\min_i (\bar\kappa_i d_i \tilde d_i)^{3-\frac{5}{p}} \\
	=& (\bar\kappa_1\bar\kappa_2)^{\frac{3}{p}-2}(\bar d_1\bar d_2)^{\frac{5}{p}-3}	
			\left(\bar\kappa_1\bar d_1^2 \frac{\kappa_2}{\bar\kappa_2}
				+ \bar\kappa_2\bar d_2^2 \frac{\kappa_1}{\bar\kappa_1}\right)^{\frac{1}{2}-\frac{1}{p}}
			\left(\bar\kappa_1\bar d_1^2 \frac{\tilde\kappa_2}{\bar\kappa_2}
				+ \bar\kappa_2\bar d_2^2 \frac{\tilde\kappa_1}{\bar\kappa_1}\right)^{\frac{1}{2}-\frac{1}{p}}
				\min_i (\bar\kappa_i d_i \tilde d_i)^{3-\frac{5}{p}}.		
\end{align*}		

For $a,b\in(0,\infty)$ write $q(a,b)=\frac{a\vee b}{a\wedge b}=\frac{a}{b}\vee\frac{b}{a}\geq 1$.
A lower bound for $D^s$ is
\begin{align}\label{5feb1}
	D^s = \frac{d_1\wedge\tilde d_1}{a_1}\wedge\frac{d_2\wedge\tilde d_2}{a_2}
	\geq \left(\frac{d_1\wedge\tilde d_1}{\bar d_1}\wedge\frac{d_2\wedge\tilde d_2}{\bar d_2}\right)
				\left(\frac{\bar d_1}{a_1}\wedge\frac{\bar d_2}{a_2}\right)
	\geq \frac{1}{q(d_1,\tilde d_1)q(d_2,\tilde d_2)} \frac{1}{\bar\kappa^s},
\end{align}
where we have used \eqref{kappasquer} in the last inequality.
And, from formula \eqref{kappas} we can deduce
\begin{align}\label{5feb2}
	\kappa^s \gtrsim \frac{\bar\kappa_1\bar\kappa_2}{a_1a_2}
				\left(\bar\kappa_1\bar d_1^2 \vee \bar\kappa_2\bar d_2^2 \right)
				\frac{\kappa_1}{\bar\kappa_1}\frac{\kappa_2}{\bar\kappa_2}
	\geq  \frac{\bar\kappa_1\bar d_1^2 \vee \bar\kappa_2\bar d_2^2}{\bar d_1\bar d_2}
				\frac{1}{q(\kappa_1,\tilde\kappa_1)q(\kappa_2,\tilde\kappa_2)}
	= \frac{\bar\kappa^s}{q(\kappa_1,\tilde\kappa_1)q(\kappa_2,\tilde\kappa_2)},
\end{align}
where we have again applied  \eqref{kappasquer} in the last step.
Combining \eqref{5feb1} and \eqref{5feb2}, we obtain
\begin{align}\label{5feb3}
	(D^s \kappa^s)^{-1} \lesssim \prod_{i=1,2} q(\kappa_i,\tilde\kappa_i)q(d_i,\tilde d_i),
	\end{align}
and then by symmetry also
$$
	(D^s \tilde\kappa^s)^{-1} \lesssim \prod_{i=1,2} q(\kappa_i,\tilde\kappa_i)q(d_i,\tilde d_i).
$$
We may now estimate the constant  $C_0^s$ in the following way, using \eqref{5feb3} in the inequality,
 \eqref{5feb1} in the second one and \eqref{kappasquer} in the third one (being generous in the exponents, since $C_0^s$  appears only logarithmically):
\begin{eqnarray*}
	C_0^s
	&=&  \frac{\bar {d^s_1}^2\bar {d^s_2}^2}{{(D^s)}^4} (D^s[\kappa^s\wedge\tilde\kappa^s])^{-\frac{1}{p}}
							(D^s\kappa^s D^s\tilde\kappa^s)^{-\frac{1}{2}}	\\
	&\leq& \frac{\bar {d^s_1}^2\bar {d^s_2}^2}{{(D^s)}^4}
						\left[\prod_{i=1,2} q(\kappa_i,\tilde\kappa_i)q(d_i,\tilde d_i)\right]^{\frac{1}{p}+1}\\
	&\leq& \left[\prod_{i=1,2} q(\kappa_i,\tilde\kappa_i)q(d_i,\tilde d_i)\right]^{\frac{1}{p}+5}
				(\bar d^s_1\bar d^s_2)^2 (\bar{\kappa^s})^4  \\
	&\leq& \left[\prod_{i=1,2} q(\kappa_i,\tilde\kappa_i)q(d_i,\tilde d_i)\right]^{\frac{1}{p}+5}
				\left(\frac{\bar d_1\bar d_2}{a_1a_2} \right)^2
				\left(\frac{\bar\kappa_1\bar d_1^2\vee\bar\kappa_2\bar d_2^2}{\bar d_1\bar d_2}\right)^4	 \\
	&=& \left[\prod_{i=1,2} q(\kappa_i,\tilde\kappa_i)q(d_i,\tilde d_i)\right]^{\frac{1}{p}+5}
			\left(\frac{(\bar\kappa_1\bar d_1^2\vee\bar\kappa_2\bar d_2^2)^2}
			{\bar\kappa_1\bar d_1^2\bar\kappa_2\bar d_2^2} \right)^2  \\
	&=& \left[\prod_{i=1,2} q(\kappa_i,\tilde\kappa_i)q(d_i,\tilde d_i)\right]^{\frac{1}{p}+5}
	q(\bar\kappa_1\bar d_1^2,\bar\kappa_2\bar d_2^2)^2.
\end{eqnarray*}
\medskip
Combining all these estimates, we finally  arrive at the following
\begin{kornr}\label{scaledBL} Let ${5}/{3}\leq p\leq2$.
For  every $\alpha>0$ there exist $C_\alpha,\gamma_\alpha>0,$ such that for every pair of patches of hypersurfaces  $S$ and  $\tilde S$ as described in our general assumptions at the beginning of this chapter and every $R>0$, we have
\begin{align}\label{firstBL}
	\| R^*_{S,\tilde S}\|_{L^2(S)\times L^2(\tilde S)\to L^p(Q^1_{S,\tilde S}(R))}
	\leq& C_{\alpha} R^\alpha (\bar\kappa_1\bar\kappa_2)^{\frac{3}{p}-2}(\bar d_1\bar d_2)^{\frac{5}{p}-3} 	
			\min_i (\bar\kappa_i d_i \tilde d_i)^{3-\frac{5}{p}} \nonumber\\	
	&\hskip-1.5cm\times \left(\bar\kappa_1\bar d_1^2 \frac{\kappa_2}{\bar\kappa_2}
			\vee \bar\kappa_2\bar d_2^2 \frac{\kappa_1}{\bar\kappa_1}\right)^{\frac{1}{2}-\frac{1}{p}}
			\left(\bar\kappa_1\bar d_1^2 \frac{\tilde\kappa_2}{\bar\kappa_2}
			\vee \bar\kappa_2\bar d_2^2 \frac{\tilde\kappa_1}{\bar\kappa_1}\right)^{\frac{1}{2}-\frac{1}{p}}\\
	&\hskip-1.5cm\times\left[1+\log^{\gamma_\alpha}\Big(q(\bar\kappa_1\bar d_1^2,\bar\kappa_2\bar d_2^2)
			\prod_{i=1,2} q\big(d_i,\tilde d_i)q(\kappa_i,\tilde\kappa_i)\Big)\right],	\nonumber
\end{align}
where, in correspondence with our Convention 2.1, we have put
$
R^*_{S,\tilde S}(f_1,f_2)=R^*_{\R^2}f_1 \cdot R^*_{\R^2}f_2, \ f_1\in L^2(S), f_2\in L^2(\tilde S).
$
\end{kornr}


\section{Globalization and  $\e$-removal}

\subsection{General results}

The next task will be to extend our inequalities \eqref{firstBL} from the cuboids $Q^1_{S,\tilde S}(R)$ to the whole space,  and to get rid of the factor $R^\alpha$. There is a certain amount   of  ``globalization'' or ``$\e$-removal'' technique  available for this purpose, in particular  Lemma 2.4  by Tao and Vargas in \cite{TV1}, which in return    follows ideas  from   Bourgain's article  \cite{Bo2}. We  shall need  to adapt those techniques to  our  setting, in which it will be important to understand more precisely how the corresponding estimates   will depend  on the  parameters  $\kappa_j$ and $d_j,$ $j=1,2.$

\medskip
To this end, let us consider  two hypersurfaces $S_1$ and $S_2$ in $\R^{d+1},$ defined as graphs $S_j=\{(x,\phi_j(x)): \, x\in U_j\},$ and  assume that  there is a constant  $A$ such that
\begin{align}\label{derivatives}
|\nabla\phi_j(x)|\le A
\end{align}
for all $x\in U_j,$ $j=1,2.$ We will consider the measures $\nu_j$ defined on $S_j$ by
\begin{equation*}
\int_{S_j}g\,d\nu_j=\int_{U_j}f(x,\phi_j(x))\,dx.
\end{equation*}
Note that, under the assumption \eqref{derivatives}, these measures are equivalent to the surface measures on $S_1$ and $S_2.$ We write again
$$
R^*_{S_1,S_2}(f_1,f_2)=R^*_{\R^{d}}f_1\,R^*_{\R^d}f_2.
$$
Denote by $B(0,R)=\{x\in\R^{d+1}:|x|\leq R\}$  the ball of radius  $R$.
Our main result in this section is the following

\begin{lemnr}\label{generaleremoval}
	Let $C_1,C_2,\alpha,s>0,R_0\geq1$, $1\leq p_0<p\leq \infty$, and let $S_1,S_2$ be hypersurfaces with $\nu_1,\nu_2,$  respectively, satisfying \eqref{derivatives}, and  let $\mu$ be a positive   Borel measure on $\R^{d+1}.$
Assume that for all $R\geq R_0$ and all $f_j\in L^2(S_j,\nu_j)$, $j=1,2,$
\begin{enumerate}
	\item $\|R^*_{S_1,S_2}(f_1,f_2)\|_{L^{p_0}(B(0,R),\,\mu)}
				\leq C_1 R^\alpha \|f_1\|_{L^2(S_1,\nu_1)} \|f_2\|_{L^2(S_2,\nu_2)}$,
	\item $|\widehat{d\nu_j}(x)|\leq C_2 (1+|x|)^{-s}$ for all $x\in\R^{d+1},$
\end{enumerate}	
and that $\left(1+{2\alpha}/{s}\right)/p<{1}/{p_0}.$ Then
\begin{align}\label{23jul1}
	\|R^*_{S_1,S_2}(f_1,f_2)\|_{L^p(\R^{d+1},\,\mu)}
				\leq C'  \|f_1\|_{L^2(S_1,\nu_1)} \|f_2\|_{L^2(S_2,\nu_2)},
\end{align}
for all $f_j\in L^2(S_j,\nu_j),\ j=1,2$, where $C'$ only depends on $C_1,C_2,R_0,\alpha,s,p,p_0$.
\end{lemnr}

\begin{proof}
We shall  follow the proof of Lemma 2.4 in \cite{TV1} and only  briefly  sketch the main arguments, indicating  those  changes in the proof that will be needed in our setting. The main difference with \cite{TV1} is that instead of a Stein-Tomas type estimate, we will use the following trivial bound:
\begin{align}\label{trivial0}
\|R^*_{\R^d}f_j\|_{L^{\infty}(\R^{d+1},\,\mu)}\leq\|f_j\|_{L^1(\nu_j)}
				&\leq \|f_j\|_{L^2(S_j,\nu_j)}^{\frac 12}|\widehat{d\nu_j}(0)|^{\frac 12}
\leq C_2^{\frac 12 } \|f_j\|_{L^2(S_j,\nu_j)}^{\frac 12} ,
\end{align}
where we have used our hypothesis (ii).

By \eqref{trivial0} and interpolation,  it then  suffices to prove  a weak-type estimate of the form
\begin{align}\label{wtype}
 \mu( E_\lambda )
\lesssim \lambda^{-p},\qquad \lambda>0,
\end{align}
assuming that $\|f_j\|_{L^2(\nu_j)}=1,$ $j=1,2.$
Here, $E_\lambda= \{ \Re (R^*_{\R^d}f_1\,R^*_{\R^d}f_2)>\lambda \}.$  Given $\lambda>0,$ let us abbreviate $E=E_\lambda.$  We may also assume  $\mu(E)\gtrsim1.$  Tschebychev's inequality implies
$$ \lambda \mu(E)\lesssim \| \chi_{E} R^*_{\R^d}f_1\,R^*_{\R^d}f_2\|_{L^1(\mu)} ,$$
and thus it suffices to show that
\begin{align}\label{old-goal}
 \| \chi_{E} R^*_{\R^d}g_1\,R^*_{\R^d}g_2\|_{L^1(\mu)} \lesssim
\mu(E)^{1/{p'}} \|g_1\|_{L^2(\nu_1)} \|g_2\|_{L^2(\nu_1)}
\end{align}
for arbitrary $L^2$- functions $g_1$ and $g_2$ (which are completely independent of $f_1$ and $f_2$).

To this end, fix $g_2$ with $\|g_2\|_{L^2(\nu_2)} \sim 1,$ and define $T = T_{E,g_2}$ as the linear operator
$$ Tg_1 = \chi_{E} R^*_{\R^d}g_1\,R^*_{\R^d}g_2.
$$
Then, \eqref{old-goal} is equivalent to the inequality
$$ \| Tg_1\|_{L^1(\mu)} \lesssim \mu(E)^{1/{p'}} \|g_1\|_{L^2(\nu_1)}.$$
By duality, it suffices to show that
$$ \| T^* F \|_{L^2(d\nu_1)}
\lesssim \mu(E)^{1/{p'}} \|F\|_{L^\infty(\mu)},
$$
where $T^*$ is  (essentially) the adjoint operator
$$ T^* F = \mathcal F^{-1} ( \chi_E R^*_{\R^d}g_2 F \mu),$$
and $\mathcal F^{-1}$ is the inverse Fourier transform.  We may
assume that $\|F\|_{L^\infty}(\mu) \lesssim 1$.

By squaring this and applying Plancherel's theorem, we reduce
ourselves to showing that
\begin{align}\label{new-goal}
 |\langle \tilde F d\mu* \widehat{d\nu_1}, \tilde Fd\mu \rangle|
\lesssim \mu(E)^{2/{p'}},
\end{align}
where $\tilde F =  \chi_E (R^*_{\R^d}g_2) F.$   Note that
the hypotheses on $F$ and $g_2$ and inequality \eqref{trivial0} imply
\begin{align}\label{point-f}
\|\tilde
F\|_{L^1(\mu)}=\| \chi_E (R^*_{\R^d}g_2) F\|_{L^1(\mu)}\le\|\chi_E\|_{L^1(\mu)}
\|R^*_{\R^d}g_2\|_{L^\infty(\mu)}\|F\|_{L^\infty(\mu)}
\lesssim \mu(E).
\end{align}

From this point on, we follow the proof of \cite{TV1} with the obvious  changes.
Let $R > 1$ be a quantity to be chosen later.  Let $\phi$ be a bump function
which equals 1 on $|x| \lesssim 1$ and vanishes for $|x| \gg 1$, and
write
$d\nu_1 = d\nu_1^{R} + d\nu_{1\,R}$, where
\begin{align}\nonumber
\widehat{d\nu_{1\,R}}(x) = \phi\bigg(\frac{x}{R}\bigg) \widehat{d\nu_1}(x).
\end{align}
From hypothesis (ii) we have
$$ \| \widehat{d\nu_1^{R}} \|_\infty \lesssim R^{-s},$$
and so by \eqref{point-f} we have
$$  |\langle \tilde Fd\mu * \widehat{d\nu_1^R}, \tilde F d\mu\rangle|
\lesssim R^{-s} \mu(E)^2.$$
We now choose $R$ to be
\begin{align}\label{r-prop}
R = \mu(E)^{\frac{2}{sp}},
\end{align}
so that the contribution of $d\nu_1^{R}$ to \eqref{new-goal}
is acceptable.   Thus \eqref{new-goal} reduces to
\begin{align}\nonumber
 |\langle \tilde Fd\mu * \widehat{d\nu_{1\,R}}, \tilde Fd\mu \rangle|
\lesssim \mu(E)^{2/{p'}}.
\end{align}

Following the arguments in \cite{TV1} and skipping details, we may then reduce the problem to proving
$$ \| \chi_E \widehat{\tilde g_1} \widehat{\tilde g_2}\|_{L^1(\mu)}
\lesssim R^{-1/2} R^{-1/2}
\mu(E)^{1/{p'}} \|\tilde g_1\|_2 \|\tilde g_2\|_2,
$$
where for $i=1,2$, $\tilde g_i$ is an arbitrary function
on the $1/R$ neighbourhood of $S_{i,R}$. By H\"older's inequality it suffices to show
\begin{align}\label{b2-targ}
 \| \widehat{\tilde g_1} \widehat{\tilde g_2}\|_{L^{p_0}(\mu)}
\lesssim \mu(E)^{-1/{p_0'}} R^{-1/2} R^{-1/2}
\mu(E)^{1/{p'}} \|\tilde g_1\|_2 \|\tilde g_2\|_2.
\end{align}
Moreover, using the first hypothesis of the lemma,  we obtain
$$ \| \widehat{\tilde g_1} \widehat{\tilde g_2}\|_{L^{p_0}(\mu)}
\lesssim R^{\alpha-1}  \| \widehat{\tilde g_1} \|_2\| \widehat{\tilde g_2} \|_2.
$$
Comparing this with \eqref{b2-targ}, we see that we will be done if
$$ R^\alpha \lesssim \mu(E)^{-1/{p_0'}} \mu(E)^{1/{p'}} = \mu(E)^{{1}/{p_0} - {1}/{p}}.$$
But this follows from \eqref{r-prop} and the assumption
$\left(1+{2\alpha}/{s}\right)/p<{1}/{p_0}.$
\end{proof}


\subsection{Application to the setting of  Chapter \ref{Scaling}}
Let us now come back to the situation described by our GENERAL ASSUMPTIONS in Chapter \ref{scaling}, i.e.,
 we are interested in pairs of surfaces $S=\graph(\phi|_{U})$, $U=r+[0,d_1]\times[0,d_2],$ with principal curvatures on $S$  comparable to $\kappa_i=r_i^{m_i-2}$, $r_i\geq d_i,$ and $\tilde S=\graph(\phi|_{\tilde U}),$ with corresponding quantities $\tilde r_i, \tilde d_i,\tilde\kappa_i,\tilde\kappa.$

 Recall also the notation defined in \eqref{model}, \eqref{quantities}, and assume that the conditions \eqref{maxproduct} and \eqref{separationcond} are satisfied.

We consider the measure $\nu_S$  supported on $S$  given by
$$\int_{S} f\,d\nu_S:=\int_U f(x_1,x_2,\phi(x_1,x_2))\,dx_1dx_2,
$$
and define $\nu_{\tilde S}$ on $\tilde S$ analogously.

\subsubsection{Decay of the Fourier transform}


\begin{lemnr}\label{FTdecay} Let $s={1}/({m_1\vee m_2})$. For any $r^0\in U\cup \tilde U$ we then have the following uniform estimate for $x\in\R^3:$
\begin{align}\label{hatnuest}
	&|\widehat{d\nu_S}(x)|+|\widehat{d\nu_{\tilde S}}(x)| \\
	&\leq C_s \bar d_1 \bar d_2
	\Big(1+|(\bar d_1(x_1+\partial_1\phi(r^0)x_3|+|\bar d_2(x_2+\partial_2\phi(r^0)x_3)|+
				|(\bar\kappa_1\bar d_1^2 \vee \bar\kappa_2\bar d_2^2)x_3)|\Big)^{-s}. \nonumber
\end{align}
\end{lemnr}
\begin{proof}
We only consider $\nu=\nu_S,$ since the proof for $\nu_{\tilde S}$ is analogous.
Recall that $\phi$ splits into $\phi(x)=\psi_1(x_1)+\psi_2(x_2),$ so that
\begin{align*}
	|\widehat{d\nu}(x)|
	=& \left|\int_{r_1}^{r_1+d_1} e^{-i(x_1\xi_1+x_3\psi_1(\xi_1))} d\xi_1
			\int_{r_2}^{r_2+d_2} e^{-i(x_2\xi_2+x_3\psi_2(\xi_2))} d\xi_2\right|.
\end{align*}
 Next, for  $i\in\{1,2\},$ we have
\begin{align*}
	I _i=& \left|\int_{r_i}^{r_i+d_i} e^{-i(x_i\xi_i+x_3\psi_i(\xi_i))} d\xi_i\right|	=
	 \left|\int_{0}^{d_i} e^{-i(x_i(r_i+y_i)+x_3\psi_i(r_i+y_i))} d y_i\right|	 \\
	=& \left|\int_{0}^{d_i} e^{-i\big((x_i+\psi_i'(r_i)x_3)y_i
							+x_3(\psi_i(r_i+y_i)-\psi_i(r_i)-\psi_i'(r_i)y_i)\big)}d y_i\right |\\
	=& \left|\int_{0}^{d_i} e^{-i((x_i+\psi_i'(r_i)x_3)y_i
											+x_3\Psi(y_i))}d y_i\right|	\\
	=& d_i\left|\int_{0}^{1} e^{-i((x_i+\psi_i'(r_i)x_3)d_iy_i
											+x_3\kappa_id_i^2\Psi_i(d_i y_i))}d y_i\right|
\end{align*}
where $\Psi_i(y_i)=(\psi_i(r_i+d_1y_i)-\psi_i(r_i)-\psi_i'(r_i)d_iy_i)/(\kappa_id_i^2)$, so  that in particular
\begin{align*}
\big|\frac{d}{d y_i} \Psi_i( y_i)\big|= \Big|\frac{ \psi_i'(r_i+d_iy_i)-\psi_i'(r_i)}{\kappa_id_i^2}d_i\big|  \lesssim \frac{\kappa_id_i}{\kappa_id_i^2}d_i\sim 1,\\
	\frac{d^2}{d y_i^2} \Psi_i(y_i)=\frac{ \psi_i''(r_i+d_iy_i)}{\kappa_id_i^2}d_i^2 \sim 1.
\end{align*}
Thus, by either applying van der Corput's lemma of order $2,$  or by integrating by parts (if $|(d_i(x_i+\psi_i'(r_i)x_3|\gg |\kappa_id_i^2 x_3|),$ we obtain that
\begin{align}\label{23sept1706}
  I_i\lesssim d_i (1+|(d_i(x_i+\psi_i'(r_i)x_3|+|\kappa_id_i^2 x_3)|)^{-\frac{1}{2}}.
\end{align}

We next claim that the  distortion  ${d_i}/{\bar d_i}$ in the sidelengths is bounded by the distortion in the size of the space variable $r_i$, i.e.,
\begin{align}\label{distortion}
	\frac{d_i}{\bar d_i} \lesssim \frac{r_i}{\bar r_i}.
\end{align}
If $r_i\sim \bar r_i$, the statement is obvious, so assume $r_i\ll \bar r_i$. Then  $\tilde r_i=\bar r_i,$ and furthermore
by our assumptions we have  $d_i\leq r_i$ and $\bar r_i\sim|r_i-\tilde r_i|\lesssim \bar d_i$ (compare the separation condition  \eqref{separationcond}). Thus \eqref{distortion} follows also in this case.
As $\kappa_i=r_i^{m_i-2}$, we conclude from \eqref{distortion} that
\begin{align}\label{distortion2}
	\frac{\kappa_i d_i^2}{\bar\kappa_i\bar d_i^2} \gtrsim \left(\frac{d_i}{\bar d_i}\right)^{m_i}.
\end{align}
In combination, the estimates \eqref{distortion} and \eqref{distortion2} imply that
\begin{align*}
	1+|(d_i(x_i+\psi_i'(r_i)x_3)|+|\kappa_id_i^2 x_3)|	
	\gtrsim \left(\frac{d_i}{\bar d_i}\right)^{m_i} \big(1+|(\bar d_i(x_i+\psi_i'(r_i)x_3)|+|\bar\kappa_i\bar d_i^2 x_3)|\big).
\end{align*}
Since we may replace the exponent $-1/2$ in the right-hand side of  \eqref{23sept1706} by $-1/m_i,$  we  now see that we may estimate
\begin{align}\label{Ii2}
	I_i \,\lesssim \bar d_i (1+|\bar d_i(x_i+\psi_i'(r_i)x_3|+|\bar\kappa_i\bar d_i^2 x_3)|)^{-\frac{1}{m_i}}.
\end{align}

Finally, in order  to pass  from the point $r$ to an arbitrary  point $r^0\in U\cup\tilde U$ in these estimates, observe that  by \eqref{separationcond} we  have
$|r_i-r_i^0|\leq |r_i-\tilde r_i|+\bar d_i {\sim} \bar d_i,$ and hence
$$
	 \bar d_i|\psi'_i(r_i)-\psi'_i(r^0_i)|	
	\leq   \bar\kappa_i|r_i-r_i^0|\bar d_i	
	\lesssim \bar\kappa_i\bar d_i^2	,					 						
$$
since $|\psi''_i|\lesssim  \bar\kappa_i$ on $[r_i,r_i+d_i]\cup [\tilde r_i,\tilde r_i+\tilde d_i].$
Therefore \eqref{Ii2} implies that also
\begin{align*}
	I_i\, \lesssim \bar d_i (1+|\bar d_i(x_i+\partial_i\phi(r^0)x_3|+|\bar\kappa_i\bar d_i^2 x_3)|)^{-\frac{1}{m_i}}.
\end{align*}

The estimate \eqref{hatnuest} is now immediate.
\end{proof}

\medskip

\subsubsection{Linear change of variables and verification of the assumptions of  Lemma \ref{generaleremoval}}

In view of Lemma \ref{FTdecay}, let us fix $r^0\in U\cup \tilde U,$ and define the linear transformation $T=T_{S,\tilde S}$ of $\R^3$ by
\begin{align*}
	T(x)=\big(\bar d_1(x_1+\partial_1\phi(r^0)x_3),\bar d_2(x_2+\partial_2\phi(r^0)x_3),(\bar\kappa_1\bar d_1^2 \vee \bar\kappa_2\bar d_2^2)x_3\big).
\end{align*}
Then estimate \eqref{hatnuest} reads
$$
|\widehat{d\nu_S}(x)|+|\widehat{d\nu_{\tilde S}}(x)|
	\leq C_s \bar d_1 \bar d_2
	\big(1+|T(x)|\big)^{-s}.
$$
Therefore, in order to apply Lemma \ref{generaleremoval}, we will consider the rescaled surfaces
\begin{align}\label{rescaleS}
S_1=(T^t)^{-1} S\quad  \mbox{and} \quad S_2=(T^t)^{-1}\tilde S.
\end{align}
Then we find  that
\begin{align*}
 S_1&=\{(T^t)^{-1}(x_1,x_2,\psi_1(x_1)+\psi_2(x_2)): (x_1,x_2)\in U\} =
    \Big\{\big(\frac{x_1}{\bar d_1},\frac{x_2}{\bar d_2},\frac1{\bar\kappa_1\bar d_1^2 \vee \bar\kappa_2\bar d_2^2}\\
   &\times(-\partial_1\phi(r_0)x_1-\partial_2\phi(r_0)x_2+\psi_1(x_1)+\psi_2(x_2))\big): (x_1,x_2)\in U\Big\} \\
   &=\{(y_1,y_2,\psi(y_1, y_2)): (y_1,y_2)\in U_1\},
\end{align*}
where $U_1=\{(y_1,y_2)=(\frac{x_1}{\bar d_1},\frac{x_2}{\bar d_2}): (x_1,x_2)\in U\}$  is a square of sidelength $\le 1 $ and
$$\psi(y_1,y_2)=\frac1{\bar\kappa_1\bar d_1^2 \vee \bar\kappa_2\bar d_2^2}\big(-\bar d_1\partial_1\phi(r_0)y_1-\bar d_2\partial_2\phi(r_0)y_2+\psi_1(\bar d_1 y_1)+\psi_2(\bar d_2 y_2)\big).
$$ We have a similar expression for $S_2.$

In $S_1$ we consider the measure  $d\nu_1$ defined by
$$
\int_{S_1} g\,d\nu_1=\frac1{\bar d_1\bar d_2}\int_S g((T^t)^{-1}x)\,d\nu_S(x).
$$
By our definition of d$\nu$ and $\psi,$ this may be re-written as
\begin{align*}
\int_{S_1} g\,d\nu_1&=\frac1{\bar d_1\bar d_2}\int_U g\big((T^t)^{-1}(x_1,x_2,\phi(x_1,x_2)\big)\,dx_1dx_2\\
&=\frac1{\bar d_1\bar d_2}\int_U g\big(\frac{x_1}{\bar d_1},\frac{x_2}{\bar d_2},\psi(\frac{x_1}{\bar d_1},\frac{x_2}{\bar d_2})\big)
\,dx_1dx_2\\
&=\int_{U_1} g(y_1,y_2,\psi(y_1,y_2))\,dy_1dy_2.
\end{align*}
Moreover, we have
\begin{align}\label{twooper}
\widehat{g d\nu_1}(\xi)=\frac1{\bar d_1\bar d_2}   \big( g\circ(T^t)^{-1}d\nu_S\big) \, \widehat{} \,\, (T^{-1}\xi),
\end{align}
and therefore
$$
|\widehat{d\nu_1}(x)|\le C_s(1+|x|)^{-s}.
$$
We have a similar estimate for $\widehat{d\nu_2}.$ Thus, the hypotheses (ii) in Lemma \ref{generaleremoval} are  satisfied. To check that condition \eqref{derivatives} is satisfied for $S_1$ and $S_2$ too,  we compute
$$
\bigg|\frac{\partial \psi}{\partial y_1}\bigg|=\frac1{\bar\kappa_1\bar d_1^2 \vee \bar\kappa_2\bar d_2^2}\big|-\bar d_1\partial_1\phi(r_0)+\bar d_1\psi_1'(\bar d_1 y_1)\big|.
$$
Writing $r_0=(\bar d_1 y_{1,0},\bar d_2 y_{2,0}),$ we see that
\begin{align*}
\bigg|\frac{\partial \psi}{\partial y_1}\bigg|&=\frac1{\bar\kappa_1\bar d_1^2 \vee \bar\kappa_2\bar d_2^2}|-\bar d_1\psi_1'(\bar d_1 y_{1,0})+\bar d_1\psi_1'(\bar d_1 y_1)|\\
&\sim \frac1{\bar\kappa_1\bar d_1^2 \vee \bar\kappa_2\bar d_2^2}|\bar d_1^2(y_1-y_{1,0})\phi_1''|\le \frac{\bar\kappa_1\bar d_1^2 }{\bar\kappa_1\bar d_1^2 \vee \bar\kappa_2\bar d_2^2}|y_1-y_{1,0}|\le C_{m_1,m_2},
\end{align*}
and in a  similar way we find that the derivative with respect  to $y_2$ is bounded.
Hence, hypothesis \eqref{derivatives} is satisfied for $\psi$ in place of $\phi.$
\bigskip

What remains to be checked is condition  (i) in  Lemma  \ref{generaleremoval}.
Observe  first that our local bilinear estimate  for $S$ and $\tilde S$  in Corollary \ref{scaledBL} is restricted to cuboids (compare \eqref{cuboidrescaled})
\begin{align}\label{cuboids2}
	Q^1(R)=Q^1_{S,\tilde S}(R)
	=\left\{x\in\R^3:~ |x_i+\partial_i\phi(r^0)x_3|\leq \frac{R}{\bar d_i},i=1,2,~|x_3|
	\leq\frac{R}{\bar\kappa_1\bar d_1^2 \wedge \bar\kappa_2\bar d_2^2}\right\},
\end{align}
where $r^0$ is either\footnote
{Recall that we  have some algorithm  how to choose $r^0$, but this  will not relevant here.} $r$ or $\tilde r$.
Obviously $T^{-1}(B(0,R))=\{x\in\R^3:|Tx|\leq R\}\subset Q^1(R)$.

\

Define
\begin{align}
	A=& (\bar\kappa_1\bar\kappa_2)^{-2}(\bar d_1\bar d_2)^{-3} 	\min_i (\bar\kappa_i d_i \tilde d_i)^{3} \nonumber\\	
	&		\times \left(\bar\kappa_1\bar d_1^2 \frac{\kappa_2}{\bar\kappa_2}
			\vee \bar\kappa_2\bar d_2^2 \frac{\kappa_1}{\bar\kappa_1}\right)^{\frac{1}{2}} \left(\bar\kappa_1\bar d_1^2 \frac{\tilde\kappa_2}{\bar\kappa_2} \vee \bar\kappa_2\bar d_2^2 \frac{\tilde\kappa_1}{\bar\kappa_1}\right)^{\frac{1}{2}}\big(1+ \log^{\gamma_\alpha}Q\big),\\
	B=&(\bar\kappa_1\bar\kappa_2)^{3}(\bar d_1\bar d_2)^{5} 	\min_i (\bar\kappa_i d_i \tilde d_i)^{-5} \nonumber\\	
	&		\times \left(\bar\kappa_1\bar d_1^2 \frac{\kappa_2}{\bar\kappa_2}
			\vee \bar\kappa_2\bar d_2^2 \frac{\kappa_1}{\bar\kappa_1}\right)^{-1} \left(\bar\kappa_1\bar d_1^2 \frac{\tilde\kappa_2}{\bar\kappa_2} \vee \bar\kappa_2\bar d_2^2 \frac{\tilde\kappa_1}{\bar\kappa_1}\right)^{-1},
\end{align}
where
$$Q=Q(S,\tilde S)=q(\bar\kappa_1\bar d_1^2,\bar\kappa_2\bar d_2^2)	\prod_{i=1,2} q\big(d_i,\tilde d_i)q(\kappa_i,\tilde\kappa_i)
$$
and  $q(a,b)=(a\vee b)/{a\wedge b}\geq1$ had been   defined to be the maximal quotient of  $a$ and $b.$ In some sense $Q$ is a "degeneracy quotient" that measures how much (for instance) quantities $d_i,\tilde d_i$ differ from their maximum $\bar d_i$.

Then the estimate \eqref{firstBL}  in Corollary  \ref{scaledBL}, valid for  $5/3\leq p\leq2,$ can be  re-written in terms of these quantities as
\begin{align}\label{localAB}Q^1(R)
	\| R^*_{S,\tilde S}\|_{L^2(S)\times L^2(\tilde S)\to L^p(Q^1(R))}
	\leq & C_{\alpha} R^\alpha AB^\frac{1}{p}.
\end{align}

Now, in order to check hypothesis (i) in  Lemma  \ref{generaleremoval}, let us choose for $\mu$ the measure on $\R^3$ given by
 $$
 d\mu=\tilde B^{-1}d\xi,\quad\mbox{where} \quad \tilde B= |\det T|\Big(\frac A{\bar d_1\bar d_2}\Big)^{p_0} B,
 $$
and   where $d\xi$ denotes the Lebesgue measure. Notice also that  \eqref{twooper} implies that, for any measurable set $E\subset \R^3$ and any exponent $p,$ we have
\begin{align}\label{restrescale}
\|R^*_{S_1,S_2}&(f_1,f_2)\|_{L^{p}(E,\mu)}
&=\frac{A^{-p_0/p}B^{-1/p}}{(\bar d_1\bar d_2)^{2-p_0/p}}\|R^*_{S,\tilde S}(f_1\circ(T^t)^{-1}),f_2\circ(T^t)^{-1})\|_{L^{p}(T^{-1}(E),d\xi)}.
\end{align}
In particular, we obtain
\begin{align*}
\|R^*_{S_1,S_2}&(f_1,f_2)\|_{L^{p_0}(B(0,R),\mu)}\\
&=\frac{A^{-1}B^{-1/p_0}}{\bar d_1\bar d_2}\|R^*_{S,\tilde S}(f_1\circ(T^t)^{-1}),f_2\circ(T^t)^{-1})\|_{L^{p_0}(T^{-1}(B(0,R)),d\xi)}\\
&\le \frac{A^{-1}B^{-1/p_0}}{\bar d_1\bar d_2}\|R^*_{S,\tilde S}(f_1\circ(T^t)^{-1}),f_2\circ(T^t)^{-1})\|_{L^{p_0}(Q^1(R),d\xi)}
\end{align*}
Invoking  \eqref{localAB},  we thus see that for $5/3\le p_0\le2$ and every $\alpha>0$
\begin{align*}
\|R^*_{S_1,S_2}&(f_1,f_2)\|_{L^{p_0}(B(0,R),\mu)}\\
&\le \frac1{\bar d_1\bar d_2}C_\alpha R^\alpha\|f_1\circ(T^t)^{-1}\|_{L^2(d\nu_S)}\|f_2\circ(T^t)^{-1}\|_{L^2(d\nu_{\tilde S})}\\
&=C_\alpha R^\alpha\|f_1\|_{L^2(d\nu_1)}\|f_2\|_{L^2(d\nu_2)},
\end{align*}
 which shows that hypothesis (i) in the Lemma \ref{generaleremoval} is satified. Applying this  lemma and using again  identity \eqref{restrescale} and  the definitions of $\mu,$ $\nu_1$  and $\nu_2,$  we find that for any  $g_1$ and $g_2$ supported in $S$ and $\tilde S,$ respectively,  and any $p$ satisfying the assumptions of Lemma \ref{generaleremoval}, we have
\begin{align}\label{Rglobal}
\|R^*_{S,\tilde S}(g_1,g_2)\|_{L^{p}(d\xi)}\le C
(\bar d_1\bar d_2)^{1-p_0/p}A^{p_0/p}B^{1/p}
\|g_1\|_{L^2(d\nu_1)}\|g_2\|_{L^2(d\nu_2)}.
\end{align}
Finally, putting $\e=1-p_0/p,$ and recalling that we may choose $\alpha$ in  Lemma \ref{generaleremoval} as small as we wish,  then by applying  H\"older's inequality  in order to replace the $L^2$-norms on the right-hand side of \eqref{Rglobal} by the $L^q$-norms, we arrive at the following global estimate:
\begin{thmnr}\label{globalBL}
	Let $5/3< p \leq 2$, $q\ge2,$ $\e>0$. Then there exist constants $C=C_{p,\e}$ and $\gamma=\gamma_{p,\e}>0$ such that
\begin{align}\label{global}
	\| R^*_{S,\tilde S}\|_{L^q(S)\times L^q(\tilde S)\to L^p(\R^n)}
	\leq& C (\bar\kappa_1\bar\kappa_2)^{\frac{3}{p}-2+2\e}(\bar d_1\bar d_2)^{\frac{5}{p}-3+4\e}
				(d_1 d_2\tilde d_1\tilde d_2)^{\frac{1}{2}-\frac{1}{q}} (1+\log^\gamma Q) \\	
	&\hskip-4cm	\times \min_i (\bar\kappa_i d_i \tilde d_i)^{3-3\e-\frac{5}{p}}	\left(\bar\kappa_1\bar d_1^2 \frac{\kappa_2}{\bar\kappa_2} \vee \bar\kappa_2\bar d_2^2 \frac{\kappa_1}{\bar\kappa_1}\right)^{\frac{1-\e}{2}-\frac{1}{p}}	\left(\bar\kappa_1\bar d_1^2 \frac{\tilde\kappa_2}{\bar\kappa_2} \vee \bar\kappa_2\bar d_2^2 \frac{\tilde\kappa_1}{\bar\kappa_1}\right)^{\frac{1-\e}{2}-\frac{1}{p}}, \nonumber
\end{align}
uniformly in $S$ and $\tilde S$,
where $Q=q(\bar\kappa_1\bar d_1^2,\bar\kappa_2\bar d_2^2)	\prod_{i=1,2} q\big(d_i,\tilde d_i)q(\kappa_i,\tilde\kappa_i)$  and $q(a,b)=(a\vee b)/{a\wedge b}$.
\end{thmnr}








\section{Dyadic Summation}\label{dyadicsum}
Recall that our hypersurface  of interest  is the graph of a smooth function $\phi(x_1,x_2)=\psi_1(x_1)+\psi_2(x_2)$  defined over the  square $]0,1[\times ]0,1[.$ We assume $\phi$ to be extended continuously to the closed square
$Q=[0,1]\times [0,1]$ (this extension will  in the end not   really play any role, but it will be more convenient to work with a closed square). By means of  a kind of Whitney decomposition of the direct product $Q\times Q$  near the ``diagonal'', following some standard procedure in the bilinear approach, we can decompose $Q\times Q$ into products of  congruent rectangles $U$ and $\tilde U$ of dyadic side lengths, which are ``well-separated neighbors''  in some sense. The next step will therefore consist in establishing bi-linear estimates for  pairs of sub-hypersurfaces supported over such pairs of neighboring rectangles. Notice that if one of these  rectangles  meets one of the coordinate axes, then  the principal curvature in at least one coordinate direction  will no longer be of a certain size, but will indeed go down to zero within this rectangle. We then perform an additional  dyadic decomposition of this rectangle in order to achieve that both principal curvatures wi
 ll be of a certain size on each of the dyadic sub-rectangles (compare Figure \ref{boxes1}).  To these  we can then apply our estimates from Theorem \ref{globalBL}.  Thus, in this section we shall work under the following
\medskip

\textsc{GENERAL ASSUMPTIONS:}
$U=[r_12^{-j_1},(r_1+1)2^{-j_1}]\times [r_22^{-j_2},(r_2+1)2^{-j_2}]$ and $\tilde U=[\tilde r_12^{- j_1},(\tilde r_1+1)2^{- j_1}]\times [\tilde r_22^{-\tilde j_2},(\tilde r_2+1)2^{-\tilde j_2}],$ with $r_i,\tilde r_i, j_i\in \N,$  are two congruent closed  bi-dyadic  rectangles in $ [0,1]\times[0,1]$    whose side length and  distance  between them in  the $x_i$ - direction is  equal to $\rho_i=2^{-j_i},$   both for $i=1$ and $i=2.$

 By $\varkappa_i$ we denote the maximum value of the  principal curvature in $x_i$ - direction of both $S=\graph(\phi|_U)$ and $\tilde{S}=\graph(\phi|_{\tilde{U}})$.

\begin{thmnr}\label{nextBLthm}
Let $5/3<p<2$, $q\geq2$, $\e>0,$ and assume that  $(m_1\vee m_2 +3)\left(1/p-1/2\right)<{1}/{q'}$. Then we have
\begin{align}\label{global2}
	\| R^*_{S,\tilde S}\|_{L^q(S)\times L^q(\tilde S)\to L^p(\R^3)}
	\leq & C_{p,q,\e} (\rho_1\rho_2)^{\frac{2}{q'}-\frac{1}{p}}
	(\varkappa_1\rho^2_1\vee \varkappa_2\rho^2_2)^{\frac{1}{p}-1+\e }
		(\varkappa_1\rho^2_1\wedge \varkappa_2\rho^2_2)^{1-\frac{2}{p}-\e}.	
\end{align}
\end{thmnr}

\begin{proof}
If $U$ does not intersect with the $x_i$-axis, then the  principal curvature in $x_i$-direction on $U$ is indeed comparable to $\varkappa_i$. Otherwise we decompose $U$ further into  sets with (roughly) constant principal curvatures in order to apply the previous results. More precisely, to each dyadic interval $I=[r2^{-j},(r+1)2^{-j}]$, $r,j\in\N,$ we associate a family of subsets $\{I(k)\}_{k\in \NN_0}$ with $\mathop{\bigcup}\limits_{k\in\NN_0}I(k)=I,$ according to the following two alternatives:
\begin{enumerate}
	\item If $r>0$, then choose $\NN_0=\{0\}$ and $I(0)=I$.
	\item If $r=0$,  then choose $\NN_0=\N=\{1,2,3,\ldots\}$ and $I(k)=[2^{-k}(r+1)2^{-j},2^{1-k}(r+1)2^{-j}]$.
\end{enumerate}
If we write $U=I_1\times I_2$, then denote by $\{I_i(k_i)\}_{k_i\in\NN_i}$ their associated family, and let $U(k)=I(k_1)\times I(k_2)$, $k=(k_1,k_2)\in\NN=\NN_1\times\NN_2$, $S(k)=\graph(\phi|_{U(k)})$. Introduce $\tilde \NN$, $\tilde U(k)$ and $\tilde S(k)$, $k\in\tilde\NN,$ in an analogous manner.
Other  relevant quantities  are the principal curvatures on $U(k)$, i.e.,
\begin{align}\label{23sep1714}
	\kappa_i(k_i):=2^{-k_i(m_i-2)}\varkappa_i,
\end{align}
and the side lengths of $U(k)$
\begin{align}\label{23sep1715}
	d_i(k_i)=2^{-k_i}\rho_i.
\end{align}

\begin{figure}
  \includegraphics{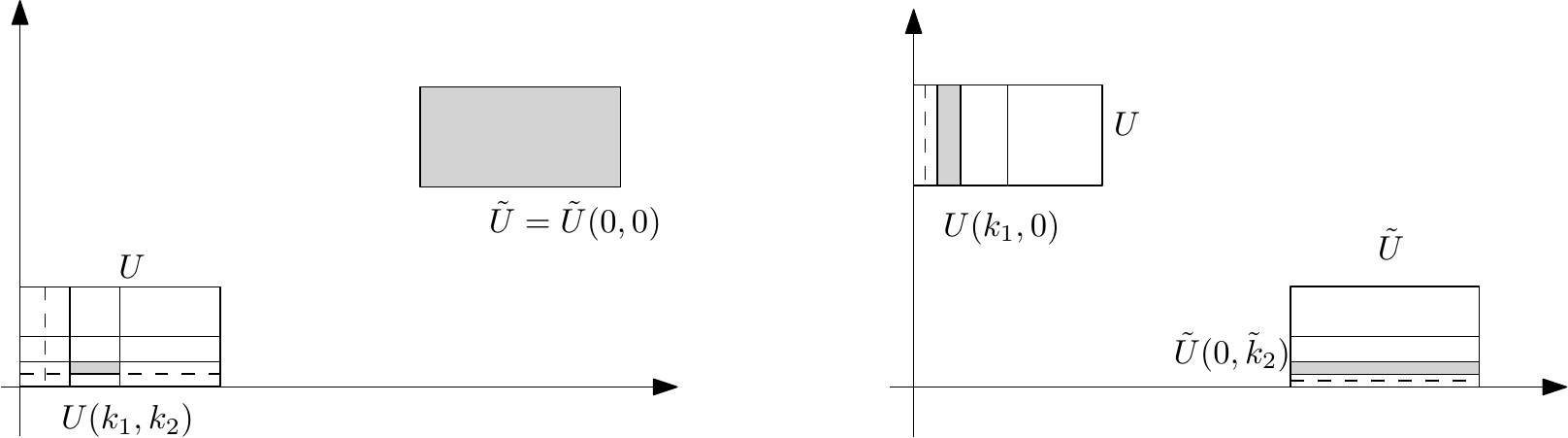}
\caption{Two possibilities for the decomposition into subboxes}
\label{boxes1}       
\end{figure}

A simple but crucial observation is that since $I_i$ and $\tilde I_i$ are separated for both $i=1$ and $i=2$, we have  $\NN_i=\{0\}$ or $\tilde\NN_i=\{0\}$ (cf. Figures \ref{boxes1}).
Hence for each pair $(k_i,\tilde k_i)\in\NN_i\times\tilde\NN_i$, $k_i=0$ or $\tilde k_i=0$, and thus
\begin{align}
	&\bar\kappa_i(k_i,\tilde k_i)
	:=\max\{\kappa_i(k_i),\tilde\kappa_i(\tilde k_i)\}
	=\max\{2^{-k_i(m_i-2)},2^{-\tilde k_i(m_i-2)}\}\varkappa_i
	=\varkappa_i \label{18jul1}\\
	& \bar d_i(k_i,\tilde k_i)
	:= \max\{d_i(k_i),\tilde d_i(\tilde k_i)\}
	= \max\{2^{-k_i},2^{-\tilde k_i}\}\rho_i
	= \rho_i.\label{18jul2}
\end{align}
We conclude that
\begin{align}
	&\frac{\kappa_i(k_i)}{\bar\kappa_i(k_i,\tilde k_i)} =2^{-k_i(m_i-2)}, \label{18jul3}\\
	&\frac{\tilde \kappa_i(k_i)}{\bar\kappa_i(k_i,\tilde k_i)} =2^{-\tilde k_i(m_i-2)}\label{18jul4}
	\frac{d_i(k_i)}{\bar d_i} = 2^{-k_i},		\\
	&d_i(k_i)\tilde d_i(\tilde k_i) =	2^{-k_i-\tilde k_i}	\rho_i^2	= 2^{-k_i\vee\tilde k_i}	\rho_i^2,								\label{18jul5}
\end{align}
hence
\begin{align}\label{logterm}
	Q =& q(\bar\kappa_1\bar d_1^2,\bar\kappa_2\bar d_2^2)
	\prod_{i=1,2}q(d_i(k_i),\tilde d_i(\tilde k_i))q(\kappa_i(k_i),\tilde\kappa_i(\tilde k_i))	\nonumber\\
	\leq& \frac{\bar\kappa_1\bar d_1^2\vee\bar\kappa_2\bar d_2^2}{\bar\kappa_1\bar d_1^2\wedge\bar\kappa_2\bar d_2^2}
	 2^{m_1(k_1+\tilde k_1) + m_2(k_2+\tilde k_2)}.
\end{align}

Thus, if we apply  inequality \eqref{global} from Theorem \ref{globalBL} to the pairs of hypersurfaces $S(k),\tilde S(\tilde k)$  and estimate by means of \eqref{18jul1}--\eqref{logterm},
 then we get
\begin{align*}
	&\| R^*_{S,\tilde S}\|_{L^q\times L^q\to L^p}
	\leq \sum_{k\in\NN,\tilde k\in\tilde \NN}
	\| R^*_{S(k),\tilde S(\tilde k)}\|_{L^q\times L^q\to L^p}\\
	&\lesssim (\varkappa_1\rho_1^2\varkappa_2\rho_2^2)^{\frac{3}{p}-2+2\e}(\rho_1\rho_2)^{\frac{2}{q'}-\frac{1}{p}} \log^{\gamma}\left(\frac{\varkappa_1\rho_1^2}{\varkappa_2\rho_2^2}	
					+\frac{\varkappa_2\rho_2^2}{\varkappa_1\rho_1^2}\right)
			\Big(\sum_{k\in\NN,\tilde k\in\tilde \NN}[1+k_1+\tilde k_1 + k_2+\tilde k_2]^{\gamma} \\
	&\times (\varkappa_1\rho_1^2 2^{-k_1-\tilde k_1}\wedge\varkappa_2\rho_2^2 2^{-k_2-\tilde k_2})^{3-3\e-\frac{5}{p}} \ 2^{-(k_1+\tilde k_1+k_2+\tilde k_2)\left(\frac{1}{2}-\frac{1}{q}\right)}\Big) \\		
	&\times \left(\varkappa_1\rho_1^2 2^{-k_2(m_2-2)} \vee \varkappa_2
	\rho_2^2 2^{-k_1(m_1-2)}\right)^{\frac{1-\e}{2}-\frac{1}{p}}
	\left(\varkappa_1\rho_1^2 2^{-\tilde k_2(m_2-2)} \vee \varkappa_2
	\rho_2^2 2^{-\tilde k_1(m_1-2)}\right)^{\frac{1-\e}{2}-\frac{1}{p}}.
\end{align*}
We claim that 
\begin{align}\nonumber
	&\sum_{k\in\NN,\tilde k\in\tilde \NN}[1+k_1+\tilde k_1 + k_2+\tilde k_2]^{\gamma}
	(\varkappa_1\rho_1^2 2^{-k_1-\tilde k_1}\wedge\varkappa_2\rho_2^2 2^{-k_2-\tilde k_2})^{3-3\e-\frac{5}{p}}\\
	&\hskip1cm\times  2^{-(k_1+\tilde k_1+k_2+\tilde k_2)\left(\frac{1}{2}-\frac{1}{q}\right)}  \left(\varkappa_1\rho_1^2 2^{-k_2(m_2-2)} \vee \varkappa_2
	\rho_2^2 2^{-k_1(m_1-2)}\right)^{\frac{1-\e}{2}-\frac{1}{p}}	\nonumber\\	
	&\hskip1cm \times \left(\varkappa_1\rho_1^2 2^{-\tilde k_2(m_2-2)} \vee \varkappa_2
	\rho_2^2 2^{-\tilde k_1(m_1-2)}\right)^{\frac{1-\e}{2}-\frac{1}{p}} \nonumber\\
	&\lesssim (\varkappa_1\rho_1^2 \wedge \varkappa_2\rho_2^2 )^{3-3\e-\frac{5}{p}}(\varkappa_1\rho_1^2 \vee \varkappa_2\rho_2^2)^{1-\e-\frac{2}{p}}. \label{doublesum1}		
\end{align}
Taking this for granted, we  arrived at estimate \eqref{global2}:
\begin{align*}
	&\| R^*_{S,\tilde S}\|_{L^2\times L^2\to L^p(Q_{S,\tilde S}(R))} \lesssim 	(\varkappa_1\rho_1^2\varkappa_2\rho_2^2)^{\frac{3}{p}-2+2\e}(\rho_1\rho_2)^{\frac{2}{q'}-\frac{1}{p}}	\\
			&\hskip1cm\times(\varkappa_1\rho_1^2 \vee \varkappa_2\rho_2^2)^{1-\e-\frac{2}{p}}(\varkappa_1\rho_1^2 \wedge \varkappa_2\rho_2^2 )^{3-3\e-\frac{5}{p}}\log^{\gamma}\left(\frac{\varkappa_1\rho_1^2}{\varkappa_2\rho_2^2}	+\frac{\varkappa_2\rho_2^2}{\varkappa_1\rho_1^2}\right)	\\
	&= (\rho_1\rho_2)^{\frac{2}{q'}-\frac{1}{p}}	
		(\varkappa_1\rho_1^2 \vee \varkappa_2\rho_2^2)^{\frac{1}{p}-1+\e}(\varkappa_1\rho_1^2 \wedge \varkappa_2\rho_2^2 )^{1-\frac{2}{p}-\e}\log^{\gamma}\left(\frac{\varkappa_1\rho_1^2}{\varkappa_2\rho_2^2}	+\frac{\varkappa_2\rho_2^2}{\varkappa_1\rho_1^2}\right).				
\end{align*}

We are thus left with the estimation of  the dyadic sum in \eqref{doublesum1}.
Let
$$\mu=\frac{1}{p}-\frac{1-\e}{2}>0, \quad \nu=3-3\e-\frac{5}{p}>0, \quad \omega=\frac{1}{2}-\frac{1}{q}>0, \quad c_i=m_i-2.
$$
Then $c_i\mu < \nu+\omega$ is equivalent to $m_i\left({1}/{p}-1/2\right) +\landau(\e) < {1}/{q'}$. This is satisfied since by  our assumptions in the theorem  we have  $m_i\left({1}/{p}-{1}/{2}\right)< {1}/{q'},$ and we can chose $\e$ arbitrarily small.	
\medskip

Estimate  \eqref{doublesum1} will then be an easy consequence  of the next lemma.
 Indeed,  recalling our earlier observation  that for each pair $(k_i,\tilde k_i)\in\NN_i\times\tilde\NN_i$ one of the entries $k_i$ or $\tilde k_i$ must be zero, we see that we  have to sum over at most two of the parameters $k_1,k_2,\tilde k_1,\tilde k_2$.

  Thus, there  are four possibilities:
 if exactly two of the parameters are nonzero, then  there are two distinct cases: either these parameters belong to the same surface (i.e., $k_1=k_2=0$ or $\tilde k_1=\tilde k_2=0$), which correspond to the left picture in Figure \ref{boxes1}, or the nonzero parameters belong to two different surfaces, as in the "over cross" situation shown in the picture on the right hand side of Figure \ref{boxes1}. The remaining two possibilities are firstly  that only one parameter $k_1,k_2,\tilde k_1,\tilde k_2$ is nonzero, which happens if only one of the rectangles $U,\tilde U$ touches only one of the axes, and secondly  the situation where both rectangles are located away from the axes. In this last situation, we have indeed no further decomposition and only one term to sum.
 \medskip

The first two of the afore-mentioned possibilities can be  dealt with directly by  the next lemma. But, notice that  the corresponding sums of course dominate the sums over fewer parameters (or even none), which allows to also handle the remaining two possibilities.
\end{proof}

\begin{lemnr}\label{abstractsum}
Let $\mu,\omega\geq0$, $\nu>0$, $n,c_1,c_2\geq0$ such that $(c_1\vee c_2)\mu<\nu+\omega$, and let $a,b\in\R_+$. Then
\begin{align*}
	  \qquad &\sum\limits_{k_1,k_2\in\N} (1+k_1+k_2)^n 2^{-(k_1+k_2)\omega} (a 2^{-k_2c_2}\vee b)^{-\mu}
		(a \vee b 2^{-k_1c_1})^{-\mu} (a2^{-k_1}\wedge b2^{-k_2})^\nu	\\
		\leq&  \sum\limits_{k_1,k_2\in\N} (1+k_1+k_2)^n 2^{-(k_1+k_2)\omega}(a\vee b)^{-\mu} (a 2^{-k_2c_2}\vee b 2^{-k_1c_1})^{-\mu} (a2^{-k_1}\wedge b2^{-k_2})^\nu \\		
		\lesssim& (a\vee b)^{-2\mu}(a\wedge b)^\nu .
\end{align*}		
In the last estimate, the constant hidden by the symbol $\lesssim$  will  depend only  on the exponent  $n$.
	\end{lemnr}
We remark that the bound  in this lemma is essentially sharp, as one  can immediately see by looking at the term with $k_1=0=k_2$. Notice that the proof is easier when  $\omega>0$.

\begin{proof}
To prove the first inequality, observe that
$a2^{-k_2c_2}\vee b2^{-k_1c_1}$ is bounded by $a2^{-k_2c_2}\vee b$ as well as  by $a\vee b2^{-k_1c_1},$  hence by the minimum of these expressions. Therefore we have
\begin{align*}
	(a2^{-k_2c_2}\vee b )\wedge(a\vee b2^{-k_1c_1})
	&\geq a 2^{-k_2c_2}\vee b 2^{-k_1c_1}\\
	(a2^{-k_2c_2}\vee b )\vee(a\vee b2^{-k_1c_1}) &= a \vee b  ,
\end{align*}
hence
\begin{align*}
	(a2^{-k_2c_2}\vee b )(a\vee b2^{-k_1c_1})
	\geq (a \vee b)(a 2^{-k_2c_2}\vee b 2^{-k_1c_1}).
\end{align*}
Using the symmetry in this estimate, it suffices to estimate
\begin{align*}
	S= a^\nu \sum_{\underset{a2^{-k_1}\leq b2^{-k_2}}{k_1,k_2\in\N}} k_1^nk_2^n 2^{-(k_1+k_2)\omega}
						(a 2^{-k_2c_2}\vee b 2^{-k_1c_1})^{-\mu}2^{-k_1\nu}.
\end{align*}
On the one hand, we have
\begin{align*}
	S\leq& a^\nu b^{-\mu} \sum_{k_1} k_1^n 2^{k_1(c_1\mu-\nu-\omega)}
						\sum\limits_{k_2:a2^{-k_1}\leq b2^{-k_2}} k_2^n 2^{-k_2\omega} \\
	\leq& a^\nu b^{-\mu} \log^{n+1}\left(\frac{a}{b}+\frac{b}{a}\right) \sum_{k_1} k_1^{2n+1} 2^{k_1(c_1\mu-\nu-\omega)}\\
	\lesssim& a^\nu b^{-\mu} \log^{n+1}\left(\frac{a}{b}+\frac{b}{a}\right).
\end{align*}
In the case $\omega>0$, we might get along even without the $\log$-term.
On the other hand,
\begin{align*}
	S&\leq a^{\nu-\mu} \sum_{k_2} k_2^n 2^{k_2(c_2\mu-\omega)} \sum_{k_1:a2^{-k_1}\leq b2^{-k_2}} k_1^n 2^{-k_1(\nu+\omega)} \\
	&\leq a^{\nu-\mu} \sum_{k_2} k_2^n 2^{k_2(c_2\mu-\omega)} \sum_{k_1:a2^{-k_1}\leq b2^{-k_2}} k_1^n 2^{-k_1\nu} \\
	&\sim  a^{\nu-\mu} \log^n\left(\frac{a}{b}+\frac{b}{a}\right)
				 \left(\frac{b}{a}\right)^\nu \sum_{k_2} k_2^{2n} 2^{k_2(c_2\mu-\nu-\omega)}	\\
	&\sim  a^{-\mu} b^\nu \log^n\left(\frac{a}{b}+\frac{b}{a}\right).	
\end{align*}
Combining these two estimates, we obtain
\begin{align*}
	\frac{S}{\log^{n+1}\left(\frac{a}{b}+\frac{b}{a}\right)}
	\lesssim& a^{-\mu} b^\nu \wedge a^\nu b^{-\mu} 		
	=(a\vee b)^{-\mu} (a\wedge b)^\nu.
\end{align*}
\end{proof}


\section{Passage from  bilinear to  linear estimates}

Recall that $\bar m=m_1\vee m_2, \, m=m_1\wedge m_2$ and ${1}/{h}={1}/{m_1}+{1}/{m_2}$.
The first step to prove our main Theorem \ref{mainthm} is the following Lorentz space estimate for the adjoint restriction operator $R^*$ associated to $\Gamma=\graph(\phi).$

\begin{thmnr}\label{withouttriangle}
Let $p_0=1+{\bar m}/(\bar m+m)$, $2p>\max\{{10}/{3},2p_0,h+1\}$ and ${1}/{s'}\geq (h+1)/{2p}$.
Then $R^*$ is bounded from $L^{s,t}(\Gamma,d\nu)$ to $L^{2p,t}(\R^3)$ for any $1\leq t\leq\infty$.
\end{thmnr}

\begin{proof}
We begin by observing that we may assume that
\begin{align}\label{steintomas}
\frac {h+1}p>1.
	\end{align}
Indeed,  if $2p\ge 2(h+1),$ then we have the Stein-Tomas type result that $R^*$ is bounded from $L^2(\Gamma,d\nu)$ to $L^{2p}(\R^3)$ (see \cite{ikm}, \cite{IM-uniform}). Interpolating this with the trivial estimate from $L^1(\Gamma,d\nu)$ to $L^{\infty}(\R^3)$ and  applying H\"older's inequality on $\Gamma,$ we see that the situation where $(h+1)/p\le 1$ is settled in Theorem \ref{withouttriangle}.
\medskip

In the remaining cases,
interpolation  theory for Lorentz spaces  (see, e.g., \cite{Gra})  shows that it suffices to prove the restricted  weak-type estimate
\begin{align}\label{restweaktype}
	\|\widehat{\chi_\Omega d\nu}\|_{2p} \lesssim |\Omega|^{\frac{1}{s}}
\end{align}
for any measurable set $\Omega\subset Q=[0,1]\times[0,1]$.

To this end we perform the kind of Whitney decomposition mentioned in Section \ref{dyadicsum} of  $Q\times Q = \bigcup\limits_j \bigcup\limits_{k\approx\tilde k} \tau_{jk}\times\tau_{j\tilde k}$  into ``well-separated neighboring rectangles''
$\tau_{jk}$ and  $\tau_{j\tilde k},$ where
$$\tau_{jk}=[(k_1-1)2^{-j_1},k_12^{-j_1}]\times[(k_2-1)2^{-j_2},k_22^{-j_2}],
$$
and  where $k\approx\tilde k$ means  that $2\leq|k_i- \tilde k_i|\leq C$, $i=1,2$ (compare \cite{lee05},\cite {v05}).  Then we may estimate
\begin{align*}
	\|\widehat{\chi_\Omega d\nu}\|_{2p}^2
	= \|\widehat{\chi_\Omega d\nu}\widehat{\chi_\Omega d\nu}\|_p	
	\leq \sum_j\big[\sum_{k\sim\tilde k} \|\FT(\chi_{\Omega\cap\tau_{jk}}d\nu)
	\FT(\chi_{\Omega\cap\tau_{j\tilde k}}d\nu)\|_p^{p^*} \big]^\frac{1}{p^*},
\end{align*}
where
\begin{align}\label{pstardef}
	p^*=\min\{p,p'\},
\end{align}
with ${1}/{p}+{1}/{p'}=1$.
The last step can be obtained  by interpolation between  the cases $p=2$, where one may apply Plancherel's theorem, and  the cases $p=1$ and  $p=\infty,$ which are simply treated by means of the  triangle inequality (compare Lemma 6.3. in \cite {TV1}).
We claim that
\begin{align}\label{24mar1721}
	(\bar m+3)\left(\frac{1}{p}-\frac{1}{2}\right) < \frac{h+1}{2p}.
\end{align}
{\bf Case 1: $\bar m\leq 2m.$}
Then $\bar m \leq 3h$ and
\begin{align*}
	(\bar m+3)\left(\frac{1}{p}-\frac{1}{2}\right)-\frac{h+1}{2p}
	\leq 3(h+1)\left(\frac{1}{p}-\frac{1}{2}\right)-\frac{h+1}{2p}
	= (h+1)\left(\frac{5}{2p}-\frac{3}{2}\right) < 0
\end{align*}
according to our assumptions.
\medskip

\noindent {\bf Case 2: $\bar m> 2m.$}
Here,
\begin{align*}
	h+1 = \frac{\bar mm+\bar m+m}{\bar m+m}>\frac{\bar mm+3m}{\bar m+m} =(\bar m+3)\frac{m}{\bar m+m},
\end{align*}
and thus
\begin{align*}
  &(\bar m+3)\left(\frac{1}{p}-\frac{1}{2}\right)-\frac{h+1}{2p}		
  < (\bar m+3)\left(\frac{1}{p}-\frac{1}{2}-\frac{m}{\bar m+m}\frac{1}{2p}\right)		\\
  &= (\bar m+3)\left(\frac{1}{2p}\left(1+\frac{\bar m}{\bar m+m}\right)-\frac{1}{2}\right) < 0,
\end{align*}
because of  our assumption $2p>2p_0$.

\medskip
In both cases, these estimates show that we may  choose $q\ge 2 $ such that
\begin{align}\label{defiofq}
	(\bar m+3)\left(\frac{1}{p}-\frac{1}{2}\right) < \frac{1}{q'} < \frac{h+1}{2p}
\end{align}
(recall  here \eqref{steintomas}, which allows to choose $q\ge 2$).

The  first inequality  allows to apply Theorem \ref{nextBLthm} to the pair of hypersurfaces  $S_{jk}=\{(\xi,\phi(\xi)):\xi\in \tau_{jk}\}$ and $ S_{j\tilde k}=\{(\xi,\phi(\xi)):\xi\in \tau_{j\tilde k}\},$ with
\begin{align}\nonumber
	\rho_i=&2^{-j_i} 	\\
	\varkappa_i\sim & (k_i2^{-j_i})^{m_i-2}\sim  (\tilde k_i2^{-j_i})^{m_i-2} \label{16apr1456} \\
	\varkappa_i \rho_i^2\sim & k_i^{m_i-2}2^{-j_im_i}. \nonumber
\end{align}
Without loss of generality, we may assume that
\begin{align}\label{setofks}
	k\in I:=\{k|k_1^{m_1-2}2^{-j_1m_1}\geq k_2^{m_2-2}2^{-j_2m_2}\},
\end{align}
i.e., $\varkappa_1 \rho_1^2 \geq \varkappa_2 \rho_2^2$.
Thus, by  Theorem \ref{nextBLthm},
\begin{align*}
& \| R^*_{S_{jk}, S_{j\tilde k}}\|_{L^q(S_{jk})\times L^q(S_{j\tilde k})\to L^p(\R^3)} \\
	\lesssim& (\rho_1\rho_2)^{\frac{2}{q'}-\frac{1}{p}} (\varkappa_1\rho^2_1\vee\varkappa_2\rho^2_2)^{-\frac{1}{p}}
	\left(\frac{\varkappa_1\rho^2_1\vee\varkappa_2\rho^2_2}
	 		 	{\varkappa_1\rho^2_1\wedge\varkappa_2\rho^2_2}\right)^{\frac{2}{p}-1+\e}	\\
	=&	2^{-(j_1+j_2)\left(\frac{2}{q'}-\frac{1}{p}\right)} k_1^{-(m_1-2)\frac{1}{p}} 2^{j_1m_1\frac{1}{p}}
		\left(\frac{k_1^{m_1-2}}{k_2^{m_2-2}} 2^{-(j_1m_1-j_2m_2)}\right)^{\frac{2}{p}-1+\e}\\
	= &A_j\cdot B_{k,j}^{\frac{1}{p}} ,
\end{align*}
if we define
\begin{align*}
	&A_j = 2^{-(j_1+j_2)\left(\frac{2}{q'}-\frac{1}{p}\right)} 2^{j_1m_1\frac{1}{p}}	2^{-(j_1m_1-j_2m_2)\frac{2}{p}-1+\e}, \\
	&B_{j,\tilde k}\sim  B_{k,j}=
					k_1^{-(m_1-2)} \left(\frac{k_1^{m_1-2}}{k_2^{m_2-2}} \right)^{2-p+\e p}.
\end{align*}
Since $|\{\tilde k: k\sim\tilde k\}|\lesssim 1$ for fixed $k$,
we conclude that
\begin{eqnarray*}
	\|\widehat{\chi_\Omega d\nu}\|_{2p}^2
	&\lesssim& \sum_j A_j\left[\sum_{k\sim\tilde k} \left(B_{k,j}^{\frac{1}{p}}
	 |\Omega\cap\tau_{jk}|^\frac{1}{q}|\Omega\cap\tau_{j\tilde k}|^\frac{1}{q}\right)^{p^*}
	 							 \right]^\frac{1}{p^*}\\
	&\lesssim& \sum_j A_j\left[\sum_{k} B_{k,j}^{\frac{p^*}{p}} |\Omega\cap\tau_{jk}|^{\frac{2p^*}{q}} \right]^\frac{1}{p^*}.
\end{eqnarray*}

Therefore we are reduced  to showing  that
\begin{align}\label{keyest1}
	\sum_j A_j\left[\sum_{k} B_{k,j}^{\frac{p^*}{p}} |\Omega\cap\tau_{jk}|^{\frac{2p^*}{q}}\right]^\frac{1}{p^*}	
	\lesssim& |\Omega|^{\frac{2}{s}}.
\end{align}

\subsection{ Further reduction}

We decompose
$$\frac{2p^*}{q}=\frac{\alpha}{r^*}+\frac{1}{{r^*}'},
$$
where $r*\in[1,\infty]$ will be determined later,
 and introduce $r={r^*p^*}/{p}$. Applying H\"older's inequality to the summation in $k,$ with Hölder exponent $r^*\geq1,$  we get
\begin{eqnarray}\label{hoelder}
	 \left(\sum_{k\in I} B_{k,j}^{\frac{p^*}{p}} |\Omega\cap\tau_{jk}|^{\frac{2p^*}{q}} \right)^\frac{1}{p^*}	
	&\leq& \left(\sum_{k\in I} B_{k,j}^{\frac{p^*r^*}{p}}|\Omega\cap\tau_{jk}|^\alpha\right)^\frac{1}{p^*r^*}
		~ \left(\sum_{k\in I} |\Omega\cap\tau_{jk}| \right)^\frac{1}{p^*{r^*}'} \nonumber\\
	&\leq& \left(\sum_{k\in I} B_{k,j}^{r}\right)^\frac{1}{pr}	~ \min\{|\Omega|,2^{-j_1-j_2}\}^\frac{\alpha}{p^*r^*} ~
		 |\Omega|^{\frac{1}{p^*}\left(1-\frac{1}{r^*}\right)}	\nonumber\\
	&=& \left(\sum_{k\in I} B_{k,j}^{r} \right)^\frac{1}{pr}
		~ \min\{|\Omega|,2^{-j_1-j_2}\}^{\frac{2}{q}-\frac{1}{p^*}+\frac{1}{pr}} ~ |\Omega|^{\frac{1}{p^*}-\frac{1}{pr}}.\nonumber	
\end{eqnarray}

Moreover we have $|\Omega|\leq|Q|=1,$ as well as $1/{s'}\geq {h+1}/{2p}$, i.e., $2-(h+1)/{p}\geq {2}/{s}$. Therefore $|\Omega|^{2-\frac{h+1}{p}}\leq |\Omega|^{\frac{2}{s}}$, and thus in order to prove \eqref{keyest1}, it suffices to show that
$$
	|\Omega|^{2-\frac{h+1}{p}-\frac{1}{p^*}+\frac{1}{pr}}
	\gtrsim \sum_j A_j	\min\{|\Omega|,2^{-j_1-j_2}\}^{\frac{2}{q}-\frac{1}{p^*}+\frac{1}{pr}}
						\left[\sum_{k\in I} B_{k,j}^{r} \right]^\frac{1}{rp}, 	
$$
i.e., that
\begin{align*}
	&|\Omega|^{2-\frac{h+1}{p}-\frac{1}{p^*}+\frac{1}{pr}} \\
	&\gtrsim  \sum_j 2^{-(j_1m_1-j_2m_2)\frac{2}{p}-1+\e}  2^{-(j_1+j_2)\left(\frac{2}{q'}-\frac{1}{p}\right)}  2^{j_1m_1\frac{1}{p}}
						\min\{|\Omega|,2^{-j_1-j_2}\}^{\frac{2}{q}-\frac{1}{p^*}+\frac{1}{pr}}	&\\
	&  \hskip2cm \times	\left[\sum_{k\in I} k_1^{(m_1-2)(1-p+\e p)r} k_2^{(m_2-2)(p-2-\e p)r} \right]^\frac{1}{pr}.
\end{align*}
We apply the change of variables
$l=j_1+j_2\in\N$, $l'=j_1m_1-j_2m_2\in\Z$, such that
$$j_1=\frac{m_2l+l'}{m_1+m_2}.$$
 Then the exponent in $j_1,j_2$ becomes
\begin{eqnarray*}
	& & (j_1m_1-j_2m_2)\left(1-\frac{2}{p}-\e\right) + (j_1+j_2)\left(\frac{1}{p}-\frac{2}{q'}\right)+j_1m_1\frac{1}{p}	\\
	&=& l'\left(1-\frac{2}{p}-\e\right) + l\left(\frac{1}{p}-\frac{2}{q'}\right)
				+\frac{m_1m_2l+m_1l'}{m_1+m_2}\frac{1}{p}	\\
	&=& \frac{l'}{p}\left(p-\e p-\frac{m_1+2m_2}{m_1+m_2}\right) + l\left(\frac{h+1}{p}-\frac{2}{q'}\right).\end{eqnarray*}
The summation over  $k\in I_{l'}=\{k_1^{m_1-2}\geq k_2^{m_2-2}2^{l'}\}$ is independent of $l$, and thus we have  finally reduced the proof of \eqref{keyest1} to   proving the following two  decoupled estimates:\footnote
{Technically, we only have to sum over the smaller set $l'\in m_1\N -m_2\N$.}
\begin{align}\label{suminl'}
	\sum_{l'=-\infty}^\infty 2^{\frac{l'}{p}(p-\e p-\frac{m_1+2m_2}{m_1+m_2})}
				\left[\sum_{k\in I_{l'}} k_1^{(m_1-2)(1-p+\e p)r} k_2^{(m_2-2)(p-2-\e p)r} \right]^\frac{1}{pr}
	<\infty
\end{align}
and
\begin{align}\label{suminl}
	\sum_{l=0}^\infty 2^{l(\frac{h+1}{p}-\frac{2}{q'})} \min\{|\Omega|,2^{-l}\}^{\frac{2}{q} -\frac{1}{p^*}+\frac{1}{pr}}
	\lesssim |\Omega|^{2-\frac{h+1}{p}-\frac{1}{p^*}+\frac{1}{pr}}.
\end{align}

\subsection{The case $m>2$}
In this case we have both $m_1>2$ and $m_2>2$. 

\subsubsection{Summation in $k$}
We compare the sum over $k$ in \eqref{suminl'} with an integral. We claim that
\begin{align}\label{sumink}
	&\lefteqn{\int\limits_{\underset{k_1^{m_1-2}\geq k_2^{m_2-2}2^{l'}}{k_1,k_2\geq1}}
			k_1^{(m_1-2)(1-p+\e p)r} k_2^{(m_2-2)(p-\e p-2)r}d k} \nonumber\\
  & \hskip1.5cm\lesssim \begin{cases}  2^{|l'|(\frac{1}{m_1-2}+(1-p+\e p)r)}, & l'\geq0\\
	|l'|2^{|l'|(\frac{1}{m_2-2}+(p-\e p-2)r)_+}, & l'<0, 	
				 \end{cases},
\end{align}
provided that
\begin{align}
		a:= \frac{1}{m_1-2}+(1-p+\e p)r <0						\label{one}
\end{align}
and
\begin{align}
		a+b =\frac{1}{m_1-2}+\frac{1}{m_2-2}-r<0,		\label{two}
\end{align}
where
\begin{align}
b&:=\frac{1}{m_2-2}+(p-\e p-2)r	\in \R.																		\label{six}
\end{align}

For the moment, we will simply assume that these conditions hold true. We shall collect several further  conditions on the exponent $r$ and verify at the end of this section  that we can indeed find an $r$ such that all  these conditions are satisfied.
\medskip

By means of the coordinate transformation $s=k_1^{m_1-2}$, $t=k_2^{m_2-2}$ (i.e., $d k \sim  s^{\frac{1}{m_1-2}-1}t^{\frac{1}{m_2-2}-1}d(s,t)$\,), \eqref{sumink} simplifies to  showing that
\begin{align}\label{keyest2}
	J(a,b)=\underset{\underset{\scriptstyle s\geq t2^{l'}}{s,t\geq1}}{\int\int}
			s^a t^b	\frac{d s}{s}\frac{d t}{t}
	\lesssim \begin{cases}  2^{|l'|a}  & l'\geq0	\\  2^{|l'|b_+} & l'<0,
					\end{cases},
\end{align}
provided that $a<0$, $a+b<0$. Here we have set $b_+=b\vee 0.$ Changing $t'=s2^{-l'}/t$, the set of integration for the $t$-variable $\{t: t\geq 1, s2^{-l'}/t\geq1\}$ transforms into $\{t':s2^{-l'}/t'\geq 1, t'\geq1\},$ and thus, since we assume that $a+b<0,$
\begin{eqnarray*}
	J(a,b) &=& \underset{\underset{\scriptstyle s\geq t'2^{l'}}{s,t'\geq1}}{\int\int}
			s^a \left(\frac{s}{t'}2^{-l'}\right)^b \frac{d s}{s} \frac{d t'}{t'}
	= 2^{-l'b} \int_{t'=1}^\infty t'^{-b}
						\int_{s=1\vee t'2^{l'}}^\infty s^{a+b}\frac{d s}{s}\frac{d t'}{t'}	\\
	&=& 2^{-l'b} \int_{1}^\infty (1\vee t'2^{l'})^{a+b}
												 t'^{-b} \frac{d t'}{t'}	. 												
\end{eqnarray*}
If $l'\geq0,$ then clearly $1\vee t'2^{l'}=t'2^{l'}$, and since $a<0,$ we get
\begin{align*}
	J(a,b) = 2^{l'a} \int_1^\infty t'^{a}\frac{d t' }{t'}\sim  2^{l'a}.
\end{align*}
And, if  $l'<0$, the we split into
\begin{align*}
	J(a,b) =& 2^{-l'b} \int_1^{2^{|l'|}} t'^{-b}  \frac{d t'}{t'}
						+ 2^{l'a} \int_{2^{|l'|}}^\infty t'^a \frac{d t'}{t'}		
	= 	2^{-l'b} \frac{1-2^{l'b}}{b}
						+ \int_{1}^\infty u^a \frac{d u}{u}\\
	\lesssim& |l'|(2^{|l'|b}+1)		
	\sim  |l'| 2^{|l'|b_+} 							
\end{align*}
(notice that the additional factor $|l'|$ arises  in fact  only  when $b=0$).
This proves \eqref{keyest2}.

\subsubsection{Summation in $l'$}

In order to apply \eqref{sumink} to  \eqref{suminl'}, we split the sum in \eqref{suminl'} into summation over $l'\geq0$ and summation over $l'<0$.
In the first case $l'\geq0$, we obtain
\begin{align}\label{suminl'>0}
	& \sum_{l'\geq0} 2^{\frac{l'}{p}(p-\e p-\frac{m_1+2m_2}{m_1+m_2})}
				\left[\sum_{k\in I} k_1^{(m_1-2)(1-p+\e p)r} k_2^{(m_2-2)(p-\e p-2)r}\right]^\frac{1}{rp}	\nonumber\\
	\lesssim& \sum_{l'\geq0} 2^{\frac{l'}{p}(p-\e p-1-\frac{m_2}{m_1+m_2})}
								2^{l'(\frac{1}{pr}\frac{1}{m_1-2} + \frac{1-p+\e p}{p})}	\\
	=&	\sum_{l'\geq0} 2^{\frac{l'}{p}(\frac{1}{r}\frac{1}{m_1-2}-\frac{m_2}{m_1+m_2})}.\nonumber
\end{align}
The sum is finite  provided
\begin{align}
 \frac{1}{r} < \frac{m_2(m_1-2)}{m_1+m_2},		\label{three}
\end{align}
which gives yet another condition for our collection.

In the second case $l'<0$, we have
\begin{align}\label{suminl'<0}
	& \sum_{l'<0} 2^{\frac{l'}{p}(p-\e p-\frac{m_1+2m_2}{m_1+m_2})}
				\left[\sum_{k\in I} k_1^{(m_1-2)(1-p+\e p)r} k_2^{(m_2-2)(p-\e p-2)r}\right]^\frac{1}{rp}	\nonumber\\
	\lesssim& \sum_{l'<0} 2^{\frac{l'}{p}(p-\e p-\frac{m_1+2m_2}{m_1+m_2})}
									|l'| 2^{|l'|(\frac{1}{pr}\frac{1}{m_2-2}+\frac{p-\e p-2}{p})_+}	\\
	=& \sum_{l'<0} |l'| 2^{\frac{l'}{p}(p-\e p-\frac{m_1+2m_2}{m_1+m_2}
													-(\frac{1}{r}\frac{1}{m_2-2}+p-\e p-2)_+)}.	\nonumber								
\end{align}
Notice that for sufficiently  small $\e>0$ we have
$p-\e p> p_0
= 1+{\bar m}/(\bar m+m)
\geq 1+{m_2}/(m_1+m_2)$,
and therefore
\begin{align}\label{6nov1353}
	p-\e p-\frac{m_1+2m_2}{m_1+m_2}>0.
\end{align}
Thus  the  last sum in \eqref{suminl'<0} converges in the case where
$$\frac{b}{r}=\frac{1}{r}\frac{1}{m_2-2}+p-\e p-2\leq0.$$
This shows that we only need to discuss the case  where $b>0$, in which we need that
\begin{align*}
0 <& p-\e p - \frac{m_1+2m_2}{m_1+m_2} -\frac{1}{r}\frac{1}{m_2-2}-p+\e p+2	\\
	=& \frac{m_1}{m_1+m_2} -\frac{1}{r}\frac{1}{m_2-2},
\end{align*}
which is equivalent to
\begin{align}
	 \frac{1}{r} < \frac{m_1(m_2-2)}{m_1+m_2}.		\label{threesym}
\end{align}
Notice that this is of the same form as  \eqref{three}, only with the roles of $m_1$ and $m_2$ interchanged.

\subsubsection{Summation in $l$}
Recall that we want to show estimate \eqref{suminl}, i.e.,
\begin{align*}
	\sum_{l=0}^\infty 2^{l(\frac{h+1}{p}-\frac{2}{q'})}
		 \min\{2^{-l},|\Omega|\}^{\frac{2}{q}-\frac{1}{p^*}+\frac{1}{pr}}
	\lesssim |\Omega|^{2-\frac{h+1}{p}-\frac{1}{p^*}+\frac{1}{pr}}.
\end{align*}
We claim that it is sufficient to show that for $\mu>0$ and $\nu-\mu>0$
\begin{align}\label{intinl}
	\int_0^\infty e^{x\mu} \min\{e^{-x},A\}^{\nu} d x \lesssim A^{\nu-\mu}.
\end{align}
Indeed, given  \eqref{intinl}, we apply it  with  $A=|\Omega|$, $\mu=(h+1)/p-2/{q'}$ and $\nu={2}/{q}-{1}/{p^*}
+1/{pr}$. Due to the choice of $q$ in \eqref{defiofq}, we have $\mu>0$. Moreover we want that
$$0<\nu-\mu=2-\frac{1}{p^*}+\frac{1}{pr}-\frac{h+1}{p} = \frac{1}{p}\left(2p-h-1-\frac{p}{p*}+\frac{1}{r}\right).
$$
Notice that if $p\leq2$, then ${p}/{p*}=1$, but if $p>2$, then
${p}/{p*}=p\left(1-{1}/{p}\right)=p-1$.
Thus ${p}/{p*}=1+(p-2)_+$ for all $1\leq p\leq\infty$, i.e., the condition which is required here is
\begin{align}
	 \frac{1}{r}>h+2-2p+(p-2)_+ \label{four}.
\end{align}

In order to verify  \eqref{intinl}, observe that
\begin{align*}
	\int_0^\infty e^{x\mu} \min\{e^{-x},A\}^{\nu} d x
	=& \int_{\ln A}^\infty e^{y\mu}A^{-\mu} \min\{e^{-y}A,A\}^{\nu} d y  \\
	=& A^{\nu-\mu} \int_{\ln A}^\infty e^{y\mu} \min\{e^{-y},1\}^{\nu} d y.
\end{align*}
The last  integral can be estimated by
\begin{align*}
	\int_{\ln A}^\infty e^{y\mu} \min\{e^{-y},1\}^{\nu} d y
	\leq \int_{-\infty}^0 e^{y\mu} d y + \int_0^\infty e^{-y(\nu-\mu)} d y,
\end{align*}
which is convergent since $\mu>0$ and $\nu-\mu>0$.

It still remains to  be checked whether there exists some  $1\leq r^*<\infty$  (for m>2) for which $r$ satisfies the conditions
\eqref{one}, \eqref{two}, \eqref{three}, \eqref{threesym} and \eqref{four}.

This task will be accomplished in Lemma \ref{choosingr}. First, we discuss the situation where $m=2$.


\subsection{The case $m=2$}
We will just give some hints how to modify the previous proof for this situation. In this case, $r=\infty$ turns out to be an appropriate choice, and the inequality that we need to start the argument with here reads
\begin{align*}
	 \left(\sum_{k\in I} B_{k,j}^{\frac{p^*}{p}} |\Omega\cap\tau_{jk}|^{\frac{2p^*}{q}} \right)^\frac{1}{p^*}	
	\leq \big(\sup_{k\in I} B_{k,j}\big)^{\frac 1p} \min\{|\Omega|,2^{-j_1-j_2}\}^{\frac{2}{q}-\frac{1}{p^*}} ~ |\Omega|^{\frac{1}{p^*}}.
\end{align*}
This is very easy to prove, provided ${2p^*}/{q}\ge1$ (notice that this condition corresponds to our previous decomposition of  ${2p^*}/{q}$  when $r=\infty)$.  To see that indeed ${2p^*}/{q}\ge1,$ recall from \eqref{defiofq}  that $1/{q'}<(h+1)/{2p}.$ Then, it is enough to check  that $2p^*\left(1-(h+1)/{2p}\right)>1,$ i.e.,
$h+1-2p+p/{p^*}<0.$ The last condition  is equivalent  to $h+2-2p+(p-2)_+<0.$ However, this is what we shall indeed verify in the proof of Lemma \ref{choosingr} (compare  estimate \eqref{goodest} when $m=2$). 	

\medskip
	
Observe next that we may re-write the integral in $\eqref{sumink}$ in terms of the $L^r$-norm as
\begin{align*}
	\left\| \left( k_1^{(m_1-2)(1-p+\e p)} k_2^{(m_2-2)(p-\e p-2)},	 \right)_{k\in I_{l'}} \right\|_{r}
	\lesssim \begin{cases}  2^{|l'|(\frac{1}{r}\frac{1}{m_1-2}+1-p+\e p)} & l'\geq0,	\\
						|2^{|l'|(\frac{1}{r}\frac{1}{m_2-2}+p-\e p-2)_+}, & l'<0, 	
					\end{cases}
\end{align*}
provided the conditions \eqref{one} and \eqref{two} hold true, i.e.,  that
\begin{align*}
	 \frac{1}{m_1-2}+(1-p+\e p)r <0 \quad\mbox{and}\quad  \frac{1}{m_1-2}+\frac{1}{m_2-2}-r<0.
\end{align*}
This gives rise to the conjecture that  (for $r=\infty$) we should have
\begin{align}\label{supink}
	  \sup_{k\in I}  k_1^{(m_1-2)(1-p+\e p)}	& k_2^{(m_2-2)(p-\e p-2)} \\
	\leq \sup_{s\geq t2^{l'}}  s^{1-p+\e p} t^{p-\e p-2}					
	 &\lesssim\ \begin{cases}  2^{|l'|(1-p+\e p)} & l'\geq0	\\
			  2^{|l'|(p-\e p-2)_+} & l'<0, \nonumber	
					\end{cases}
\end{align}
which  would suffice in this case.
But notice that the conditions \eqref{one} and \eqref{two} are formally fulfilled for $r=\infty,$ and
it  is then  easy to check that \eqref{supink} indeed holds true, even in the case $m=2.$

\subsubsection{Summation in $l'$}
The summation in $l'$ becomes simpler here. We split again into the sums over $l'\geq0$ and $l'<0$, and obtain for the first half of the sum in \eqref{suminl'>0}
\begin{align*}
	\sum_{l'\geq0} 2^{\frac{l'}{p}(p-\e p-\frac{m_1+2m_2}{m_1+m_2})}
								2^{l'\frac{1-p+\e p}{p}}	
	=	\sum_{l'\geq0} 2^{-\frac{l'}{p}\frac{m_2}{m_1+m_2}} <\infty.
\end{align*}
The second part of the sum becomes (compare \eqref{suminl'<0})
\begin{align*}
	\sum_{l'<0} 2^{\frac{l'}{p}(p-\e p-\frac{m_1+2m_2}{m_1+m_2}
													-(p-\e p-2)_+)}.							\end{align*}
We already know from \eqref{6nov1353} that
$p-\e p-(m_1+2m_2)/(m_1+m_2)>0$. Thus the sum converges if 
$p-\e p\leq2$. For $p-\e p>2$, notice that
\begin{align*}
	  p-\e p-\frac{m_1+2m_2}{m_1+m_2}-(p-\e p-2)_+	
	= \frac{m_1}{m_1+m_2}	>0,
\end{align*}
and thus the sum is finite.

\subsubsection{Summation in $l$}
It remains to show that
\begin{align*}
	\qquad \sum_{l=0}^\infty 2^{l(\frac{h+1}{p}-\frac{1}{q'})}
		 	\min\{|\Omega|,2^{-l}\}^{\frac{1}{q}-\frac{1}{p^*}}	\lesssim |\Omega|^{2-\frac{h+1}{p}-\frac{1}{p^*}},
\end{align*}
which is the special case $r=\infty$ of \eqref{suminl}. We saw that this holds true provided \eqref{four} is valid, i.e., if  $1/{r}>h+2-2p+(p-2)_+$.

However, if $m=2,$ the
$$2p > p_0
= \frac{2\bar m}{\bar m+2}+2= h+2.$$
Thus for the case $p\leq2$ we have $h+2-2p+(p-2)_+ = h+2-2p <0$. For the case $p>2$ notice that
$$h+2-2p+(p-2)_+ = h-p = \frac{2\bar m}{\bar m+2} -p < 2-p <0.$$
~\end{proof}
~

\subsection{Final considerations}

We finally verify that there is indeed always some $r$ for which  all necessary  conditions  \eqref{one}, \eqref{two}, \eqref{three}, \eqref{threesym} and \eqref{four} are satisfied in the   case $m>2$. Recall  that
\begin{align*}
	\frac{2p^*}{q}=\frac{\alpha}{r^*}+\frac{1}{{r^*}'},
\end{align*}
and notice that it will suffice to verify the following equivalent inequalities:
\begin{align}
	& \frac{1}{r} < (m_1-2)(p-1),	\label{one'}\\
	& \frac{1}{r} < \frac{(m_1-2)(m_2-2)}{m_1+m_2-4},\label{two'}\\
	& \frac{1}{r} < \frac{m_2(m_1-2)}{m_1+m_2},\label{three'}\\
	& \frac{1}{r} < \frac{m_1(m_2-2)}{m_1+m_2},\label{threesym'}\\
	&  \frac{1}{r}>h+2-2p	+(p-2)_+.	\label{four'}
\end{align}
\begin{lemnr}\label{choosingr}
Assume  that  $m>2$ and  $2p>\max\{2p_0,h+1\},$ where we recall that $p_0=1+{\bar m}/(\bar m+m).$ Define
\begin{align*}
	J=\Big]0,1+(p-2)_+\Big]\cap\Big]h+2-2p+(p-2)_+,\frac{\bar m(m-2)}{\bar m+m}\Big[.
\end{align*}
Then $J\neq\emptyset,$ and for every $1/{r}\in J$ we have
\begin{align}
	r^* =\frac{rp}{p^*}\geq 1, \quad \mbox{and} \quad
	\alpha = r^*\left(\frac{2p^*}{q}-\frac{1}{{r^*}'}\right) >0, \label{14apr1411}
\end{align}
and moreover the inequalities \eqref{one'}, \eqref{two'}, \eqref{three'}, \eqref{threesym'} and \eqref{four'}  are valid.
\end{lemnr}

\begin{proof}
First of all, we will show that $J\neq\emptyset$. We need to see that
\begin{align}\label{goodest}
h+2-2p+(p-2)_+<\frac{\bar m(m-2)}{\bar m+m}=h-\frac{2\bar m}{\bar m+m},
\end{align}
 i.e., that $2p_0=2+2\bar m/{\bar m+m} < 2p-(p-2)_+$. For the case $p\leq 2$, this holds true since $2p> 2p_0$.
If $p>2$, observe that
\begin{align*}
			h+2-2p+(p-2)_+
	= h-p
	< h-2
	\leq h-\frac{2\bar m}{\bar m +m}
	= \frac{\bar m(m-2)}{\bar m+m}.
\end{align*}
Thus both intervals used for the definition of $J$ are not empty, but we still have to check that their intersection is not trivial.
Since we assume that  $2p>h+1$, we have
$$h+2-2p+(p-2)_+<1+(p-2)_+.
$$
And, for  $m>2$, we also have $0<\bar m(m-2)/( \bar m+m)$,  which shows that  $J\neq\emptyset$.

\medskip
Next, if $1/{r}\in J$, then in particular ${1}/{r}\leq1+(p-2)_+={p}/{p^*},$ and thus $r^*=r{p}/{p^*}\geq1$.
To prove \eqref{14apr1411}, observe that due to our choice of $q$ in \eqref{defiofq} we have $1/{q}>1-(h+1)/{2p}$, and thus it suffices to prove that
\begin{align*}
	2p^*\left(1-\frac{h+1}{2p}\right)>\frac{1}{{r^*}'}=1-\frac{p^*}{rp}.
\end{align*}
This  inequality is equivalent to
\begin{align*}
	\frac{1}{r}>\frac{p}{p*}+h+1-2p = h+2-2p+(p-2)_+,
\end{align*}
and thus satisfied.

Considering the remaining conditions listed before the statement of the lemma, notice that  \eqref{four'} is immediate  by the definition of $J$.
Furthermore we have
\begin{align*}
	\frac{1}{r} < \frac{\bar m(m-2)}{\bar m+m} = \frac{m_1m_2-2\bar m}{m_1+m_2} \leq \frac{m_1m_2-2m_i}{m_1+m_2}
\end{align*}
for both $i=1,2$, which gives \eqref{three'} and \eqref{threesym'}. To obtain \eqref{one'}, we estimate
\begin{align*}
	\frac{1}{r} < \frac{\bar m(m-2)}{\bar m+m}
	\leq \frac{\bar m(m_1-2)}{\bar m+m}
	=  (p_0-1)(m_1-2) < (p-\e p-1)(m_1-2).
\end{align*}
Finally, observe that the following inequalities  are all equivalent:
\begin{align*}
\frac{\bar m(m-2)}{\bar m+m} &\leq \frac{(m_1-2)(m_2-2)}{m_1+m_2-4},\\
\frac{\bar m}{\bar m+m} &\leq \frac{\bar m-2}{\bar m+m-4},\\
\bar m(\bar m+m)-4\bar m &\leq \bar m(\bar m+m)-2(\bar m+m),\\
 m &\leq \bar m.
\end{align*}
Hence \eqref{two'} holds true as well.
\end{proof}

\subsection{Finishing the proof}

We can now conclude  the prove of our main result, Theorem \ref{mainthm}:

\begin{kornr}\label{finish}
Let $2p>\max\{{10}/{3},h+1\}$, ${1}/{s'}\geq (h+1)/{2p}$ and $1/{s}+(2\bar m+1)/{2p}<(\bar m+2)/{2}$.
Then $R^*$ is bounded from $L^{s,t}(\Gamma)$ to $L^{2p,t}(\R^3)$ for every  $1\leq t\leq\infty$.
If moreover $s\leq 2p$ or $1/{s'}>(h+1)/{2p}$, then $R^*$ is bounded from $L^{s}(\Gamma)$ to $L^{2p}(\R^3)$.
\end{kornr}

\begin{proof}
The  crucial observation is that the intersection point of the  two lines
$$\frac{1}{s'}=\frac{h+1}{2p} \quad \mbox{and} \quad \frac{2\bar m+1}{2p}+\frac{1}{s}=\frac{\bar m+2}{2}
$$
has the $p$-coordinate  $p=\widehat{p_0}=1+{\bar m}/(\bar m+m)$ (comparing  with \eqref{p0},  notice that  $\widehat{p_0}=p_0/2$).
So, what remains is to establish  estimates for $R^*$ for the  missing points $(1/s,1/p)$ lying  within the sectorial region defined by the conditions $(2\bar m+1)/{2p}+1/{s}<(\bar m+2)/{2}$ and  $1/p>1/\widehat{p_0}$ (the region above  the horizontal threshold line  $1/p=1/p_0$ from Theorem \ref{withouttriangle} (cf. Figure \ref{trick1}).

\medskip
Notice also  that if $m\ge \bar m/2,$ then $\widehat{p_0}\le 5/3,$ i.e., $p_0\le 10/3,$  and hence the condition $1/{s}+(2\bar m+1)/{2p}<(\bar m+2)/{2}$ becomes redundant.

Moreover, the  condition $1/s+(2\bar m+1)/{2p}<(\bar m+2)/{2}$ does only depend on $\bar m$, but not on $m$, whereas the condition $1/{s'}=(h+1)/{2p}$ depends on the height $h$, i.e., on both $m_1$ and $m_2.$

\begin{figure}[ht]\begin{center}
  \includegraphics{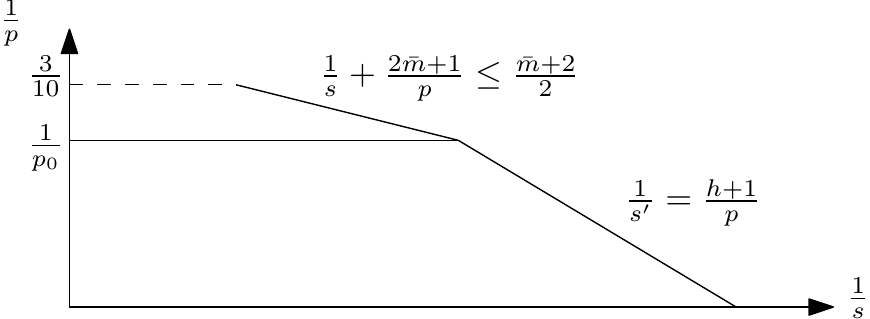}
\caption{Range of $p$ and $q$ in Theorem \ref{mainthm}}
\label{trick1}
\end{center}\end{figure}
\medskip

This leads to the following {\it heuristic idea}:
Assume we fix $\bar m$ and  consider a  family of  surfaces $\Gamma_{\bar m,m^\sharp}$  corresponding to exponents   $m_1=\bar m$ and $m_2=m^\sharp$ for different exponents
$m<m^\sharp$ such that  $\Gamma_{\bar m,m}=\Gamma$  (think for instance of the graph of $x_1^{m^\sharp}+x_2^{\bar m}$ for $m^\sharp\ne m$).
Let us then  compare the restriction estimates that we got so far  for the surface $\Gamma=\Gamma_{\bar m,m}$     with the ones for the hypersurfaces  $\Gamma_{\bar m,m^\sharp}.$  Denote by $h$ and $h^\sharp$ the heights of these hypersurfaces.  Then $h<h^\sharp,$  so that the critical line $1/{s'}=(h^\sharp+1)/{2p}$ lies  below the critical line $1/{s'}=(h+1)/{2p}$ for $\Gamma,$ but its intersection point with the   corresponding  horizontal threshold line  $1/p=1/p_0^\sharp,$ where $p_0^\sharp=2+{2\bar m}/{(\bar m+m^\sharp)}<p_0,$ lies above the previous intersection point (cf. Figure \ref{trick2}).

\begin{figure}[ht]\begin{center}
  \includegraphics{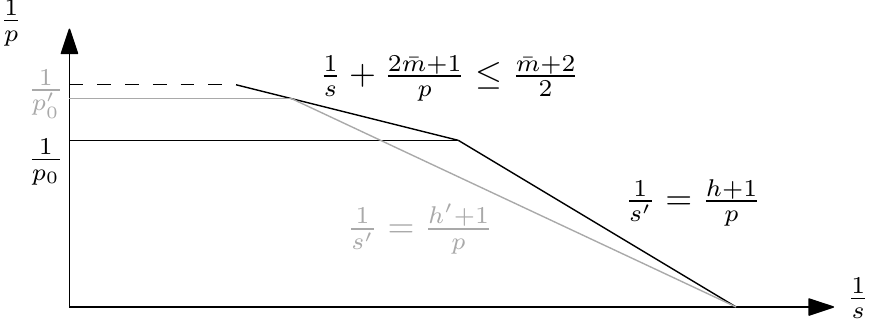}
\caption{Variation of the minimal exponent}
\label{trick2}
\end{center}\end{figure}

This suggests that for our theorem, it should morally be  sufficient to ``increase'' $m^\sharp$ until $\bar m= 2m^\sharp,$ because then we would have $p_0^\sharp=2+{2\bar m}/(\bar m+m^\sharp)={10}/{3}$. In other words, for any point $({1}/{s},1/{2p})$ fulfilling the assumptions of Theorem \ref{mainthm}, we  would find an $m^\sharp\in ]m,\bar m/2]$ such that $({1}/{s},1/{2p})$ satisfies the requirements of Theorem \ref{withouttriangle} corresponding  to the surface $\Gamma_{\bar m,m'}.$  Thus we  would obtain the restriction estimate for the surface $\Gamma_{\bar m,m^\sharp}$ at the point $({1}/{s},1/{2p})$. However, since this surface has ``less curvature''  than $\Gamma_{\bar m,m},$  as $m^\sharp>m,$ the  corresponding restriction inequality should hold true for $\Gamma_{\bar m,m}=\Gamma$ as well.

\medskip
To turn  these heuristics  into a  solid proof, we just need to check that the bound for the bilinear operator that we obtained in Theorem \ref{nextBLthm} is increasing in $m$. Recall that for subsurfaces $S,\tilde S\subset S_{\bar m,m}$ under the assumptions of the afore mentioned theorem we obtained the bound
\begin{align*}
	&\| R^*_{S,\tilde S}\|_{L^s(S)\times L^s(\tilde S)\to L^p(\R^3)}\\
	&\quad \lesssim C_{\bar m,m} := (\rho_1\rho_2)^{\frac{2}{s'}-\frac{1}{p}}
	(\varkappa_1\rho^2_1\vee \varkappa_2\rho^2_2)^{\frac{1}{p}-1+\e }
		(\varkappa_1\rho^2_1\wedge \varkappa_2\rho^2_2)^{1-\frac{2}{p}-\e},
\end{align*}
which we apply to $\rho_i=2^{-j_i}$ and $\varkappa_i= (k_i2^{-j_i})^{m_i-2}$
(cf. \eqref{16apr1456}). If we denote by $\rho_i^\sharp,\varkappa_i^\sharp$ the corresponding quantities  associated to the exponents  $\bar m$ and $m^\sharp,$  then clearly $\rho^\sharp_i=\rho_i$ and $\varkappa_i^\sharp\leq \varkappa_i.$ Since we seek to extend the range of validity of  Theorem \ref{withouttriangle}, we may assume
 that  $2p\leq p_0 < 4$,  and moreover  that $2p\geq p_0^\sharp>2.$ Then we have $1/{p}-1+\e <0$ and
 ${1-2/{p}-\e}<0$ for sufficient small $\e>0$, and hence
\begin{align*}
	C_{\bar m,m} \leq
	(\rho^\sharp_1\rho^\sharp_2)^{\frac{2}{s'}-\frac{1}{p}}
	(\varkappa_1^\sharp(\rho^\sharp_1)^2\vee \varkappa^\sharp_2(\rho^\sharp_2)^2)^{\frac{1}{p}-1+\e }
		(\varkappa^\sharp_1(\rho^\sharp_1)^2\wedge \varkappa^\sharp_2(\rho^\sharp_2)^2)^{1-\frac{2}{p}-\e}
	=C_{\bar m,m^\sharp}.
\end{align*}
Proceeding with the latter estimate from here on as before in our proof of Theorem \ref{withouttriangle},  but working now with $m^\sharp$ in place of $m,$  we arrive at the statement of Corollary \ref{finish}.
\end{proof}

\color{black}

\setcounter{equation}{0}
\section{Appendix}\label{appendix}

\subsection{A short argument to improve [FU09] to the critical  line}\label{summationtrick}

We consider the set $A_0=\{x\in\mathbb R^2\,:\;1/2<|x|\le1\}$ and define $H={2\bar m}/{(2+\bar m)}.$ Note that $H<h.$ Ferreyra and Urciuolo proved that for  every $p$ for which  $p>4$  and $1/{s'}> {(H+1)}/p,$ there is a constant $C_{p,s}>0$ such that, for every function $f_0$ with supp\,$f_0\subset A,$ we have
$$
\|R^*_{\R^2}f_0\|_p\le C_{p,s}\|f_0\|_s.
$$
Re-scaling this, we obtain
\begin{align}\label{fu-resc}
\|R^*_{\R^2}f_j\|_p\le C_{p,s}2^{\frac jh\big(-\frac1{s'}+\frac{h+1}p\big)}\|f_j\|_s
\end{align}
for every function $f_j$ such that supp$\,f_j\subset \{(x_1,x_2)\,:\; 2^{-\frac{j+1}{m_1}}\le x_1\le 2^{-\frac{j}{m_1}},\;
2^{-\frac{j+1}{m_2}}\le x_2\le 2^{-\frac{j}{m_2}}\}$ and the same range of $p,s.$

\medskip

Given a function $f$ supported in the unit ball of $\R^2,$ we decompose $f=\sum_{j=0}^\infty f_j,$ where the functions $f_j$ have supports as above. Then,
$$
|\{x\,:\;|R^*_{\R^2} f(x)|>\lambda\}|\le |\{x\,:\;|\sum_{j=J}^\infty R^*_{\R^2} f_j(x)|>\lambda/2\}|+|\{x\,:\;|\sum_1^{j=J}R^*_{\R^2} f_j(x)|>\lambda/2\}|,
$$
for some $J$ to be chosen appropriately. Using Chebyshev's inequality, the last expression can be bounded by
$$
\le \bigg(\frac2{\lambda}\bigg)^{p_1}\|\sum_{j=J}^\infty R^*_{\R^2} f_j\|_{L^{p_1}}^{p_1}+
\bigg(\frac2{\lambda}\bigg)^{p_2}\|\sum_{j=1}^J R^*_{\R^2} f_j\|_{L^{p_2}}^{p_2}.
$$
Let us choose exponents $p_1>p>p_2$ such that $1/{s'}={(h+1)}/p$ and
${(h+1)}/{p_2}>1/{s'}>{(h+1)}/{p_1}>{(H+1)}/p.$
We use the triangular inequality and  \eqref{fu-resc} and sum the resulting geometric series, obtaining the inequality
$$
|\{x\,:\;|R^*_{\R^2} f(x)|>\lambda\}|\lesssim \bigg(\frac2{\lambda}\bigg)^{p_1} 2^{\frac{J}h\big(-\frac{p_1}{s'}+{h+1}\big)}\|f\|_{L^s}^{p_1}+\bigg(\frac2{\lambda}\bigg)^{p_2} 2^{\frac{J}h\big(-\frac{p_2}{s'}+{h+1}\big)}\|f\|_{L^s}^{p_2}.
$$
By  choosing $J$ such that $2^{J}=\big({\|f\|_{L^s}}/\lambda\big)^{hs'},$
we then arrive at  the  weak type estimate
$$
|\{x\,:\;|R^*_{\R^2} f(x)|>\lambda\}|\lesssim\bigg(\frac{\|f\|_{L^s}}{\lambda}\bigg)^{(h+1)s'}=
\bigg(\frac{\|f\|_{L^s}}{\lambda}\bigg)^p.
$$
From this, by interpolation with the trivial bound $\|R^*_{\R^2}\|_{L^1\to L^\infty}\le1,$ we obtain the desired strong type estimate.

\setcounter{equation}{0}
\subsection{Fa\`a di Bruno's theorem and completion of the proof of Lemma \ref{reparametrize}}\label{faa}

The formula of Fa\`a di Bruno is a chain rule for higher order derivatives of the composition of two functions. This is well-known for functions in one real variable. However, we need a version for several variables.


\begin{lemnr}[Formula of Faa di Bruno]\label{bruno}
Let $U\subset\R^n$ and  $V\subset\R^m,$ and  let $g=(g^1,\ldots,g^m)\in C^\infty(U,V)$ and $f\in C^\infty(V,\R^l).$ For
$\alpha\in\N^n,$ we put  $A_\alpha=\{\gamma\in\N^n:1\leq|\gamma|\leq|\alpha|\}.$  Then $f\circ g$ is smooth, and for every $\alpha\in\N^n$ we have
\begin{align*}
	\partial^\alpha(f\circ g) = \alpha! \sum\limits_{1\leq|\beta|\leq|\alpha|} (\partial^\beta f)\circ g \sum\limits_k \prod_{j=1}^m\prod_{\gamma\in A_\alpha} \left[\frac{\partial^\gamma g^j}{\gamma!}\right]^{k_\gamma^j},
\end{align*}
where the sum in $k$ is over all mappings $k:\{1,\ldots,m\}\times A_\alpha\to \N, (j,\gamma)\mapsto k_\gamma^j,$ such that
\begin{align}\label{brunocond1}
	\sum_{\gamma\in A_\alpha} k_\gamma^j =\beta_j
\end{align}
for all $j=1,\ldots,m$ and
\begin{align}\label{brunocond2}
	\sum_{j=1}^m\sum_{\gamma\in A_\alpha} k_\gamma^j \gamma = \alpha.
\end{align}
\end{lemnr}

\begin{proof}
The elegant  short proof in \cite{Sp05} for the one-dimensional case can easily be adapted to the higher-dimensional  situation.
\end{proof}

We now come back to the proof of Lemma \ref{reparametrize} and establish the still missing estimates for the derivatives of the function $\phi_2$ (given explicitly by \eqref{phi2}).  Notice that these estimates can not simply  be obtained by means of a scaling argument,  since the first order derivatives are assumed to exhibit a different behavior than the higher order derivatives.

 We shall not really make use of   formula  \eqref{phi2}, but rather proceed as follows:
denoting by  $e_1,\dots,e_{d+1}$ the canonical basis of $\R^{d+1},$
after applying a suitable orthogonal transformation to $\R^{d+1}$ we may and shall assume that $n_1=(0,\dots,0,1)=e_{d+1},$  and that $E_1=e_1, \dots, E_{d-1}=e_{d-1}$ and $e_d=h_1$  (recall here from the first part of the proof of  Lemma \ref{reparametrize} that $E_1,\dots,E_{d-1}$ is an orthonormal basis of $K=H_1\cap H_2$).  Then we may regard $U_1$ as a subset of $\R^d,$ and we  consider the function
$$H(\eta,\tau)=\tau-\phi_1(\eta),\qquad \eta\in U_1,\tau\in\R,$$
whose set of zeros agrees  exactly with $S.$  Observe first that the derivatives of $H$
 satisfy almost the same kind of  estimates as $\phi_1:$
\begin{align}\label{boundsonF}
	\|H'\|_\infty \leq& \sqrt{A^2+1},	\quad	\|H^{(l)}\|_\infty \leq A_l B r^l \quad \mbox{for every } l\geq2.
\end{align}
Let $\psi(\xi)=\xi+\phi_2(\xi)n_2,\  \xi\in U_2,$ be the parametrization of $S$ induced by $\phi_2.$ Moreover, we introduce coordinates on $U_2$ by writing  $\xi= \xi_1E_1+\cdots +\xi_{d-1} E_{d-1} +\xi_d h_2.$ Then obviously \begin{align}\label{25sep1659}
	H(\psi(\xi))=0 \quad \mbox{for all}\quad  \xi\in U_2.
\end{align}
Furthermore,
\begin{align}\label{djpsi}
	\frac{\partial \psi}{\partial \xi_j} =E_j +\frac{\partial \phi_2}{\partial \xi_j} n_2,\quad j=1,\ldots,d-1, \quad \frac{\partial \psi}{\partial \xi_d} =h_2 +\frac{\partial \phi_2}{\partial \xi_d} n_2,
\end{align}
and
\begin{align}\label{25sep1658}
	\partial_\xi^\alpha \psi=\partial_\xi^\alpha\phi_2\,  n_2  \quad \mbox{for all}\quad \alpha\in\N^d, |\alpha|\geq 2.
\end{align}
From \eqref{25sep1659} and \eqref{djpsi} we obtain that for $j=1,\dots,d,$
\begin{align}\label{paphi2}
	\frac{\partial \phi_2}{\partial \xi_j}(\xi) =	- \frac{\langle(\nabla H)(\psi(\xi)), \tilde e_j\rangle} {\langle (\nabla H)(\psi(\xi)),n_2\rangle},
	\end{align}
	if we put $\tilde e_j=E_j=e_j,$ if $j=1,\dots, d-1$ and $\tilde e_d=h_2.$
Notice also that our transversality condition  $|\langle n_2, N(v)\rangle|\ge a>0$ for all $v\in S$ implies that $|\langle (\nabla H)(\psi(\xi)),n_2\rangle|\ge a.$   Thus  \eqref{paphi2} implies  that
\begin{align}\label{25sep1653}
	\big|\frac{\partial \phi_2}{\partial \xi_j}(\xi)\big| \leq \frac{A+1}{a}.			
\end{align}
It remains to show that
\begin{align}\label{boundsong}
	\|\partial^\alpha\phi_2\|_\infty = \|\partial^\alpha \psi\|_\infty \leq& \tilde A_l B r^{|\alpha|}
\end{align}
for every $|\alpha|\geq 2,$  where we have used the abbreviation $\partial=\partial_\xi.$  By induction, we may assume that for every $\gamma\in\N^d$ with $2\leq|\gamma|<|\alpha|$ inequality \eqref{boundsong} holds true\footnote{At the start of the induction with $|\alpha|=2$, the range of such $\gamma$'s is empty.}.
Applying the partial derivative of order $\alpha$ to \eqref{25sep1659} yields
\begin{align*}
	\partial^\alpha(H\circ \psi)=0.
\end{align*}
We apply the formula of Fa\`a di Bruno (Lemma \ref{bruno}).
First, we discuss the summands in Fa\`a di Bruno formula with $|\beta|=1$, say $\beta=e_{j_0}$ for some $j_0=1,\ldots,m$. How many $k$'s are there for which  $\sum_{\gamma\in A_\alpha}k_\gamma^j=\beta_j=\delta_{jj_0}$ and $\sum_{j=1}^m\sum_{\gamma\in A_\alpha}k_\gamma^j\gamma=\alpha$? By the first condition, there exists a $\gamma_0$ such that $k^{j_0}_{\gamma_0}=1$ and $k^{j}_{\gamma}=0$ for $j\neq j_0$ or $\gamma\neq\gamma_0$. But then the second condition implies that
$\gamma_0=\alpha$. Thus we obtain
\begin{align*}
	&\sum\limits_{|\beta|=1} (\partial^\beta H)\circ \psi \sum\limits_k \prod_{j=1}^m\prod_{\gamma\in A_\alpha} \left[\frac{\partial^\gamma \psi^j}{\gamma!}\right]^{k_\gamma^j}
	= \sum_{j_0=1}^m (\partial_{j_0} H)\circ \psi \left[\frac{\partial^\alpha \psi^{j_0}}{\alpha!}\right]^{k_\alpha^{j_0}}\\
	&= \frac{1}{\alpha!} \langle (\nabla H)\circ \psi, \partial^\alpha \psi\rangle
	= \frac{\partial^\alpha \phi_2 }{\alpha!} \langle (\nabla H)\circ \psi, n_2\rangle,
\end{align*}
where we have used \eqref{25sep1658} once more. This implies that
\begin{align*}
	|\partial^\alpha \phi_2|\leq& \frac{\alpha!}{a}
	\left|\sum\limits_{|\beta|=2}^{|\alpha|} (\partial^\beta H)\circ \psi \sum\limits_k \prod_{j=1}^m\prod_{\gamma\in A_\alpha} \left[\frac{\partial^\gamma \psi^j}{\gamma!}\right]^{k_\gamma^j}\right|,
\end{align*}
where the sum in $k$ is over all mappings $k:\{1,\ldots,m\}\times A_\alpha\to \N, (j,\gamma)\mapsto k_\gamma^j$ such that
$	\sum_{\gamma\in A_\alpha} k_\gamma^j =\beta_j$
for all $j=1,\ldots,m$ and $\sum_{j=1}^m\sum_{\gamma\in A_\alpha} k_\gamma^j \gamma = \alpha$.
Observe that for all $k$ appearing in the above sum, we have $k^j_\alpha=0$ for all $j=1,\ldots,m$:
\smallskip

For, otherwise there would be some $j_0$ such that $k_\alpha^{j_0}=1$ and $k_\gamma^j=0$ if $\gamma\neq\alpha$ or $j\neq j_0$, a contradiction to $2\leq|\beta|=\sum_{j,\gamma} k_\gamma^j$.\\
Thus, if $k_\gamma^j\neq0$ for an exponent in the above sum, then  we have $|\gamma|<|\alpha|,$ and therefore  our induction hypothesis applies the following:

If $|\gamma|\geq2$, then we  may estimate $|\partial^\gamma \psi^j|\leq A_{|\gamma|} B r^{|\gamma|}.$ And, if $|\gamma|=1$, then  in view of \eqref{djpsi} and \eqref{25sep1653}, we may estimate $|\partial^\gamma \psi^j|\leq 1+(A+1)/a \lesssim 1.$    Making also use of \eqref{boundsonF}, we  then arrive at the estimation
\begin{align*}
	|\partial^\alpha \phi_2|\lesssim&
	\sum\limits_{|\beta|=2}^{|\alpha|} B r^{|\beta|} \sum\limits_k
	\prod_{j=1}^m\prod_{|\gamma|=2}^{|\alpha|} \left[B r^{|\gamma|}\right]^{k_\gamma^j}\\
	\leq& \sum\limits_{|\beta|=2}^{|\alpha|} \sum\limits_k
	 B^{1+\sum_j\sum_{|\gamma|=2}^{|\alpha|}k_\gamma^j} \  r^{|\beta|+\sum_j\sum_{|\gamma|=2}^{|\alpha|}k_\gamma^j|\gamma|}.
\end{align*}
Notice that we have
\begin{align*}
	|\beta|=\sum_j\beta_j =& \sum_j\sum_{|\gamma|=1}^{|\alpha|} k_\gamma^j
	= \sum_j\sum_{|\gamma|=2}^{|\alpha|} k_\gamma^j + \sum_j\sum_{|\gamma|=1} k_\gamma^j |\gamma|,
\end{align*}
and thus
\begin{align*}
	B^{1+\sum_j\sum_{|\gamma|=2}^{|\alpha|}k_\gamma^j}  r^{|\beta|+\sum_j\sum_{|\gamma|=2}^{|\alpha|}k_\gamma^j|\gamma|}
	= B r^{\sum_j\sum_{|\gamma|=1}^{|\alpha|} k_\gamma^j |\gamma|}
					(Br)^{\sum_j\sum_{|\gamma|=2}^{|\alpha|}k_\gamma^j}
	\leq B r^{|\alpha|},
\end{align*}
where we  have made use of our assumption $Br\leq1$. This proves also \eqref{boundsong}.

~
\thispagestyle{empty}
\newpage

\renewcommand{\refname}{References}


\begin{thebibliography}{----------}


\bibitem [Ba85]{Ba1} B. Barcelo Taberner, On the restriction of the Fourier transform to a conical surface, Transactions of the American Mathematical Society, 292 No1 (1985).
\bibitem [Ba86] {Ba2}B. Barcelo, The restriction of the Fourier transform to some curves and surfaces, Studia Mathematica 84 No1 (1986), 39–69.
\bibitem [BCT06] {BCT}J. Bennet, A. Carbery, T. Tao, On the multilinear restriction and Kakeya conjectures. Acta Math.  196  (2006),  no. 2, 261–302.
\bibitem[Bou85]{Bo85} Bourgain, J.,
  \newblock Estimations de certaines fonctions maximales.
 \newblock {\em C. R. Acad. Sci. Paris S\'er. I Math.}, 301 (1985) no. 10, 499--502.
\bibitem [Bo91] {Bo1}J. Bourgain, Besicovitch-type maximal operators and applications to Fourier analysis, Geometric and Functional Analysis 22 (1991), 147-187.
\bibitem [Bo95a] {Bo2} J. Bourgain, Some new estimates on oscillatory integrals. Essays in Fourier Analysis in honor of E. M. Stein. Princeton Math. Ser. 42, Princeton University Press, Princeton, NJ 1995, 83–112.
\bibitem [Bo95b]  {Bo3}  J. Bourgain, Estimates for cone multipliers, Operator Theory: Advances and Applications 77 (1995), 1-16.
\bibitem [BoG11] {BoG} J. Bourgain, L. Guth, Bounds on oscillatory integral operators based on multilinear estimates, Geometric and Functional Analysis, Vol.21 (2011) 1239-1295.
\bibitem [Bu15] {Bu} S. Buschenhenke, A sharp $L^p$-$L^q$ Fourier restriction theorem for a conical surface of finite type, Mathematische Zeitschrift, Vol. 280, Issue 1 (2015), 367-399.
\bibitem [F70] {F1} C. Fefferman, Inequalities for strongly singular convolution operators. Acta Math., (1970), 9–36.
\bibitem [F71] {F}C. Fefferman, The multiplier problem for the ball, Annals of Mathematics 94 (1971) 330-336.
\bibitem [FU09] {FU}E. Ferreyra, M. Urciuolo, Fourier restriction estimates to mixed homogeneous surfaces, Journal of inequalities in pure and applied mathematics, Vol. 10 (2009) Issue 2, Article 35, pp. 11.
\bibitem [Gra08] {Gra} L. Grafakos, Classical Fourier Analysis, Graduate Texts in Mathematics, Springer, 2008.
\bibitem [Gr81] {Gr}A. Greenleaf, Principal Curvature and Harmonic Analysis, Indiana University Mathematics Journal, Vol. 30, No. 4 (1981).

\bibitem [IKM10] {ikm} I. A. Ikromov, M. Kempe, D. Müller, Estimates for maximal functions associated with hypersurfaces in $\R^3$ and related problems in harmonic analysis, Acta Mathematica 204 (2010), 151-271.

\bibitem[IM11]{IM-uniform}
 I. A. Ikromov, D. M\"uller,
\newblock Uniform estimates for the Fourier transform of surface carried measures in  $\R^3$ and an application to Fourier restriction.
\newblock {\em    J. Fourier Anal. Appl.,} 17 (2011), no. 6, 1292--1332.


\bibitem [IM15] {IM} Ikromov, M\"uller, Fourier Restriction for Hypersurfaces in Three Dimensions and Newton Polyhedra, Part I and II; \, http://arxiv.org/abs/1208.6090 and  http://arxiv.org/abs/1410.3298; to appear Annals of Mathematics Studies.
\bibitem [L03]{lee03} S. Lee, Endpoint estimates for the circular maximal function, Proc. Amer. Math. Soc. 131 (5) (2003) 1433–1442.
\bibitem [L05] {lee05} S. Lee, Bilinear restriction estimates for surfaces with curvatures of different signs, Transactions of the American Mathematical Society, Vol. 358, No. 8, 3511-2533, 2005.
\bibitem [LV10] {LV} S. Lee, A. Vargas, Restriction estimates for some surfaces with vanishing curvatures, Journal of Functional Analysis 258 (2010), 2884-2909.
\bibitem [MVV96] {MVV1} Moyua, A., Vargas, A. and L. Vega, L., Schr\"odinger maximal function and restriction
properties of the Fourier transform. Internat. Math. Res. Notices 16 (1996), 793–815.
\bibitem [MVV99] {MVV2} Moyua, A., Vargas, A. and L. Vega, L., Restriction theorems and maximal operators related to oscillatory integrals in $\R^3,$ Duke Math. J., 96 (3), (1999), 547-574.
\bibitem [S87] {So87} C. D. Sogge, A sharp restriction theorem for degenerate curves in $\R^2$, American Journal of Mathematics, Vol. 109, No. 2 (1987), pp. 223-228.

\bibitem [Sp05] {Sp05} K. Spindler, A short proof of the formula of Fa\a di Bruno, Elemente der Mathematik, Swiss Mathematical Society (2005), 33-35.

\bibitem [St86] {St1} E. Stein, Oscillatory Integrals in Fourier Analysis. Beijing Lectures in Harmonic Analysis. Princeton Univ. Press 1986.
\bibitem [Sto15] {Sto} B. Stovall, Linear and bilinear restriction to certain rotationally symmetric hypersurfaces. Preprint 2015.
\bibitem [Str77] {Str} Strichartz, R. S.,Restrictions of Fourier transforms to quadratic surfaces and decay of solutions of wave equations. Duke Math. J.  44  (1977), no. 3, 705–714.
\bibitem [T01a] {T4}T. Tao Endpoint bilinear restriction theorems for the cone, and some sharp null-form
estimates, Math. Z. 238 (2001), 215-268.solutions of wave equations. Duke Math. J., 44 (1977), 705–714.
\bibitem [T01b] {T1}T. Tao, From Rotating Needles to stability of waves: emerging connections between combinatorics, analysis, and PDE, Notices of the American Mathematical Society, 48 (2001).
\bibitem [T03a] {T2} T. Tao, A Sharp bilinear restriction estimate for paraboloids, Geometric and Functional Analysis 13, 1359-1384, 2003.
\bibitem [T03b] {T3}T. Tao, Recent progress on the Restriction conjecture, arXiv:math/0311181v1 [math.CA] 12 Nov 2003.
\bibitem [To75] {To} P. A. Tomas, A restriction theorem for the Fourier transform. Bull. Amer. Math. Soc. 81 (1975), 477–478.
\bibitem [TVI00] {TV1} T. Tao, A. Vargas, A bilinear approach to cone multipliers I. Restriction estimates, Geometric and Functional Analysis 10, 185-215, 2000.
\bibitem [TVII00] {TV2} T. Tao, A. Vargas, A bilinear approach to cone multipliers II. Applications, Geometric and Functional Analysis 10, 216-258, 2000.
\bibitem [TVV98] {TVV} Tao, T., Vargas, A. and Vega, L., A bilinear approach to the restriction and Kakeya
conjectures. J. Amer. Math. Soc. 11 (1998) no. 4 , 967–1000.
\bibitem [V05]{v05} A. Vargas, Restriction theorems for a surface with negative curvature, Mathematische Zeitschrift 249, 97-111 (2005).
\bibitem [W95] {W1} T. Wolff, An improved bound for Kakeya type maximal functions, Revista Mat.
Iberoamericana 11 (1995), 651-674.
\bibitem [W01] {W2} WT. Wolff, A Sharp Bilinear Cone Restriction Estimate, The Annals of Mathematics, Second Series, Vol. 153, No. 3, 661-698, 2001.
\bibitem [Z74] {Z} A. Zygmund, On Fourier coefficients and transforms of functions of two variables, Studia Mathematica 50 (1974), 189-201.




\thispagestyle{plain}
\newpage





\end{thebibliography}
\end{document}